\documentclass[a4paper]{amsart}
\usepackage{csquotes}

\usepackage[style=alphabetic]{biblatex}
\usepackage[T1]{fontenc}
\usepackage{geometry}
\usepackage{mathrsfs}

\usepackage{stmaryrd}
\title{Kontsevich-Zagier conjecture for Hensel-minimal fields with sections}

\author{Xavier Pigé}

\usepackage{hyperref}
\usepackage{my-macros}
\usepackage{tikz-cd}

\theoremstyle{plain}
\newtheorem{fact}{Fact}[subsection]
\newtheorem{thm}[fact]{Theorem}
\newtheorem*{thm*}{Theorem}
\newtheorem{lem}[fact]{Lemma}
\newtheorem{prop}[fact]{Proposition}
\newtheorem{cor}[fact]{Corollary}
\newtheorem*{cor*}{Corollary}

\newtheorem{thmintro}{Theorem}

\theoremstyle{definition}
\newtheorem{rmk}[fact]{Remark} 
\newtheorem{dfn}[fact]{Definition}

\newtheorem{question}[fact]{Question}
\newtheorem*{question*}{Question}
\newtheorem{example}[fact]{Example}

\numberwithin{equation}{section}

\newcommand\RV{\mathrm{RV}}
\newcommand\VF{\mathrm{VF}}
\newcommand\VFR{\mathrm{VFR}}
\DeclareMathOperator\rv{rv}
\DeclareMathOperator\sn{sn}

\DeclareMathOperator\ac{\overline{ac}}
\newcommand\fL{\mathfrak L}
\DeclareMathOperator\Kp{K_+}
\DeclareMathOperator\Ob{Ob}
\newcommand\Isp{I_{\mathrm{sp}}}
\newcommand{\ACVF}{\mathrm{ACVF}}

\renewcommand{\sp}{\mathrm{sp}}
\newcommand\Kb{\overline K}
\DeclareMathOperator\Jac{Jac}
\newcommand\Inj{\mathrm{Inj}}
\newcommand\Def{\mathrm{Def}}
\newcommand\I{\mathrm I}
\newcommand\bbL{\mathbb L}
\newcommand\DP{\mathrm{DP}}
\newcommand\vol{\mathrm{vol}}
\newcommand\tgamma{\widetilde \gamma}
\newcommand\tGamma{\widetilde \Gamma}
\newcommand\tun{\widetilde 1}

\newcommand\bdd{\mathrm{bdd}}
\DeclareMathOperator\Supp{Supp}
\newcommand\sC{\mathscr{C}}

\newcommand\setItemnumber[1]{\setcounter{enumi}{\numexpr#1-1\relax}}
\begin{document}

\maketitle

\begin{center}
\begin{minipage}{15cm}
We generalise Yin's motivic integration with section to the Hensel-minimal framework, with and without volume forms, building on Stout and Vermeulen's work on a Hrushovski-Kazhdan style motivic integral for 1-h-minimal valued fields. In particular, we get a bijective motivic integral \emph{à la} Hrushovski-Kazhdan between certain Grothendieck groups of definable sets on the valued field and the coarse Grothendieck groups on $\RV$ in the Denef-Pas language, with angular component. In the second half of the article, we build a Cluckers-Loeser-like framework and show that our constructions indeed give back Cluckers-Loeser motivic integral. Thus Cluckers-Loeser motivic integration derives from an isomorphism of Hrushovski-Kazhdan motivic integration, but in the Denef-Pas language with a section of $\RV$.
\end{minipage}
\end{center}

\tableofcontents

\section{Introduction}

In 2006, E. Hrushovski and D. Kazhdan introduced the motivic integration theory which now bears their name \cite{ginzburg_integration_2006}. The integral defined in this way was an isomorphism between some Grothendieck \hbox{(semi-)}groups of definable subsets on the valued field $\VF$ and others on the quotient $\RV^\times = \VF^\times / (1+\mathfrak m)$; this is an amalgam of the residue field $\Kb$ and value group $\Gamma$ which can be thought of as a leading term. The concept of motivic integration was introduced by Kontsevich in 1995 over the field $\bC((t))$, to prove that birational Calabi-Yau varieties have the same Hodge numbers. The idea is to get a sort of integration theory for algebraic varieties over non-local valued fields, mimicking $p$-adic integration for local fields but with values in a Grothendieck ring of varieties over the residue field (instead of the real numbers).

Hrushovski and Kazhdan's theory suffered from some weaknesses: unlike Cluckers-Loeser motivic integration, which works with any characteristic $0$ residue field and with discrete value groups, it was restricted to characteristic $0$ algebraically closed residue fields and divisible value groups. Besides, whereas Cluckers-Loeser motivic integration employs the language of Denef-Pas, classical in model theory of valued fields and which includes an angular component, Hrushovski and Kazhdan chose to work with $\RV$ rather than with an angular component.

The second difficulty was solved by Yin \cite{yin_integration_2013} by adding a section, although this is not made explicit in that article. For the first one, the introduction of Hensel minimality by Cluckers, Rideau and Halupczok \cite{cluckers_hensel_2022} was a first step, which allowed Stout and Vermeulen, in a recent work \cite{stout_integration_2025}, to broadly generalise the principle of Hrushovski-Kazhdan motivic integration. However, this generalisation still doesn't make it possible to add an angular component. More precisely, while they do prove the existence of a motivic integration morphism \emph{à la} Hrushovski-Kazhdan even in the Denef-Pas language, this is not an isomorphism, yet bijectivity is one of the core features of Hrushovski-Kazhdan motivic integration: such bijectivity means that this integration theory is, in some sense, universal.

Let us give some more explanations on that point. Many applications of motivic integration -- including Kontsevich's initial argument -- follow this pattern:
\begin{enumerate}
\item one wishes to prove an equality between two classes $[R]$ and $[S]$ in a Grothendieck group over $\RV$ -- or the residue field, for instance;
\item one finds two classes $[X]$ and $[Y]$ in the corresponding Grothendieck group over $\VF$ such that with $\int [X] = [R]$ and $\int [Y]=[S]$;
\item using Fubini or change of variables formula, one finally proves $[X]=[Y]$.
\end{enumerate}
In other words, the category in which we must work over $\RV$ is fixed, and we want to find the "good category" to work in over $\VF$ in order to get an isomorphism of Grothendieck (semi-)groups between them. Note that we have much freedom, because many different structures on $\VF$ may all induce the same structure on $\RV$. The converse problem, on the other hand, would be much less interesting, for the structure on $\VF$ determines that on $\RV$. In this regard, note however that, in the case we have a 1-h-minimal theory, and we take over $\VF$ the categories defined by Stout and Vermeulen \cite{stout_integration_2025}, then we have a motivic integration isomorphism if we take on $\RV$ the categories whose morphisms are the bijections which lift to $\VF$ (quotienting by the usual congruence).

\subsection{Hrushovski-Kazhdan motivic integration in 1-h-minimal fields with section}
Now let us take a closer look at the following four examples:
\begin{enumerate}
\item $\RV[*]=\bigoplus_{n\in\bN}\RV[n]$ the classical categories defined by Hrushovski-Kazhdan and Stout-Vermeulen, for which the objects of $\RV[n]$ are $\cL_{val}$-definable subsets $R\subset(\RV^\times)^n\times\RV^m$ without parameters, for some $m\in\bN$, with the condition that the projection $R\to(\RV^\times)^n$ is finite-to-one;
\item $\RV[*,\cdot]$ the coarse categories defined by Hrushovski-Kazhdan, analogous to the above but without finiteness condition;
\item $\RV_{\DP}[*]$ analogous to the classical categories, but in the Denef-Pas language (and without finiteness condition on $\Gamma$);
\item $\RV_{\DP}[*,\cdot]$ analogous to the coarse categories, but in the language of Denef-Pas.
\end{enumerate}
In the first case, if the residue field is algebraically bounded over $\emptyset$ (\emph{cf.} \cite{dries_dimension_1989}; this is the case of most standard theories of fields, such as algebraically closed, real closed, $p$-adically closed, pseudofinite or more generally PAC fields) and $\cL=\cL_{val}$, then Stout and Vermeulen \cite{stout_integration_2025} prove, without any condition on the value group, that the "good category" to consider over $\VF$ is the classical category $\VF$ (colimit of the $\VF[n], n\in\bN$), again in the language $\cL=\cL_{val}$. 

In this paper, we show that, in the three other cases, the categories over $\VF$ we have to consider in order to get a motivic integration isomorphism have a similar description, but in an expansion of the language: in the third case, one must add a section $\sn:\Gamma\to\VF^\times$, and in the second and fourth cases a section $\sn:\RV^\times\to\VF^\times$. Actually, in the second and fourth cases,  we do not even need any assumption on the residue field (apart from characteristic $0$).

Let us focus on the last example, thus denoting 
by $\cL'_{\DP}$ the language $\cL_{\DP}\cup\{\sn\}$
. We introduce the following $\VF$-category:
\begin{enumerate}
\setItemnumber{4}
\item $\VF_{\DP}[\cdot]$ is the category whose objects are $\cL'_{\DP}(\emptyset)$-definable subsets $X\subset \VF^n\times\RV^m$, for some $n,m\in\bN$.
\end{enumerate}
Morphisms are just definable bijections (without parameters, in the language $\cL'_{\DP}$). Note that the unrestricted presence of $\RV$-coordinates echoes the definition of coarse $\RV$-categories, so that there is a natural lifting map
\[\fL:\Ob\RV_{\DP}[*,\cdot]\to\Ob\VF_{\DP}[\cdot]\,.\]


Our Theorem \ref{thm_int_sans} for motivic integration without volume forms states the following:
\begin{thmintro}\label{thm_intro_HK}
There is a semiring isomorphism between the Grothendieck semigroups

\[\int:\Kp\VF_{\DP}[\cdot]\to\Kp\RV_{\DP}[*,\cdot]/\Isp\,,\]
where $\Isp$ an congruence generated by one explicit relation.

This isomorphism is characterised by the equality
\[\int[\fL R] = [R]\]
for every object $R$ of $\RV_{\DP}[*,\cdot]$.
\end{thmintro}
We also prove the analogous result for motivic integration with volume forms, see Theorem \ref{thm_int_avec}. To do so, we construct a dimension theory for $\cL'_{\DP}$-definable subsets, see Theorem \ref{thm_dim}. This dimension theory satisfies almost all properties from the dimension in a 1-h-minimal theory stated in \cite{cluckers_hensel_2022}, Proposition 5.3.4.

Note that Theorem \ref{thm_intro_HK} is just a particular case of our results (but a relevant one); throughout this paper, we will adopt a more general approach than solely adding sections -- mainly to simplify notations-- then specialising to these different cases in examples. However, we were guided in our approach by those few cases, which the reader may keep in mind. More precisely, throughout this article, we will work in a expansion of a 1-h-minimal theory satisfying some axiomatic assumptions. An  interesting question to study would be to see how far one can go in generalising Hensel minimality to a framework including sections: is it possible to drop the assumption that there is an underlying Hensel-minimal theory and to find directly some axiomatic assumptions, satisfied by at least some of the above examples, which generalise Hensel minimality?

\subsection{A variant of Kontsevich-Zagier conjecture for Cluckers-Loeser motivic integration}
Note that the last case presents a strong resemblance to the categories in which Cluckers and Loeser work. Indeed, let $\mu_\Gamma\RV_{\DP}^\bdd[*,\cdot]$ be the subcategory of bounded objects (see Definition \ref{dfn_borné}) and $I$ be an ideal generated by two explicit relations, one from Theorem \ref{thm_intro_HK} and one to normalise the measure, see Definition \ref{def_ideal_volume}. Then the Grothendieck group $\mathrm K\mu_\Gamma\RV_{\DP}^\bdd[*,\cdot]/I$ is isomorphic (as a ring) to $C^\vol(\{*\})$, the ring of Cluckers-Loeser constructible functions (or Functions) on the point. Also, Cluckers and Halupczok found out that Hrushovski-Kazhdan motivic integrals could be seen as a kind of generalisation of Kontsevich-Zagier conjecture. In particular, in \cite{cluckers_p-adic_2021}, they treat the $p$-adic case, showing that $p$-adic integration satisfies a non-archimedean version of Kontsevich-Zagier. So from our Hrushovski-Kazhdan motivic integration isomorphisms, we will construct a motivic integration formalism \emph{à la} Cluckers-Loeser (Proposition \ref{prop_formalisme_CL}) -- which we may therefore interpret as the universal integration theory in the theory of henselian valued fields with angular component and a section of $\RV$, as it derives from an isomorphism -- and we will show that this framework indeed corresponds to Cluckers-Loeser motivic integration (Theorem \ref{thm_cluckers_loeser}). Here "universal integration theory" means that any relation between the integrals comes from geometric relations between the functions we integrate, as in Kontsevich-Zagier conjecture for periods. More precisely, we prove the following theorem:

\begin{thmintro}\label{thm_intro_CL}
For $Z$ an $\cL_{\DP}(\emptyset)$-definable set, let $C^{\cL,\vol}(Z)$ be the group of constructible motivic Functions over $Z$ and let $\I_* C^{\cL,\vol}(Z)$ be the subgroup those Functions that are integrable, as defined in \cite{cluckers_constructible_2008}.

For $Z$ an $\cL'_{\DP}(\emptyset)$-definable set, let $C^\vol(Z)$ and  $\I_* C^\vol(Z)$ be their $\cL'_{\DP}$-definable counterparts (Definition \ref{def_Fonctions_CL}).

There is a group morphism $C^{\cL,\vol}(Z)\to C^\vol(Z)$ such that the following diagram
\[\begin{tikzcd}
\I_* C^{\cL,\vol}(Z) \arrow[r] \arrow[d, "\int"] & \I_* C^\vol(Z) \arrow[d, "\int"]\\
C^{\cL, \vol}(*) \arrow[r, equal] & C^\vol(*)
\end{tikzcd}\]
is commutative.
\end{thmintro}
The actual theorem that we prove, Theorem \ref{thm_cluckers_loeser}, is much more general than that; in particular, it takes into account relative motivic integration.

Two aspects of this deserve emphasis:
\begin{itemize}
\item first, it finally allows to properly connect Hrushovski-Kazhdan and Cluckers-Loeser motivic integrals, by matching Cluckers-Loeser motivic integration to an \emph{isomorphism} of motivic integration \emph{à la} Hrushovski-Kazhdan. In that way, not only those two integrals no longer live in two disjoint universes (which was already the case, actually, since the work by Cluckers, Halupczok and Rideau \cite{cluckers_hensel_2022} on the one hand, and by Stout and Vermeulen \cite{stout_integration_2025} on the other hand), but Cluckers-Loeser integral corresponds to a case in which we have the full power of Hrushovski-Kazhdan's results, \emph{i.e.} an isomorphism;
\item however, the required structure to make this isomorphism appear is totally unexpected: it seems very surprising at first sight that it is necessary to add a section of $\RV$. That being said, the reason for this is that the angular component induces "too many" definable bijections on $\RV$, and we need to lift them, which is made possible by the section. In that case, what is remarkable is not that we can lift bijections, but that we still have an integration theory that behaves, in almost all respects, like that of \cite{stout_integration_2025}.
\end{itemize}

The question of rephrasing Hrushovski-Kazhdan motivic integration to connect it to a universal version of Cluckers-Loeser will also be addressed by Stout and Vermeulen in an upcoming article, with a different approach.

\subsection{Structure of the article}
This paper is organised as follows: in Section \ref{sec_axiomes}, we follow a general, axiomatic approach for adding sections, showing that there is a well-behaved dimension theory in that context. We need three axioms: the first is the existence of what we call $\RV$-partitions (following Yin \cite{yin_integration_2013}, see Definition \ref{def_RV_partition}), which is the most used notion in the article. Basically, it allows us to reduce many problems to the 1-h-minimal case. The second axiom, which we call "thin theories" (see Definition \ref{def_theorie_fine}), is a very weak property, which is essentially there to avoid pathologies and to yield a well-behaved dimension theory. The last one, a bit more technical to state, is the fact that the theory is "refined" (see Definition \ref{def_theorie_raffinee}). We only use it once, to show that the integration morphism in dimension one is well-defined, but actually this seems to be a very important notion, which would probably be worth studying on its own. Still in Section \ref{sec_axiomes}, we prove Theorems \ref{thm_int_sans} and \ref{thm_int_avec}, our main results on Hrushovski-Kazhdan motivic integration, closely following \cite{stout_integration_2025}. Then in Section \ref{sec_henselien} we study the example of theories in the pure valued fields language (or some $\RV$-expansions of it, like the Denef-Pas language), with a section of either the residue field, the value group or $\RV$, and at that point we obtain Theorem \ref{thm_intro_HK}. The proof that these theories satisfy the three axioms relies upon a relative quantifier elimination result: when adding such a section, the resulting theory eliminates quantifiers relative to the pure valued fields language (or the considered $\RV$-expansion). On the syntactical side, this means that any formula is equivalent to a boolean combination of quantifier-free formulas in the full language (included the section), and of formulas in the restricted language (without the section); on the semantic side, this just means that any partial embedding for the full language that is elementary for the restricted language is actually elementary for the full language.

The second part of the article, Sections \ref{sec_formalisme_CL} and \ref{sec_integration_CL}, is dedicated to the connection with Cluckers-Loeser motivic integration, and is mostly independent from the first one, apart the main definitions and results. In Section \ref{sec_formalisme_CL}, we explain in detail how to go from a Hrushovski-Kazhdan motivic integral to a Cluckers-Loeser-like formalism. More precisely, we first define relative variants of the Hrushovski-Kazhdan motivic integral, then we use them to define direct image morphisms $f_!:C_+(Z')\to C_+(Z)$ for every definable map (without parameters) $f:Z'\to Z$. Here $C_+(Z)$ is a semi-group defined using the (relative) $\RV$-categories on $Z$. This is done without any further hypothesis on the theory, and both with and without measure -- but the price to pay is that the semigroups of Functions are defined in terms of $\RV$ instead of the residue field. We remedy this in Section \ref{sec_integration_CL}, under the assumption that the value group is elementarily equivalent to $\Z$, and connect our groups of Functions to those of Cluckers-Loeser, allowing us to prove Theorem \ref{thm_intro_CL}.

\subsection*{Acknowledgements} Xavier Pigé acknowledges the support of the CDP C2EMPI, together with the French State under the France-2030 programme, the University of Lille, the Initiative of Excellence of the University of Lille, the European Metropolis of Lille for their funding and support of the R-CDP-24-004-C2EMPI project. The author also thanks Silvain Rideau-Kikuchi as well as Mathias Stout and Floris Vermeulen for their interest and stimulating discussions. Finally, he is very grateful to Raf Cluckers and Arthur Forey for suggesting this topic and for their constant support.

\section{Axiomatic approach for adding sections}\label{sec_axiomes}

Let $\cL$ be a valued field language with two sorts $\VF$ and $\RV$, and $\cL'$ be an expansion of $\cL$. Let $\cT$ be a 1-h-minimal $\cL$-theory as defined by Cluckers, Halupczok and Rideau \cite{cluckers_hensel_2022}, and $\cT'$ be an expansion of $\cT$ to the language $\cL'$. Throughout this section, we will introduce the right framework for Hrushovski-Kazhdan motivic integration with section, get some tools and prove our main general theorems on motivic integration, Theorems \ref{thm_int_sans} and \ref{thm_int_avec}.

\begin{rmk}
In the examples, the language $\cL'$ will be an expansion of $\cL$ by a function symbol $\sn$, and the theory $\cT'$ will be obtained from $\cT$ by saying that $\sn$ is a section of $\rv$ (or of the valuation $|\cdot |$), with compatibilities to addition and multiplication.
\end{rmk}

\subsection{Definitions and axiomatic framework}

Now we will introduce an axiomatic framework in which we will be able to define a motivic integration isomorphism. In particular, we will define $\RV$-partitions in Definition \ref{def_RV_partition} and $\RV$-thin theories in Definition \ref{def_theorie_fine}. The last important axiom will only appear later, in Definition \ref{def_theorie_raffinee}. Let us start by the following definition, which will be essential to our work:

\begin{dfn}\label{def_RV_partition}
Let $M\models\cT', A\subset M$ be a set of parameters such that $A=\dcl_{\cL'}(A)$ and $X\subset\VF^n\times\RV^m$ be a  $\cL'(A)$-definable subset. An \emph{$A$-definable $\RV$-partition} of $X$ (relative to $\cT$) is an $\cL'(A)$-definable map $\chi:X\to Y, Y\subset\VF^d, d\in\bN,$ such that:
\begin{enumerate}
\item for all $y\in Y$, the fibre $X_y=\chi^{-1}(y)$ is $\cL(Ay)$-definable;
\item for some $l\in\bN$, there is an injective $\cL'(A)$-definable map $Y\to \RV^l$.
\end{enumerate}

If such $\RV$-partitions exist for every set $A$ of parameters and every $A$-definable set, we say that $\cT'$ \emph{has $\RV$-partitions}.
\end{dfn}

Certainly adding parameters both in $\VF$ or $\RV$ preserve the existence of $\RV$-partitions. This notion is inspired by the notion introduced by Yin under the same name (\cite{yin_integration_2013}, Definition 4.1). Saying that $\cT'$ has $\RV$-partitions approximatively means that every definable set is an $\RV$-union of subsets which are definable is a 1-h-minimal theory.

As one can see, sets that map into $\RV$ play an important role. This gives motivation for the following definition
\begin{dfn}
Let $M\models\cT'$ and $A\subset M$ be some parameters. An $\cL'(A)$-definable set maps into $\RV$ (for the parameters $A$) if there exist an injective $\cL'(A)$-definable map $i:Y\to\RV^l$ for some $l\in\bN$.

We denote by $\Inj_A$ the Ind-definable (over $\cL'(A)$) set given by the union of all $Y$ that map into $\RV$.
\end{dfn}
Any $\cL'(A)$-definable subset $Y\subset \Inj_A$ (in a saturated model) then maps into $\RV$. Indeed, it is enough to check that, if $Y_1,Y_2$ map into $\RV$, then so does $Y_1\cup Y_2$, which is clear. Also remark that, for every model $M$ of $\cT'$, it holds that $\Inj_A(M)\subset\dcl_{\cL'}(A\cup\RV(M))$.

To be more accurate, we should write $\Inj_A^n$ for the union of all $Y\subset\VF^n$ which map into $\RV$; $\Inj_A$ would then be a sequence of subset of $(\VF^n)_{n\in\bN}$. We will keep with this notational abuse in what follows.

\begin{lem}\label{lem_dcl_acl}
Assume that $\cT'$ has $\RV$-partitions relative to $\cT$. Let $M\models \cT'$, and set $a$ to be some tuple in $M$. Let $S_a = \dcl_{\cL'}(a)\cap \Inj_\emptyset$. Also denote $\widetilde S_a = \acl_{\cL'}(a)\cap\Inj_\emptyset$, so that $S_a\subset \widetilde S_a$. Then 
\[\dcl_{\cL'}(a)=\dcl_\cL(a S_a) \text{ and } \acl_{\cL'}(a)=\acl_\cL(a \widetilde S_a)\,.\]
\end{lem}
\begin{proof}
We first deal with $\dcl$. Let $f:X\to Y$ be an $\cL'(\emptyset)$-definable map, $a\in X$, and let $\Gamma_f$ be its graph, $\pi_1:\Gamma_f\to X$ the first projection. Let $\chi:\Gamma_f\to Z$ be an $\cL'(\emptyset)$-definable $\RV$-partition, as well as $g:Z\to\RV^l$ an $\cL'(\emptyset)$-definable injection. Let $c=\chi(a,f(a))$. Then the fibre $\chi^{-1}(c)$ is $\cL(c)$-definable, \emph{i.e.} the restriction of $f$ to $\pi_1(\chi^{-1}(c))$ is $\cL(c)$-definable; on the other hand, $c\in S_a$. In particular, $f(a)\in\dcl_\cL(ac)\subset\dcl_\cL(aS_a)$. The converse inclusion is straightforward.

For $\acl$, we proceed in a similar fashion. Let $\Xi\subset X\times Y$ be an $\cL'(\emptyset)$-definable relation with $\pi_1:\Xi\to X$ having finite fibres, and let $\chi:\Xi\to Z$ be an $\cL'(\emptyset)$-definable $\RV$-partition of $\Xi$ as well as $g:Z\to\RV^l$ an $\cL'(\emptyset)$-definable injection. Let $b\in Y$ be such that $(a,b)\in \Xi$, and set $c=\chi(a,b)$. Then the fibre $\chi^{-1}(c)$ is $\cL(c)$-definable, and $c\in \widetilde S_a$. In particular, $\pi_1:\chi^{-1}(c)\to X$ has finite fibres, whence $b\in\acl_\cL(ac)\subset\acl_\cL(a\widetilde S_a)$. The reciprocal inclusion is, again, straightforward.
\end{proof}

So the set $\Inj_\emptyset$ (or $\Inj_A$, if one uses parameters inside $A$) plays an important role in the theory. Here is an example of a pathological situation:
\begin{example}
Take $\cT=\ACVF(0,0)$ and let $K$ be the algebraic closure of $\bQ(t)$ together with the $t$-adic valuation, which is a model of $\cT$. Both the sorts $\VF(K)$ and $\RV(K)$ are countable; let $f:\RV(K)\to\VF(K)$ be a (set-theoretic) bijection. We then define $\cL'=\cL\cup\{f\}$, and $\cT'=\mathrm{Th}_{\cL'}(K)$. For any definable set $X$, the function $\id_X$ trivially gives an $\RV$-partition.

Thus the existence of $\RV$-partitions is not at all a guarantee for good geometric behavior. Nonetheless, in the example, we note that $\Inj_\emptyset=\VF$, and that this seems quite directly connected to our problems.
\end{example}
This gives motivation for the following definition:
\begin{dfn}\label{def_theorie_fine}
The theory $\cT'$ is \emph{$\RV$-thin} if, for every cardinal $\kappa>|\cT|$, every $\kappa$-saturated model $M\models\cT'$ and every set of parameters $A\subset M$ with $|A|<\kappa$, the set $\Inj_A$ contains no ball.
\end{dfn}

In the next section, we will work from our hypotheses to deduce good dimension theory, cellular decompositions and motivic integration. From now on, when we take a parameter set $A$, we always assume $A=\dcl_{\cL'}(A)$. We we write "$A$-definable", we always mean $\cL'(A)$-definable, with no more parameters; while "definable" without further precision means $\cL'(A)$-definable for some set of parameters $A$ we do not want to make precise.

From now on and until the end of section \ref{ssec_intmot_vol}, we always assume that $\cT'$ has $\RV$-partitions relative to $\cT$.

\subsection{Dimension on \texorpdfstring{$\VF$}{VF} : definitions and properties}
In this section we assume that $\cT'$ is $\RV$-thin. Now we want to endow subsets $\VF$ with a notion of dimension. This dimension will satisfy van den Dries axioms \cite{dries_dimension_1989} and a few other properties, see Theorem \ref{thm_dim}. To define it, we will make use of the $\RV$-partitions, as explained in the appendix \ref{sec_appendice_dim}. We will give several definitions and prove their equivalence and they  satisfy a number of good properties. They will obviously not satisfy all the properties of dimension in a 1-h-minimal theory; for instance, it would be unreasonable to give a non-zero dimension to a set that maps into $\RV$, while such a set could very well be infinite as will the examples show.

Recall the definition from the appendix:

\begin{dfn}
For a definable subset $X\subset\VF^n\times\RV^m$ and $\chi:X\to Y$ an $\RV$-partition, we define $\dim_\VF (X,\chi) = \max_{y\in Y}\dim_{\cT} \chi^{-1}(y)$, where $\dim_\cT$ is the dimension in the sense of the 1-h-minimal theory $\cT$.
\end{dfn}
Obviously this is well defined, by virtue of the properties of $\RV$-partitions and the dimension in $\cT$, and $\dim_\VF(X,\chi)\leq n$.

In order to apply the results from the appendix, we only need proving equality \ref{egalite_dim_sec}. To do so, we will use the algebraic closure. The following lemma is a classical fact, see for instance \cite{cluckers_hensel_2022}, Lemma 2.5.3. With respect to this reference, we avoid $\exists^\infty$-elimination, which does not hold in our situation, by means of a saturation hypothesis.

\begin{lem}\label{lem_ens_fini_int_RV}
Let $X\subset\VF^n$ be a finite $\emptyset$-definable set. Then there is a $\emptyset$-definable subset $Y\subset\RV^m$ such that $X$ and $Y$ are in $\emptyset$-definable bijection.

This holds for families: if $Z$ is a $\emptyset$-definable set, and $f:X\to Z$ is a $\emptyset$-definable map with finite fibres (in a sufficiently saturated model), then there is a $\emptyset$-definable set $Y\subset\RV^m\times Z$ and a $\emptyset$-definable bijection $g:X\to Y$ such that $f\circ g^{-1}$ is the projection.
\end{lem}
\begin{proof}
By saturation and compactness, the cardinal of the fibres of $f$ is bounded. Up to $\emptyset$-definably partitioning $Z$ according to this cardinal, we may even assume that it is constant, equal to $N$. Reason by induction on $N$.

If $N=0$, there is nothing to do. If $N=1$, take $Y=Z$ and $g=f$. In general, let $h:Z\to\VF^n$ be given by $h(z) = \frac 1N\sum_{x\in f^{-1}(z)}x$. The map $\chi:x\in f^{-1}(z)\mapsto \rv(x-h\circ f(x))\in\RV^n$ is non-constant; so for all $(z,\xi)\in Z\times\RV^m$, the fibre of $(f,\chi):X\to Z\times\RV^n$ at $(z,\xi)$ contains strictly less than $N$ elements. By induction hypothesis, we find $Y\subset\RV^{m+n}\times Z$ such that $X$ and $Y$ are in $\emptyset$-definable bijection relative to $Z\times\RV^n$, and in particular relative to $Z$.
\end{proof}

\begin{lem}\label{lem_acl_echange}
Let $M\models\cT$ be a model. The algebraic closure relative to $\Inj_\emptyset(M)$ has the exchange property: if $a,b\in\VF$ are such that $b\in\acl_{\cL'}(ac\cup\Inj_\emptyset(M))\setminus\acl(c\cup\Inj_\emptyset(M))$, with $c\in\VF^n$, then also $a\in\acl(bc\cup\Inj_\emptyset(M))$.
\end{lem}
\begin{proof}
By Lemma \ref{lem_dcl_acl}, we have $\acl_{\cL'}(ac\cup\Inj_\emptyset(M)) = \acl_{\cL}(ac\cup\Inj_\emptyset(M))$ and $\acl_{\cL'}(c\cup\Inj_\emptyset(M)) = \acl_{\cL}(c\cup\Inj_\emptyset(M))$. Consequently, exchange property for 1-h-minimal theories (\cite{cluckers_hensel_2022}, Lemma 5.3.6) ensures that $a\in\acl_{\cL}(bc\cup\Inj_\emptyset(M))$, hence that $a\in\acl_{\cL'}(bc\cup\Inj_\emptyset(M))$.
\end{proof}

Write $\mathrm{rg}^{\acl}_{\Inj}$ for the rank we deduce from this Lemma (corresponding to the algrbraic closure relative to $\Inj_\emptyset(M)$). This rank is invariant by elementary extension. Also write $\dim^{\acl}_{\Inj}$ for the dimension it induces. Remark that this dimension can obviously only be computed in $|\cT'|^+$-saturated models, for generic points to exist.

The proof of Lemma \ref{lem_acl_echange} did not involve $\cT'$ being $\RV$-thin. This will be used in the next lemma:

\begin{lem}\label{lem_acl_sn}
Assume that $\cT'$ is $\RV$-thin relative to $\cT$. Let $M\models\cT'$ be a model, which we assume to be saturated enough. Let $B_1,\dots, B_n$ be (open or closed) balls in $M$. Then there are $b_i\in B_i, 1\leq i\leq n$, such that the $b_i$ are algebraically independent over $\Inj_\emptyset(M)$.
\end{lem}
\begin{proof}
We prove this by induction on $n$. It is enough to show that, for every finite $A\subset \VF(M)$ and every ball $B$, there is $b\in B$ which is not $\cL'$-algebraic over $S_M$. Then, adding the parameters $A$ to the language, we may assume $A=\emptyset$. One can also assume that $B$ is definable.

Assume by contradiction that $B(M)\subset\acl_{\cL'}(\Inj_\emptyset(M))$. Then, for all $b\in B(M)$, there are $s_b\in \Inj_\emptyset(M)$, $\varphi_b(x,y)$ an $\cL$-formula and $N_b\in\bN$ such that \[M\models \exists^{\leq N_b}x,\varphi_b(x,s_b)\wedge\varphi_b(b,s_b)\,.\] If $\varphi$ is an $\cL$-formula, $N\in\bN$ a natural integer and $Y$ an $\cL'(\emptyset)$-definable set which maps into $\RV$ ($\cL'(\emptyset)$-definably), we define the formula $\psi_{\varphi,N,Y}(x)$ by setting \[\psi_{\varphi,N,\iota}(x) = \exists y\in Y, [\varphi(x,y)\wedge \exists^{\leq N}z,\varphi(z,y)]\,,\]
so that for each $b\in B(M)$ it holds that $M\models\psi_{\varphi_b,N_b,Y_b}(b)$ for some $Y_b$. By compactness (using that $M$ is sufficiently saturated), there is a finite number of formulas $(\varphi_i(x,y_i),N_i,Y_i)_{i=1}^n$ such that, for all $b\in B$, we have $M\models\bigvee_{i=1}^n\psi_{\varphi_i,N_i,Y_i}(x)$. We may assume that all $y_i$ have the same arity, and that $N_1=\dots=N_n=N$. Then we replace $\varphi_i(x,y)$ by
\[\widetilde\varphi_i(x,y):\varphi_i(x,y)\wedge \exists^{\leq N}z,\varphi_i(z,y)\,,\]
and define $\varphi(x,y)$ as the disjunction of the $\widetilde\varphi_i(x,y)$. If we let $Y=Y_1\cup\dots\cup Y_n$, we then find:
\[M\models\forall x\in B, \psi_{\varphi,nN,Y}(x)\,.\]

Let $g:Y\to R$ be an $\cL'$-definable bijection, with $R\subset\RV^l$. The formula $\varphi(x,g^{-1}(r))$ defines a family $X$ of (uniformly) finite sets indexed by $r\in R$, whose projection on the coordinate $x$ equals $B$. By Lemma \ref{lem_ens_fini_int_RV}, this set is in bijection with a subset $Z\subset R\times\RV^m$. Then we find a surjective map from $Z\subset\RV^l\times\RV^m$ onto $B$, contradiction.
\end{proof}

Now we are able to prove the following fact:

\begin{lem}\label{lem_dim_T-def}
Let $M\models\cT'$, and let $X\subset\VF^n\times\RV^m$ be an $\cL$-definable subset (with parameters). Then, for any $\RV$-covering $(Y_r)_{r\in R}, R\subset\RV^l$ of $X$, the equalities
\[\max_{r\in R}\dim_\cT Y_r = \dim_\cT X = \dim^{\acl}_{\Inj} X\]
hold.
\end{lem}
\begin{proof}
Start with the second equality, and set $d=\dim_\cT X$. We assume that $M$ is $|\cT'|^+$-saturated, in order to be able to compute $\dim^{\acl}_{\Inj} X$. Then \[\dim^{\acl}_{\Inj}X\leq \dim^{\acl}_{\cT} X = d\]
according to 1-h-minimal dimension theory (\cite{cluckers_hensel_2022}, Section 5.3). On the other hand, by \cite{cluckers_hensel_2022}, Proposition 5.3.4.(1), there is a subset of the $\VF$-sort coordinates such that the projection along these coordinates contains a $d$-dimensional polydisc. By Lemma \ref{lem_acl_sn}, there is in this polydisc a point $b$ whose coordinates are algebraically independent relative to $\Inj_\emptyset(M)$. By taking any point in the preimage of $b$ inside $X$, we deduce the converse inequality.

We now turn to the first equality. By 1-h-minimal dimension theory for $\cT$, we have $\dim_\cT Y_r\leq\dim_\cT X$ for all $r$, and on the other hand by the second equality $\dim_\cT Y_r=\dim_\cT X$ for some $r$ containing a point of maximal rank.
\end{proof}

As a consequence, the results from Appendix \ref{sec_appendice_dim} apply to this setting:

\begin{prop}
Let $X\subset\VF^n\times\RV^m$ be a definable subset. Then the dimension of $X$ does not depend on the choice of an $\RV$-partition; denote it by $\dim_\VF X$. Moreover, this dimension coincides with $\dim^{\acl}_{\Inj}$.
\end{prop}
\begin{proof}
This follows from the above and Proposition \ref{prop_appendice}.
\end{proof}

Now we may state the following theorem, which summarises the main properties of dimension.

\begin{thm}\label{thm_dim}
Assume that $\cT'$ is $\RV$-thin and has $\RV$-partitions relative to $\cT$. Let $X,Y\subset\VF^n\times \RV^m$ be definable sets, and let $Z\subset\VF^{n'}\times \RV^{m'}$ be definable. Let $f:X\to Z$ be a definable map. Then:
\begin{enumerate}
\item\label{item_thm_dim_1} For all $d\leq n$, we have $\dim_\VF X\geq d$ if and only if there is a projection $\VF^n\to\VF^d$ along a subset of the coordinates such that the image of $X$ in $\VF^d\times \RV^m$ has nonempty interior.
\item\label{item_thm_dim_2} $\dim_\VF(X\cup Y)=\max(\dim_\VF X,\dim_\VF Y)$.
\item\label{item_thm_dim_3} For all $d\leq n$, the set of $z\in Z$ such that $\dim_\VF f^{-1}(z)=d$ is definable over the same parameters as $f$.
\item\label{item_thm_dim_4} If all the fibres have the same dimension $d$, then $\dim_\VF X=d+\dim_\VF Z$.
\item\label{item_thm_dim_5} There is $x\in X$ such that the local dimension of $X$ at $x$ equals $\dim_\VF X$, \emph{i.e.} for every open ball $B\subset\VF^n$ containing $x$, we have $\dim_\VF(X\cap B)=\dim X$.
\item\label{item_thm_dim_6} If $X\subset\VF^n$, then $\dim_\VF(\overline X\setminus X)<n$.
\item\label{item_thm_dim_7} Dimension is invariant upon adding parameters.
\item\label{item_thm_dim_8} In a sufficiently saturated model, dimension $A$-definables sets coincides with $\acl$-dimension relative to $\Inj_A$.
\end{enumerate}
\end{thm}\begin{proof}
Items \ref{item_thm_dim_2}, \ref{item_thm_dim_3}, \ref{item_thm_dim_4}, \ref{item_thm_dim_7} follow from van den Dries axioms, see \cite{dries_dimension_1989}. Item \ref{item_thm_dim_8} comes from Lemma \ref{lem_dim_T-def}.
As for Item \ref{item_thm_dim_6}, let $\chi:\overline X\setminus X\to Y$ be an $\RV$-partition of $\overline X\setminus X$. Each fibre of $\chi$ necessarily has empty interior, so has dimension strictly less than $n$ (for $\cT$), proving the result.

We finally prove items \ref{item_thm_dim_1} and \ref{item_thm_dim_5}; they follow  from analogous statements for $\cT$-dimension in \cite{cluckers_hensel_2022} via $\RV$-partitions. Let $\chi:X\to R$ be an $\RV$-partition of $X$. Assume $\dim_\VF X\geq d$; then there is $r\in R$ such that $\dim_\VF \chi^{-1}(r)\geq d$. By \cite{cluckers_hensel_2022}, Proposition 5.3.4.(1), there is a projection $\VF^n\to\VF^d$ along a subset of the coordinates such that the image of $\chi^{-1}(r)$ inside $\VF^d\times\RV^m$, and \emph{a fortiori} that of $X$, has nonempty interior. Conversely, assume we have such a projection that the image of $X$ has nonempty interior; let $\pi$ be this projection. Then $\dim_\VF X \geq \dim_\VF\pi(X)=d$, where the inequality comes from \cite{dries_dimension_1989} and the equality from Lemma \ref{lem_acl_sn}.

Now let $r\in R$ be such that $\dim_\VF X = \dim_\VF \chi^{-1}(r)$. According to \cite{cluckers_hensel_2022}, Proposition 5.3.4.(5), there is $x\in \chi^{-1}(r)$ such that the local dimension of $\chi^{-1}(r)$ at $x$ equals $\dim_\VF \chi^{-1}(r)$. In particular, the local dimension of $X$ at $x$ equals $\dim_\VF X$.
\end{proof}
\begin{rmk}
The first five items are almost \emph{verbatim} those from \cite{cluckers_hensel_2022}, Proposition 5.3.4, the only difference being that it is no longer true that dimension $0$ sets are automatically finite, and that here we chose to allow auxiliary sorts. Indeed, the main reason why Cluckers, Halupczok and Rideau \cite{cluckers_hensel_2022} do not include these is that it would make wrong their reinforced conclusion for Item 1, but this has no chance to hold anyway in our context.

However, we have had to weaken Item 6 from \cite{cluckers_hensel_2022} Proposition 5.3.4, about the dimension of the boundary, by only taking the "easy" case where the definable subset $X\subset\VF^n$ has maximal dimension $n$. The stronger version is here false without further hypothesis: if we let $\sn:\RV^\times\to\VF^\times$ be a section of $\rv$, we will show in Corollary \ref{cor_section} that the theory we obtain by adding $\sn$ to the language satisfies the above assumptions; nevertheless, $\sn(\RV^\times)$ is not closed since $0$ lies in its closure. The following question provides a conjectural hypothesis under which the stronger conclusion could hold:
\end{rmk}
\begin{question}
Assume that, for every definable subset $X\subset\VF^n$ of dimension $0$, we have $\overline X = X$. Is it then true that, for all definable subsets $X\subset\VF^n$, we have $\dim_\VF(\overline X\setminus X)<\dim_\VF X$?

In other words, does the obstruction to this result entirely lie in dimension $0$?
\end{question}
Note that the hypothesis that dimension $0$ sets are closed is reasonable: it is in particular the case if we take $\cL=\cL_{val}$ and $\cL'=\cL\cup\{\sn\}$ and we interpret $\sn$ as a section $\sn:\Kb^\times\to\VF^\times$ of $\rv$, compatible to sum and product. More generally, if $\Delta\subset\Gamma$ is a bounded subgroup (for instance, a proper convexe subgroup), we may take $R=\{r\in\RV^\times:|r|\in\Delta\}$ and a section $\sn:R\to\VF^\times$ of $\rv$. The fact that these theories satisfy the hypothesis of the question will be a consequence of the results from Section \ref{sec_henselien}, more precisely of Corollary \ref{cor_section}.

Another way to motivate this question would be to notice that, if the theory $\cT'$ have $t$-stratifications (in the sense of \cite{halupczok_non-archimedean_2014}, Theorem 4.12 et Lemma 3.17), then the same argument as in \cite{cluckers_hensel_2022} applies and we find the result. However, the theory of $t$-stratifications is unlikely to be well-behaved in the aforementionned examples. For instance, while a $0$-dimensional subset (which is, a finite subset) only had the identity map as a risometry from itself to itself (let us call these autorisometries) in \cite{halupczok_non-archimedean_2014}, now $0$-dimensional subsets as simple as $\sn(\Kb)$ violate this principle in the above example. Even worse: if $|t|<1$, the set of autorisometries of $\sn(\Kb)+t\sn(\Kb)$ is properly $\mathrm{ind}$-definable, \emph{i.e.} there are such autorisometries with arbirary complexity.

On the opposite, if we rather add sections from the value group $\Gamma$ (with compatibility to the product), then we can prove elementarily that $0$-dimensional subsets no longer have other autorisometries that the identity map. It seems then reasonable to believe that $t$-stratifications could exist and be well-behaved in that context (in which, conversely, the hypothesis of the question is clearly false). Studying these theories and proving directly a kind of Hensel minimality (with all its consequences) will be the topic of a future article.

\subsection{Cell decomposition}\label{ssec_dec_cell}
Here we write the results on cell decompositions we will need to construct motivic integration. Note that we weakened a bit the notion of a cell decomposition adapted to a subset of $\VF^n\times\RV^m$, by working in $\cL$ relative to a definable set that maps into $\RV$. The results of this section will be immediate from 1-h-minimality and compactness. We also introduce the last important hypothesis for motivic integration, which is the notion of refined theories, in Definition \ref{def_theorie_raffinee}.

First let us recall the classical definitions of cells and cell decomposition.

\begin{dfn}
Let $A$ be some set of parameters (which may be imaginary), and let $X\subset\VF^n$ be an $A$-definable nonempty set. For $i=1,\dots,n$ we set some $\cL(A)$-definable maps $c_i:\pi_{<i}(X)\to\VF$. We say that $X$ is a \emph{twisted $\cL(A)$-definable box} of \emph{center} $c=(c_i)_{1\leq i\leq n}$ if there is $s\in\RV^n$ such that
\[X=\{x\in\VF^n:(\rv(x_i-c_i(x_{<i})))_{1\leq i\leq n}= s\}\,.\]
In this situation, we write $X=\rtimes (c,s)$. 
\end{dfn}

\begin{dfn}
An $\cL(A)$-definable \emph{cell decomposition} of $\VF^n$ is the datum of an $\cL(A)$-definable map $\chi:\VF^n\to\RV^N$ such that, for all $\xi\in\RV^N$, the fibre $\chi^{-1}(\xi)$ is a twisted $\cL(A\xi)$-definable box, whose center $c(\xi)$ we ask to depend $\cL(A)$-definably of $\xi$. The maps $c$ are part of the data of the cell decomposition.

Then a \emph{twisted box of $\chi$} is merely a fibre of $\chi$.
\end{dfn}
Remark that, in that case, the $n$-tuple $s$ also depends $\cL(A)$-definably of $\xi$.

We introduce the notion of cell decomposition convenient for our situation:
\begin{dfn}
Let $X\subset\VF^n\times\RV^k$ be an $A$-definable set, for some parameter set $A$. An $A$-definable \emph{cell decomposition} adapted to $X$ then is a couple of $A$-definable maps $(\psi:X\to Y, \chi:\VF^n\times Y\to \RV^N)$ where $Y$ maps into $\RV$, $psi$ is an $\RV$-partition and, for all $y\in Y$, the map $\chi_y:\VF^n\to\RV^N$ is an $\cL(yA)$-definable cell decomposition adapted to $X_y = \psi^{-1}(y)$ (in the classical 1-h-minimal sense), \emph{i.e.} for all $u\in\RV^k$ the fibre $X_{y,u}$ is a union of twisted boxes of $\chi_y$.

If $(\psi_i:X_i\to Y_i,\chi_i)$ are two cell decompositions adapted to $X$, we say that $(\psi_1,\chi_1)$ \emph{refines} $(\psi_2,\chi_2)$ if:
\begin{enumerate}
\item the $\RV$-partition $\psi_1$ refines $\psi_2$, \emph{i.e.} there is $\pi:Y_1\to Y_2$ such that $\psi_2=\pi\circ\psi_1$;
\item for all $y\in Y_1$, every fibre of the map $\chi_2(\cdot,\pi(y))$ is a union of fibres of $\chi_1(\cdot,y)$.
\end{enumerate}
\end{dfn}

Thanks to the existence of $\RV$-partitions, we immediatly get that of cell decompositions from the 1-h-minimal cell decomposition theorem (\cite{cluckers_hensel_2022}, Theorem 5.7.3; voir aussi \cite{stout_integration_2025}, Theorem 2.5.5):

\begin{thm}\label{thm_dec_cell}
Let $A$ be some parameters, and let $X\subset\VF^n\times\RV^m$ be an $A$-definable set. Then there is a cell decomposition $(\psi:X\to Y,\chi:\VF^N\times Y\to \RV^N)$ adapted to $X$. Moreover, we can ask for $(\psi,\chi)$ to satisfy the following additional properties:
\begin{enumerate}
\item\label{item_thm_dec_cell_1} (Preparing functions) if there is an $A$-definable map $\xi:X\to\RV^l$ given, we may assume $\xi$ to be constant on each $\chi_y^{-1}(r), y\in Y, r\in\RV^N$;
\item\label{item_thm_dec_cell_2} (Compatible domain and image preparation) let us assume $n=1$, and let $X_i\subset\VF\times\RV^{m_i},i=1,2$ be $A$-definable, with $Y_i$ mapping into $\RV$. Let $f:X_1\to X_2$ be an $A$-definable bijection. Then there are cell decompositions $(\psi_i:X_i\to Y, \chi_i:\VF\times Y\to\RV^{N_i})$ adapted to $X_i,i=1,2$, such that $f$ descends to a bijection between the sets \[\{(\chi_i(x,y),y,r):(x,r)\in X_i, y=\psi_i(x,r)\}, i=1,2\,.\]
\end{enumerate}

Finally, if $(\psi_i:X\to Y_i,\chi_i),i=1,2$ are two cell decompositions adapted to $X$, there is $(\psi,\chi)$ a cell decomposition adapted to $X$ which refines both of them.
\end{thm}
\begin{proof}
We prove the main assertion together with \ref{item_thm_dec_cell_1}. Let $\psi_0:\Gamma_\xi\to Y$ be an $\RV$-partition of the graph of $\xi$. In particular, the composition of $\psi_0$ with the inclusion $i:X\to\Gamma_\xi$ defines an $\RV$-partition $\psi=\psi_0\circ i$ of $X$. For all $y\in Y$, set $X_y=\psi^{-1}(y)$. The restriction $\xi_{|X_y}:X_y\to\RV^l$ is $\cL(y)$-definable; so by \cite{cluckers_hensel_2022}, Theorem 5.7.3, Addendum 1, there is an $\cL(y)$-definable cell decomposition $\chi_y:\VF^n\to\RV^{N_y}$ adapted to $X_y$ and $\xi_{|X_y}$. We glue by compactness, yielding the result.

Now let us show compatible domain and image preparation. Let $\Gamma_f$ be the graph of $f$, and let $\pi_i:\Gamma_f\to X_i$ be the projections. Let $\psi_0:\Gamma_f\to Y$ be an $A$-definable $\RV$-partition. For every $y\in Y$, we take $\chi_{1,y}:\VF\to \RV^{N_{1,y}},\chi_{2,y}:\VF\to\RV^{N_{2,y}}$ to be $\cL(yA)$-definable cell decompositions that satisfy compatible domain and image preparation for $f_y:X_{1,y}\to X_{2,y}$, which exist by \cite{stout_integration_2025}, Lemma 5.4.6. Again we glue by compactness.

Finally we turn to the additional item. Let $W$ be the graph of $(\psi_1,\psi_2)$, and let $\iota:X\to W$ be the injection. Let $\widetilde\psi:W\to Y$ be an $\RV$-partition of $W$ refining the projection $W\to Y_1\times Y_2$, and set $\psi=\widetilde\psi\circ\iota$ which is an $\RV$-partition of $X$. Write $\pi_i:Y\to Y_i$ for the canonical map factorising $\psi_i$. For $y\in Y$, the map $\chi_i(\cdot,\pi_i(y))$ is $\cL(y)$-definable; conclude like the main assertion.
\end{proof}

\begin{rmk}\label{rmk_dec_cell}
\begin{enumerate}
\item We have actually proven that, for every $\RV$-partition $\psi:X\to Y$, there is $\chi:\VF^n\times Y\to \RV^N$ such that $(\psi,\chi)$ is a cell decomposition adapted to $X$.
\item The objects of the categories we will use for motivic integration will be of the form $X\subset\VF^d\times Y\times\RV^m$, where $Y\subset\VF^l$ maps into $\RV$. Since $d$ is the dimension, given the particular role of dimension $1$, it would be quite inconvenient to see $X$ as a subset of $\VF^{d+l}\times\RV^m$. We will rather take advantage of the fact that $Y$ maps into $\RV$ by saying that a cell decomposition adapted to $X$ is a couple $(\psi,\chi)$, where $\psi:X\to Y'$ is an $\RV$-partition which additionally refines the projection $X\to Y$, and $\chi:\VF^d\times Y'\to\RV^N$ is an $\emptyset$-definable map such that, for all $y'\in Y'$, the restriction $\chi_{y'}:\VF^d\to\RV^N$ is an $\cL(y')$-definable cell decomposition, and moreover for all $(y,\xi)\in y\times\RV^m$ the set $\{x\in\VF^d:(x,y,\xi)\in X\wedge\psi(x,y,\xi)=y'\}$ is a union of fibres of $\chi$. In other words, instead of seeing $Y$ a variables in $\VF$, we see it a variables in $\RV$.
\item Actually, the objects we will look at will be of the aforementioned form with the additional property that $X\to Y$ be an $\RV$-partition. Sometimes, by abuse of language, we will just say "let $\chi:\VF^d\times Y\to\RV^N$ be a cell decomposition adapted to $X$", meaning (in a legitimate way, by the first point of this remark) that the underlying $\RV$-partition is the projection $X\to Y$.
\end{enumerate}
\end{rmk}

\subsubsection{Refined theories}\label{sssec_raffine}
Let us take a look forward and examine now how is built Hrushovski-Kazhdan motivic integral, and more precisely how we get to control the kernel of the lifting map $\fL$. In the original version \cite{ginzburg_integration_2006}, this is done through the means of so-called special bijections (Section 7.1), which have a very concrete description, whereas in the more recent version \cite{stout_integration_2025} it is done by using cell decompositions (Section 5.4). The key argument to connect these to versions comes from the fact that a definable map $\RV^m\to\VF$ must have finite image; this is crucal in \emph{loc. cit.}, Lemma 5.4.5, with an argument mostly taken from \cite{cluckers_constructible_2008}, Lemma 9.1.3. This is of course not anymore the case in our situation; that being said, we will anyway be capable of connecting them in the situation we are interested in, thanks to a synctactical argument inspired from \cite{yin_integration_2013}, Section 4.

\begin{dfn}
Let $X\subset\VF$ be an $\emptyset$-definable set, and let $(\psi:X\to Y,\chi:\VF\times Y\to\RV^N)$ be a cell decomposition adapted to $X$. Set $I = \{(\chi(x,y),y):x\in X, y=\psi(x)\}$. Then the partition of $X$ associated to $(\psi,\chi)$ is
\[X = \bigsqcup_{(\xi,y)\in I}\chi_y^{-1}(\xi)\,.\]
\end{dfn}
It is a partition of $X$ into points and open balls indexed by some set $I$ in $\emptyset$-definable bijection with a subset of $\RV$. For $\sigma\in I$, we let $B_\sigma = \chi_y^{-1}(\xi)$ and $r(\sigma)\in\Gamma_0$ be the radius of $B_\sigma$.

A partition satisfying these hypotheses will be called a \emph{cell partition}.

\begin{dfn}\label{def_theorie_raffinee}
Let $X\subset\VF$ be an $\emptyset$-definable set, and let $X = \bigsqcup_{\sigma\in I}B_\sigma$ be a cell partition of $X$. Let $I_0\subset I$ be an $\emptyset$-definable set such that $r(\sigma)>0$ for all $\sigma\in I_0$, and let $c:I_0\to\VF$ be an $\emptyset$-definable map such that, for every $\sigma\in I_0$, we have $c(\sigma)\in B_\sigma$. Write $I_1=I\setminus I_0$.

The \emph{refinement of length $1$} of the cell partition $X=\bigsqcup_{\sigma\in I}B_\sigma$ associated to these data is the cell partition
\[X=\bigsqcup_{\sigma\in I_1}B_\sigma\sqcup\bigsqcup_{\sigma\in I_0}\bigsqcup_{\xi\in\RV, |\xi|<r(\sigma)}(c(\sigma)+\rv^{-1}(\\xi))\,.\]

A \emph{refinement of length $n$} is the composition of $n$ refinements of length $1$. A \emph{refinement of finite length} is a refinement of length $n$ for some $n\in\bN$.

Say that a cell partition $X=\bigsqcup_{\sigma\in I}B_\sigma$ has \emph{finite length} if it can be extended to a cell partition $\VF=\bigsqcup_{\sigma\in I'}B'_\sigma$ of all $\VF$, which is in turn a refinement of finite length of the elementary cell partition $\VF=\bigsqcup_{\xi\in\RV}\rv^{-1}(\xi)$.
\end{dfn}
Let us explain more concretely what we have done. We had a partition of $X$ in points and open balls, and inside some of those balls we have added a center, then refined the partition by taking the $\rv$-balls centered at that new point.

\begin{rmk}
If the original cell partition corresponds to a cell decomposition, then the one we find after a refinement of length $1$ can be written in a canonical way as the partition corresponding to such a decomposition. However, we do not want to keep track of these decompositions, which would only make things more obscure here.

As a contrary, those cell partitions correspond to a maybe more natural notion of cell decomposition, in which we do not ask for $\cL$-definability. Nonetheless, we rather keep our notion of cell decomposition, more suitable to what we are doing in general.
\end{rmk}

Let us justify the terminology introduced above:

\begin{lem}\label{lem_raff_fini}
Let $X\subset\VF$ be an $\emptyset$-definable set, and let $\cP:X=\bigsqcup_{\sigma\in I}B_\sigma$ be a cell partition of $X$. Also let $\cP':X=\bigsqcup_{\sigma\in I'}B'_\sigma$ be a cell partition which refines $\cP$. We assume that $\cP'$ has finite length. Then $\cP'$ is a refinement of finite length of $\cP$.
\end{lem}
\begin{proof}
We still denote by $\cP'$ an extension of $\cP'$ to $\VF$ which is a refinement of finite length of the elementary cell partition, and let $(I_0^{(1)},c^{(1)}),\dots, (I_0^{(n)},c^{(n)})$ be a sequence of data making $\cP'$ a refinement of finite length of the elementary cell partition. Let $C^{(i)}$ be the image of $c^{(i)}, 1\leq n$, and also let $c^{(0)}=\{0\}$. We inductively define a sequence of cell partitions $\cP_i:X=\bigsqcup_{\sigma\in I_i}B^{(i)}_\sigma, 0\leq i\leq n+1$ such that:
\begin{itemize}
\item $\cP_0=\cP$;
\item for all $0\leq i\leq n$, each of the $B^{(i)}_\sigma,\sigma\in I_i$ contains at most one point of $C^{(i)}$. Then we define $I_0^{(i)}$ as the subset of those $\sigma$ such that moreover $r(\sigma)>0$, as well as $c'^{(i)}$ the map $I_0^{(i)}\to \VF$ which sends such a $\sigma$ to the only point of $C^{(i)}$ contained in $B^{(i)}\sigma$; 
\item for all $0\leq i\leq n$, the cell partition $\cP_{i+1}$ is the refinement of length $1$ of $\cP_i$ associated to $(I_0^{(i)},c'^{(i)})$.
\end{itemize}
Denote by $\widetilde\cP_i:\VF=\bigsqcup_{\sigma\in\widetilde I^{(i)}}\widetilde B_\sigma^{(i)}, 2\leq i\leq n+1,$ the cell partition of length $i$ defined using the data $(I_0^{(j)},c^{(j)})_{1\leq j<i}$ starting from $\widetilde\cP_1:\VF=\bigsqcup_{\xi\in\RV}\rv^{-1}(\xi)$. By induction, the partition $\cP_i$ refines the restriction of $\widetilde\cP_i$ to $X$ (which is only a partition, not necessarily a cellular one), for all $1\leq i\leq n+1$. In particular, $\cP_{n+1}$ refines $(\widetilde\cP_{n+1})_{|X}=\cP'$. Conversely, by induction, since $\cP'$ refines $\cP=\cP_0$, then $\cP'$ refines $\cP_i$ for every $0\leq i\leq n+1$. Finally $\cP_{n+1}=\cP'$ as we want.

The only thing left to prove is to justify the second statement. But this follows from the fact that $\cP_i$ refines $(\widetilde\cP_i)_{|X}$, which satisfies this assumption. 
\end{proof}

Now we introduce our new hypothesis:

\begin{dfn}
Say that $\cT'$ is \emph{refined} if, for every set of parameters $A$, every $A$-definable subset $X\subset\VF$, and every cell partitions $\cP_i,i=1,2$ of $X$, there is a cell partition $\cP$ of $X$ which is a refinement of finite length of each of the two partitions.
\end{dfn}
Probably less natural than the previous ones, this hypothesis is mostly an \emph{ad hoc} statement designed to easily yield control of the kernel. We will prove it in the examples using quantifier elimination relative to the theory $\cL$. The point of Lemma \ref{lem_raff_fini} is to reduce the proof of this hypothesis to that of the existence of a cell partition refining both given partition and of finite length (in an absolute sense), rather than of finite length relative to each of the two given partitions.

\subsection{Motivic integration without volume forms}
Here we will define the natural categories between which to do Hrushovski-Kazhdan motivic integration in $\cT'$, then as usual prove that lifting induces a surjection between Grothendieck semigroups, whose kernel is generated by a single element, thus proving a motivic integration theorem \ref{thm_int_sans}. For now we limit ourselves to the case without volume forms, which corresponds to a generalised Euler characteristic. Recall that we assume that $\cT'$ has $\RV$-partitions, but not that it is $\RV$-thin. From now on, we also assume that $\cT'$ is refined; we will use it only once, in Lemma \ref{lem_int1_ind_chi}.

\begin{dfn}\label{dfn_cat_VF}
For $d\in\bN$, let $\VF_{\cT'}[d]$ be the category whose objects are $\emptyset$-definable subsets $X\subset \VF^d\times Y\times\RV^m$, where $m\in\bN$ and $Y\subset\VF^l$ maps into $\RV$, such that the projection $X\to\VF^d\times Y$ has finite fibres and the projection $X\to Y$ is an $\RV$-partition. Morphisms are the $\emptyset$-definable maps.

We have inclusions $\VF_{\cT'}[d_1]\subset\VF_{\cT'}[d_2]$ for $d_1\leq d_2$; denote by $\VF_{\cT'}$ the union of all these categories.
\end{dfn}
The definition may look artificial or \emph{ad hoc}. We provide the reader with the following alternative definitions, less amenable but probably more natural-looking.

\begin{dfn}
For $d\in\bN$, let $\VF_{\cT'}^*[d]$ be the category whose objets are $\emptyset$-definable subsets $X\subset\VF^d\times Y\times\RV^m$, where $m\in\bN$ and $Y\subset\VF^l$ maps into $\RV$, such that the projection $X\to\VF^d\times Y$ has finite fibres. Morphisms are the $\emptyset$-definable maps.

Assume $\cT'$ to be $\RV$-thin. For $d\in\bN$, let $\VF_{\cT'}^\circ[d]$ be the category whose objects are $\emptyset$-definable subsets of $\VF^n\times\RV^m, n,m\in\bN$, of $\VF$-dimension at most $d$ and such that the projection to $\VF^n$ has finite fibres. Morphisms in this category are the $\emptyset$-definable maps.

Assume $\cT'$ to be $\RV$-thin and effective (in the sense of \cite{stout_integration_2025}, Definition 3.1.1). For  $d\in\bN$, let $\VF_{\cT'}^\dagger[d]$ be the category whose objects are $\emptyset$-definable subsets of $\VF^n$ with $\VF$-dimension at most $d$. Morphisms in this category are the $\emptyset$-definable maps.
\end{dfn}
\begin{lem}\label{lem_eq_cat}
For all $d\in\bN$, the inclusions $\VF_{\cT'}^\dagger[d]\subset\VF_{\cT'}^\circ[d]$ and $\VF_{\cT'}[d]\subset \VF_{\cT'}^*[d]\subset \VF_{\cT'}^\circ[d]$ are equivalences of categories.
\end{lem}
\begin{proof}
Full faithfullness is clear, so we only have to prove essential surjectivity. For the first inclusion, let $X\subset\VF^n\times\RV^m$ be an object in $\VF_{\cT'}[d]$. For $a\in\VF^n$, the theory $\cT'(a)$ is effective and the fibre $X_a\subset\RV^m$ is a finite $a$-definable set. By lifting property for finite sets from $\RV$ to $\VF$ in effective theories (\cite{stout_integration_2025}, Lemma 3.1.2), there is $X'_a\subset\VF^{N_a}$ in $a$-definable bijection with $X_a$. By compactness, we may glue them into an $\emptyset$-definable set $X'\subset\VF^{n+N}$ which is in $\emptyset$-definable bijection with $X$; in particular $\dim_\VF X'\leq d$, hence $X'\in\Ob\VF_{\cT'}^\dagger[d]$, which proves essential surjectivity for the first inclusion.

Let $X\subset\VF^d\times Y\times\RV^m$ be an object of $\VF_{\cT'}^*[d]$. Let $\chi:X\to Y'\subset\VF^{l'}$ be an $\RV$-partition of $X$. Then there is an $\emptyset$-definable bijection between $X$ and the graph of $\chi$, which is $\Gamma_\chi\subset \VF^d\times (Y\times Y')\times\RV^m$. This is indeed an object from $\VF_{\cT'}[d]$.

For the last equivalence, let us again take an $\RV$-partition, and reduce to the analogous fact for 1-h-minimal theories. Let $X\subset\VF^n\times\RV^m$ be an object of $\VF_{\cT'}[d]$. Let $\chi:X\to Y$ be an $\RV$-partition. For all $y\in Y$, the fibre $\chi^{-1}(y)$ is a subset of $\VF^n$ which is $\cL(y)$-definable and has dimension at most $d$; by \cite{stout_integration_2025}, Section 5.1, there is an $\cL(y)$-definable bijection $f_y:\chi^{-1}(y)\to X'_y\subset\VF^d\times\RV^k$, where the projection $X'_y\to \VF^d$ has finite fibres. We glue these by compactness to get an $\emptyset$-definable bijection $f:X\to X'\subset\VF^d\times Y\times\RV k$ such that $X'\to \VF^d$ has finite fibres, so in particular $X'\in\Ob\VF_{\cT'}^*[d]$.
\end{proof}

\begin{rmk}
Unlike what happens in the classical case, here the category $\VF_{\cT'}[d]$ (or $\VF_{\cT'}^*[d], \VF_{\cT'}^\circ[d]$) is not stable under definable bijections: for instance, if $Y\to R$ is an $\emptyset$-definable bijection between two infinite sets $Y\subset\VF^l$ and $R\subset\RV^m$, only $Y$ is an object of $\VF_{\cT'}[d]$ for all $d\in\bN$.
\end{rmk}

Taking inspiration from Definition \ref{dfn_cat_VF}, we may introduce our categories on $\RV$ and the lifting map $\fL$.

\begin{dfn}
For $d\in\bN$, let $\RV_{\cT'}[d]$ be the category whose objects are $\emptyset$-definable subsets $R\subset (\RV^\times)^d\times Y\times\RV^m,$ where $m\in\bN$ and $Y\subset\VF^l$ maps into $\RV$, such that the projection $R\to(\RV^\times)^d\times Y$ has finite fibres and the projection $X\to Y$ is an $\RV$-partition. Morphisms in this category are the $\emptyset$-definable maps.

Denote by $\RV_{\cT'}[\leq d]$ the direct sum of the $\RV_{\cT'}[d'],d'\leq d$, and by $\RV[*]$ the union of the $\RV_{\cT'}[\leq d]$.

We define a map $\fL:\Ob\RV_{\cT'}[d]\to\Ob\VF_{\cT'}[d]$ by setting, for $X\subset(\RV^\times)^d\times Y\times\RV^m$:
\[\fL X = \{(x,y,u)\in(\VF^\times)^d\times Y\times \RV^m:(\rv(x),y,u)\in X\}\,.\]
We again denote by $\fL$ the maps $\Ob\RV_{\cT'}[\leq d]\to\VF_{\cT'}[d]$ and $\Ob\RV_{\cT'}[*]\to\VF_{\cT'}[*]$ induced by the above definition.
\end{dfn}

\begin{rmk}
\begin{enumerate}
\item We may define, for $d\in\bN$, the category $\RV_{\cT'}^*[d]$ whose objects are $\emptyset$-definable subsets $R\subset(\RV^\times)^d\times Y\times\RV^m$, where $m\in\bN$ and $Y\subset VF^l$ maps into $\RV$, such that the projection $R\to(\RV^\times)^d\times Y$ has finite fibres.

Then again, as in Lemma \ref{lem_eq_cat}, the inclusion $\RV_{\cT'}[d]\subset\RV_{\cT'}^*[d]$ is an equivalence of categories.
\item We may see $\RV_{\cT'}[\leq d]$ as a category whose objects are $\emptyset$-definable subsets $R\subset\RV^d\times Y\times\RV^m$ such that $Y$ maps into $\RV$, $R\to Y$ is an $\RV$-partition and $R\to \RV^d\times Y$ has finite fibres. Morphisms are not anymore all the $\emptyset$-definable bijections, we have to add an additional "compatibility to the zeroes" condition, \emph{cf.} \cite{stout_integration_2025}, Remark 5.1.3.
\end{enumerate}
\end{rmk}

Recall that, to get a Hrushovski-Kazhdan motivic integral, we need three main results:
\begin{enumerate}
\item show that each object on $\VF$ is isomorphic to the lift of an object on $\RV$, which gives surjectivity;
\item show that, if the lifts of two objects on $\RV$ are in definable bijection, then these two objects on $\RV$ are congruent, which gives control of the kernel;
\item lifting definable bijections from $\RV$ to $\VF$, which gives good definition.
\end{enumerate}
As in \cite{stout_integration_2025}, we will prove the first two results without needing any further hypothesis on the theory. Note that, whereas the first one will be (as usual) rather straightforward, the second is not a direct consequence of their results via $\RV$-partitions. As in Yin \cite{yin_integration_2013}, it will be necessary to go through the main steps of their proof. As for the last result, we will make it a hypothesis which we will prove in a number of cases. That being said, as in \cite{stout_integration_2025}, we will be able to define the motivic integral without needing this hypothesis; this morphism will just not have to be injective.

We now prove the first result as announced. For the sake of future proofs, we rather show the following reinforced version:
\begin{lem}\label{lem_fL_surj}
Let $X\subset\VF^d\times Y\times\RV^m$ be an of $\VF_{\cT'}[d]$, and let $(\psi:X\to Y',\chi:\VF^d\times Y'\to \RV^N)$ be a cell decomposition adapted to $X$ (in the sense of Remark \ref{rmk_dec_cell}). Then there is an object $R$ of $\RV_{\cT'}[\leq d]$ and an $\emptyset$-definable bijection $g:X\to\fL R$ which descends to a bijection
\[\{(\chi(x,y'),y',\xi):(x,y,\xi)\in X, y'=\psi(x,y,\xi)\}\to R\,.\]
\end{lem}
\begin{proof}
Let $y'\in Y'$. Since $X\to Y'$ refines $X\to Y$ by assumption, the fibre $X_{y'}$ is an object of $\VF_{\cT(y')}[d]$. As $\chi_{y'}$ is $\cL(y')$-definable, by \cite{stout_integration_2025}, Lemma 5.3.1, there is an object $R_{y'}$ of $\RV_{\cT(y')}[\leq d]$ and an $\cL(y')$-definable bijection $g_{y'}:X_{y'}\to\fL R_{y'}$ which descends to a bijection
\[\{(\chi_{y'}(x),\xi):(x,\xi)\in X_{y'}\}\to R_{y'}\,.\]
We conclude by glueing by compactness.
\end{proof}

We also introduce the following definition:
\begin{dfn}
Say that $\cT'$ has the bijection lifting property if, for every model $M\models\cT'$, every set of parameters $A\subset\VF(M)$ and every $A$-definable bijection $h:R\to S$ between two objects of $\RV_{\cT'(A)}[*]$, there is an $A$-definable bijection $f:\fL R\to \fL S$ lifting $h$.
\end{dfn}
\begin{rmk}
The bijection lifting properties is obviously not preserved in general upon adding parameters from $\RV$, but it is preserved by adding parameters in $\VF$.
\end{rmk}

\subsubsection{Integration in dimension \texorpdfstring{$1$}{1}}
We proceed analogously to Stout and Vermeulen \cite{stout_integration_2025}, by constructing the converse.

\begin{dfn}
Let $\Isp^{\cT'}$ be the semiring congruence on $\Kp\RV_{\cT'}[*]$ generated by the relation $([\RV^\times_{<1}]_1+[1]_0,[1]_1)$.

We denote by $\Isp^{\cT'}[n]$ the induced semigroup congruence by $\Isp^{\cT'}$ on $\Kp\RV_{\cT'}[\leq n]$, or $\Isp^{\cT'}$ when it is clear from the context.

For $A\subset \VF\cup\RV$, let $I_{\sp,A}^{\cT'}$ be the same relation, but relative to the language $\cL'(A)$.
\end{dfn}

Certainly $\Isp^{\cT'}$ is in the kernel of $\fL$ as soon as this makes sense. We show that there is nothing else. To do so, we closely follow \cite{stout_integration_2025}, reusing their results as much as possible. We were unable to deduce directly the final results from theirs, but nonetheless we are able to avoid quite a number of lemmas, such as going through relatively unary maps (Lemma 5.5.3) or Fubini (Lemma 5.5.14).

\begin{lem}\label{lem_Isp_ponctuel}
Let $R,S\in\Ob\RV_{\cT'}[\leq n]$, and let $f:R\to \RV^N,g:S\to \RV^N$. Assume that, for all $t\in \RV^N$, we have $([f^{-1}(t)],[g^{-1}(t)])\in I_{\sp,t}^{\cT'}$. Then $([R],[S])\in\Isp^{\cT'}$.
\end{lem}
\begin{proof}
We follow the exact same proof as \cite{stout_integration_2025}, Lemma 5.4.2, with a correction indicated to us by the authors, who we thank for that. For $t\in\RV^N$, we will write $R_\xi=f^{-1}(t)\subset R, S_t = g^{-1}(t)\subset S$. By definition of the congruence generated by a relation, there are for each $t\in\RV^N$ an integer $k\in\bN$ and objects $T^{(l)}_{i,j,t}\in\Ob\RV_{\cT'}[\leq n], i,j\leq n, l\leq k$, such that, in the semiring $\Kp\RV_{\cT'}[*]$, we have the equalities
\[\begin{cases}
[R_t] = \sum_{i,j\leq n} [T^{(0)}_{i,j,t}]([\RV^\times_1]_1+[1]_0)^i([1]_1)^j\,, \\
[S_t] = \sum_{i,j\leq n} [T^{(k)}_{i,j,t}]([\RV^\times_1]_1+[1]_0)^j([1]_1)^i\,, \\
\sum_{i,j\leq n} [T^{(l)}_{i,j,t}] ([\RV^\times_1]_1+[1]_0)^i([1]_1)^j = \sum_{i,j\leq n} [T^{(l-1)}_{i,j,t}] ([\RV^\times_1]_1+[1]_0)^j([1]_1)^i\,, 1\leq l\leq k\,.
\end{cases}\] 
The sets $T^{(l)}_{i,j,t}$ are $t$-definable; by compactness, we may glue them together into sets $T^{(l)}_{i,j}$, and assume that $k$ does not depend on $t$. Moreover, the first equality allows us to see $T^{(0)}_{i,j,t}\times\{1\}_0^i\times\{1\}_1^j$ as a $t$-definable subset of $R_t$, hence after glueing we may see $T^{(0)}_{i,j}$ as an $\emptyset$-definable subset of $R$, and in particular as an object of $\RV_{\cT'}[\leq n]$. Following the same reasoning, by induction each $T^{(l)}_{i,j}$ is an object of $\RV_{\cT'}[\leq n]$.

By compactness, we can glue the bijections which were giving the above equalities for all $t\in\RV^N$, and deduce
\[\begin{cases}
[R] = \sum_{i,j\leq n} [T_{i,j}]([\RV^\times_1]_1+[1]_0)^i([1]_1)^j\,, \\
[S] = \sum_{i,j\leq n} [T_{i,j}]([\RV^\times_1]_1+[1]_0)^j([1]_1)^i\,, \\
\sum_{i,j\leq n} [T^{(l)}_{i,j}] ([\RV^\times_1]_1+[1]_0)^i([1]_1)^j = \sum_{i,j\leq n} [T^{(l-1)}_{i,j}] ([\RV^\times_1]_1+[1]_0)^j([1]_1)^i\,, 1\leq l\leq k\,.
\end{cases}\]
So $([R],[S])\in\Isp^{\cT'}$.
\end{proof}

\begin{dfn}
Let $X\subset\VF\times Y\times\RV^m$ be an object of $\VF_{\cT'}[1]$, and let $(\psi:X\to Y',\chi:\VF\times Y'\to \RV^N)$ be a cell decomposition adapted to $X$ (in the sense of Remark \ref{rmk_dec_cell}). Then we let
\[I_{\psi,\chi} X = \{(\chi(x,y'),y',\xi):(x,y,\xi)\in X, y'=\psi(x,y,\xi)\}\,,\]
which according to Lemma \ref{lem_fL_surj} is isomorphic to an object of $\RV_{\cT'}[\leq 1]$, thus induces a well-defined isomorphism class inside $\RV_{\cT'}[\leq 1]$.
\end{dfn}

Actually this does not depend on the choice of $(\psi,\chi)$. This crucial fact explains simultaneously the quotient by $\Isp^{\cT'}$ and the hypothesis that $\cT'$ be refined. It is the only place in the article where we will use this hypothesis (together with the analogous fact in the case with volume forms).

\begin{lem}\label{lem_int1_ind_chi}
Let $X\in\VF_{\cT'}[1]$, $X\subset\VF\times Y\times\RV^m$, and let $(\psi_i:X\to Y_i,\chi_i),i=1,2$ be two cell decompositions adapted to $X$. Then $[I_{\psi_1,\chi_1}X]=[I_{\psi_2,\chi_2}X]$ in $\Kp\RV_{\cT'}[\leq 1]/\Isp^{\cT'}$.
\end{lem}
\begin{proof}
Since $\psi_1,\psi_2$ refine the projection map $X\to Y$ by assumption, we are given maps
\[f_i: I_{\psi_i,\chi_i} X\to Y\times\RV^m\]
such that $f_i^{-1}(y,\xi) = I_{\psi_i,\chi_i} X_{y,\xi}$ for all $(y,\xi)\in Y\times\RV^m$. By Lemma \ref{lem_Isp_ponctuel}, we may therefore assume $X\subset\VF$. In that case, note that $I_{\chi_i,\psi_i}X$ is exactly the set $I$ from Section \ref{sssec_raffine}, and that the corresponding partition is the same; so in particular, since we are assuming $\cT'$ to be refined, there is a cell partition of $X$ which is a finite length refinement of both the partitions of $X$ associated to $(\chi_i,\psi_i)$. Note that, to such a refinement, we may associate an object of $\RV_\cT'[\leq 1]$: indeed, it is enough to prove it by induction. For the base case, it is a consequence of Lemma \ref{lem_fL_surj}. Then, when we take a refinement of length $1$, by using the notation of Section \ref{sssec_raffine}, it is enough to show that we may indeed see $I'_0=\{(\sigma,\xi)\in I_0\times\RV:|\xi|<r(\sigma)\}$ as an object of $\RV_{\cT'}[\leq 1]$. But we have at our disposal an injective map $c:I_0\to \VF$; by identifying $I_0$ to $c(I_0)$ we find a set that maps into $\RV$. Then the result is rather clear. Moreover, remark that $[I'_0] = [I_0]_0\cdot[\RV_{<1}] = [I_0]_0([\RV^\times_{<1}]_1+[1]_0)$. As a consequence, modulo $\Isp^{\cT'}$, we have $[I'_0]=[I_0]$, so taking a refinement of length $1$ (and, by induction, also of finite length) preserves the class inside $\Kp\RV_{\cT'}[\leq 1]/\Isp^{\cT'}$. In particular, $I_{\psi_1,\chi_1}X= I_{\psi_2,\chi_2}X$.
\end{proof}

\begin{lem}\label{lem_int1_ind_bij}
Let $X_i\subset\VF\times Y_i\times\RV^{m_i}, i=1,2,$ be two objects of $\VF_{\cT'}[1]$, and let $f:X_1\to X_2$ be a definable bijection. Then there are $(\psi_i:X_i\to Y'_i,\chi_i),i=1,2,$ two cell decompositions adapted to $X_i,i=1,2$, such that moreover $f$ descends to a bijection $I_{\psi_1,\chi_1}X_1\to I_{\psi_2,\chi_2}X_2$.
\end{lem}
\begin{proof}
We have actually already proved it: it is Theorem \ref{thm_dec_cell}, item \ref{item_thm_dec_cell_2}.
\end{proof}

\begin{dfn}
For $[X]\in\Kp\VF_{\cT'}[1]$, we set $I[X]=[I_{\psi,\chi} X]\in\Kp\RV_{\cT'}[1]/\Isp^{\cT'}$, where $X$ is any representative of $X$ and $(\psi,\chi)$ any cell decomposition adapted to $X$.
\end{dfn}
Lemmas \ref{lem_int1_ind_chi} and \ref{lem_int1_ind_bij} prove that this is well-defined.

\subsubsection{Integration in dimension \texorpdfstring{$d$}{d}}

Now we may turn to higher dimensions.

\begin{dfn}
Let $k,l\in\bN$. We define $\VFR_{\cT'}[k,l]$ to be the category whose objects are the $\emptyset$-definable subsets $X\subset \VF^k\times\RV^l\times Y\times\RV^m$ with $m\in\bN$ and $Y\subset\VF^n$ an $\emptyset$-definable set mapping into $\RV$, such that the projection $X\to\VF^k\times\RV^l\times Y$ has finite fibres and the projection $X\to Y$ is an $\RV$-partition. Morphisms are $\emptyset$-definable functions relative to $\VF^k$ that preserve the zeroes.
\end{dfn}
\begin{rmk}
For $k=0$, there is an equivalence $\VFR_{\cT'}[0,l]\simeq\RV_{\cT'}[\leq l]$. For $l=0,k>0$, there is no equivalence $\VFR_{\cT'}[k,0]\simeq\VF_{\cT'}[k]$, because the morphisms are relative to $\VF^k$.
\end{rmk}

\begin{dfn}
Let $\Isp^{\cT'}$ be the equivalence relation on $\Kp\VFR_{\cT'}[k,l]$ given by:
\[\forall X,Y\in\Ob\VFR_{\cT'}[k,l], ([X],[Y])\in\Isp^{\cT'}\iff \forall a\in\VF^k, ([X_a],[Y_a])\in I_{\sp,a}^{\cT'}\,.\]
\end{dfn}

\begin{dfn}
Let $X\in\Ob\VFR_{\cT'}[k,l]$, say $X\subset\VF^k\times\RV^l\times Y\times\RV^m$, and let $1\leq j\leq k$ be an integer. Let $\psi:X\to Y'$ be an $\RV$-partition refining the projection $X\to Y$ and $\chi:\VF^k\times Y\to\RV^N$ an $\emptyset$-definable map such that, for all $a\in\VF^k$, the couple $(\psi_{\hat a_j},\chi(\hat a_j,\cdot):K\times Y'\to\RV^N)$ is a cell decomposition adapted to $X_{\hat a_j}$. Then we set
\[I^j_{\psi,\chi} X= \{(\hat x_j,\xi,\chi(x,y'),y',\xi')\in\VF^{k-1}\times\RV^l\times\RV^N\times Y'\times \RV^m:(x,\xi,y,\xi')\in X, y'=\psi(x,\xi,y,\xi')\}\,,\]
which is in bijection relative to $\VF^{k-1}\times\RV^l$ with an object of $\VFR_R[k-1,l+1]$ (of well-defined isomorphism class) according to Lemma \ref{lem_fL_surj}.
\end{dfn}
We have denoted by $\hat x_j$ the $(k-1)$-tuple obtained from $x$ by deleting the $x_j$ coordinate.
\begin{lem}\label{lem_int_VFR}
$I^j$ induces a well-defined homomorphism
\[I^j:\Kp\VFR_{\cT'}[k,l]\to\Kp\VFR_{\cT'}[k-1,l+1]/\Isp^{\cT'}\,.\]
\end{lem}
\begin{proof}
This is essentially a consequence of dimension $1$, as explained in \cite{stout_integration_2025}, Lemma 5.5.10. First, since everything (objects, morphisms, \emph{etc}) is relative to $\VF^{k-1}$, by compactness we may assume $k=1$. Let $X_i\subset\VF\times\RV^l\times Y_i\times\RV^{m_i},i=1,2$ be two isomorphic objects in $\VFR_{\cT'}[1,l]$, as well as $(\psi_i:X_i\to Y'_i,\chi_i),i=1,2$ two couples as in the definition, which is, since $k=1$, cell decompositions adapted to $X_i$. Let $f:X_1\to X_2$ be an isomorphism, \emph{i.e.} an $\emptyset$-definable bijection which commutes to the projections $X_i\to\VF$ and preserves the zeroes on $\RV^l$. Let $\psi':\Gamma_f\to Y$ be an $\RV$-partition of the graph of $f$ refining the projection $\Gamma_f\to Y'_1\times Y'_2$, and $\psi'_i:X_i\to Y'$ be the $\RV$-partition of $X_i$ deduced from $\psi'$, which refines $\psi_i$. For $y\in Y'$, let $f_y:X_{1,y}\to X_{2,y}$ be the isomorphism of $\VFR_{\cT(y)}[1,l]$ deduced from $f$.

Since $f$ commutes to the projection to $\VF$, we may see $\Gamma_f$ as a subset of $\VF\times\RV^{2l}\times Y_1\times Y_2\times \RV^{m_1+m_2}$, of which $\psi'$ is an $\RV$-partition. Also take $\chi':\VF\times Y'\to \RV^N$ such that $(\psi',\chi')$ is a cell decomposition adapted to the graph of $f$. For $y\in Y'$ and $\xi\in\RV^N$, the restriction of $f_y$ to $\chi'^{-1}_{y}(\xi)\subset\VF$ defines a $(y,\xi)$-definable bijection $X_{1,a}\to X_{2,a}$ for each $a\in \chi'^{-1}_y(\xi)$, which does not depend on $a$. In particular, $(\psi'_i,\chi')$ is a cell decomposition adapted to $X_i$, and $I_{\psi'_1,\chi'}X_1\simeq I_{\psi'_2,\chi'}X_2$. The only thing left to prove is that $[I^1_{\psi'_1,\chi'}X_1] = [I^1_{\psi_1,\chi_1}X_1$ modulo $\Isp^{\cT'}$.

Let $\xi\in\RV^l$, so that the fibre $X_\xi$ is an object of $\VFR_{\cT'(\xi)}[1,0]$. We have projections $I^1_{\psi_1,\chi_1}X, I^1_{\psi'_1,\chi'}X\to \RV^l$ such that the fibre at $\xi$ is respectively $I^1_{\psi_1,\chi_1}X_\xi, I^1_{\psi'_1,\chi'}X_\xi$. By Lemma \ref{lem_Isp_ponctuel}, it is enough to deal with the case of $X_\xi$, \emph{i.e.} $l=0$. But this is a straightforward consequence of Lemma \ref{lem_int1_ind_chi}.
\end{proof}

The following lemma is the analogue for our case of \cite{stout_integration_2025}, Lemma 5.5.12, and is proved in the same way:

\begin{lem}\label{lem_int_VFR_RV}
Let $X\in\Ob\VFR_{\cT'}[k,l]$, and let $f:X\to\RV^m$ be an $\emptyset$-definable map. Then, for all $1\leq j\leq k$,
\[I^j[X]=\sum_{\xi\in\RV^m}I^j[X_\xi]\,.\]
\end{lem}
\begin{proof}
By compactness, since everything is relative to the coordinates of $\VF^n$ other than $x_j$, we assume that $k=1$. Let $\widetilde\psi:\Gamma_f\to Y$ be an $\RV$-partition of the graph of $f$, and denote by $\psi:X\to Y$ the induced map. Let $\chi:\VF\times Y\to \RV^N$ be such that $(\psi,\chi)$ is a cell decomposition adapted to the graph of $f$. Then, for all $\xi\in\RV^m$, the restriction $(\psi_{f^{-1}(\xi)},\chi)$ is a cell decomposition adapted to $X_\xi$. Now we have, by Lemmas \ref{lem_Isp_ponctuel} and \ref{lem_int_VFR},
\[I^1[X] = [I^1_{\psi,\chi}X] = \sum_{\xi\in\RV^N}[I^1_{\chi,\psi}X_\xi] = \sum_{\xi\in\RV^N}I^1[X_\xi]\,.\]
\end{proof}

Again, the following lemma is exactly analogous to \cite{stout_integration_2025}, Lemma 5.5.13:

\begin{lem}\label{lem_int_rv}
Let $X,X'\in\Ob\VFR_{\cT'}[k,l]$ be such that $[X]=[X']$ modulo $\Isp^{\cT'}$. Then for all $1\leq j\leq k$, the equality $I^j[X]=I^j[X']$ holds modulo $\Isp^{\cT'}$.

In other words, the map $I^j$ induces an homomorphism
\[I^j:\Kp\VFR_{\cT'}[k,l]/\Isp^{\cT'}\to\Kp\VFR_{\cT'}[k-1,l+1]/\Isp^{\cT'}\,.\]
\end{lem}
\begin{proof}
By assumption, for every $x\in\VF^k$, we have $([X_x],[X'_x])\in\Isp^{\cT'(x)}$. As in Lemma \ref{lem_Isp_ponctuel}, we find objects $T_{i,j,x}\in\RV_{\cT'(x)}[\leq l], i,j\leq l$ that give equations attesting to this fact. By compactness, we may see $T_{i,j,x}$ as the fibre at $x$ of an object $T_{i,j}$ of $\VFR_{\cT'}[k,l]$, for all $i,j\leq k$. Moreover, we have injective maps $f_{i,j}:T_{i,j}\times\RV^i_{<1}\times \{0\}^j\to X_1, g_{i,j}:T_{i,j}\times\RV^j_{<1}\times\{0\}^j\to X_2$ which are morphisms in $\VFR_{\cT'}[k,l]$ and whose images define a partition of $X_1,X_2$ respectively. In particular, we may replace $X_1$ by $\bigsqcup_{i,j}T_{i,j}\times\RV^i_{<1}\times\{0\}^j$ and $X_2$ by $\bigsqcup_{i,j}T_{i,j}\times\RV^j_{<1}\times\{0\}^i$. Let $(\psi,\chi)$ be a cell decomposition adapted to $T=\bigsqcup_{i,j} T_{i,j}$, and let us denote by $Y$ the codomain of $\psi$, $\RV^N$ that of $\chi$. The composition of $\psi$ with the canonical maps $X_1,X_2\to T$ then gives $\RV$-partitions of $X_1,X_2$, which we denote by $\psi_1,\psi_2$ respectively, and $(\psi_1,\chi)$ is a cell decomposition adapted to $X_1$ while $(\psi_2,\chi)$ is a cell decomposition adapted to $X_2$. Take some $y\in Y$ and $\xi\in\RV^N$; over $R=\chi_y^{-1}(\xi)$, the sets $T_{i,j,y}$ are constant relative to $\VF$, hence of the form $R\times T'_{i,j,y,\xi}$. Consequently, the sets $X_{i,y}$ over $R$ are themselves of the form $R\times X'_{i,y,\xi}$, and moreover $[X'_{1,y,\xi}]=[X'_{2,y,\xi}]$ in $\Kp\RV_{\cT'(y,\xi)}[\leq l]/\Isp^{\cT'(y,\xi)}$. Finally, $I^j_{\psi_i,\chi}(R\times X'_{i,y,\xi}) = (I^j_\chi R)\times X'_{i,y,\xi}$ by construction, so $I^j[R\times X'_{1,y,\xi}]=I^j[R\times X'_{2,y,\xi}]$ in $\Kp\VFR_{\cT'(y,\xi)}[k-1,l+1]/\Isp^{\cT'(y,\xi)}$. We conclude using Lemma \ref{lem_int_VFR_RV}.
\end{proof}

So we may define $I=I^1\circ\dots I^d:\Kp\VFR_{\cT'}[d,0]\to\Kp\VFR_{\cT'}[0,d]/\Isp = \Kp\RV_{\cT'}[\leq d]$. The next lemma states that the integration morphism defined in this way is well defined. Note that we do not need all the intermediary steps of \cite{stout_integration_2025}, we are able to skip some of them (Lemmas 5.5.3 and 5.5.14) and to reuse the final result.

\begin{lem}
Let $X_1,X_2\in\Ob\VF_{\cT'}[\leq d]$ be such that $[X_1]=[X_2]$ in $\Kp\VF_{\cT'}[d]$. Then $I[X_1]=I[X_2]$ in $\Kp\RV_{\cT'}[\leq d]/\Isp^{\cT'}$.
\end{lem}
\begin{proof}
Let $f:X_1\to X_2$ be an isomorphism, and let $\chi:\Gamma_f\to Y$ be an $\RV$-partition of the graph of $f$. For $y\in Y$, denote by $f_y:X_{1,y}\to X_{2,y}$ the $\cL(y)$-definable bijection induced by $f$. By Lemma \ref{lem_int_rv}, since $Y$ maps into $\RV$, it is enough to verify that, for all $y\in Y$, the equality $I[X_y]=I[Y_y]$ holds in $\Kp\RV_{\cT(y)}[\leq d]$. But this is a consequence of \cite{stout_integration_2025}, Lemma 5.5.15.
\end{proof}

From that we immediatly deduce the main theorem, thanks to the (elementary) fact that $I[\fL U]=[U]$.

\begin{thm}\label{thm_int_sans}
Assume that $\cT'$ has $\RV$-partitions relative to $\cT$, that it is refined, and has the bijection lifting property.

For all $d\in\bN$, there is an isomorphism of Grothendieck semigroups
\[\int:\Kp\VF_{\cT'}[d]\to\Kp\RV_{\cT'}[\leq d]/\Isp\,,\]
as well as a semirings isomorphism
\[\int:\Kp\VF_{\cT'}\to\Kp\RV_{\cT'}[*]/\Isp\,,\]
given by \[\int[X]=[R]\iff [X]=[\fL R]\,.\]
\end{thm}
\begin{rmk}
Actually all these morphisms $\int$ are well defined even without the bijection lifting property, thanks to the Stout-Vermeulen construction to compute the kernel of $\fL$ -- which we just realised. They are given by $\int = I$. Only, in that case, there is no reason for the motivic integral to be injective. Surjectivity holds, on the contrary, as $\int [\fL R] = [R]$ for all $R\in\Ob\RV_{\cT'}[*]$.
\end{rmk}

\subsection{Motivic integration with volume forms}\label{ssec_intmot_vol}
Now let us deal with the case of motivic integration with volume forms, the main result being Theorem \ref{thm_int_avec}. We still assume that $\cT'$ has $\RV$-partitions and is refined, and now also that $\cT'$ is $\RV$-thin, in order to have a dimension theory.

\begin{dfn}
Let $X\subset \VF^d\times Y\times\RV^m$ be an $\emptyset$-definable set, where $Y$ maps into $\RV$. We say that a property holds for almost every $x\in X$ if it holds for every $x$ away from an $\emptyset$-definable subset $Z\subset X$ of dimension strictly smaller than $d$.

Let $X_i\subset\VF^d\times Y_i\times \RV^{m_i},i=1,2$ be $\emptyset$-definable sets. An $\emptyset$-definable \emph{essential bijection} $f:X_1\to X_2$ is an $\emptyset$-definable bijection $f:\widetilde X_1\to \widetilde X_2$, where $\widetilde X_i\subset X_i, \dim_\VF(X_i\setminus \widetilde X_i)<d$.

Two essential bijections $f,g:X_1\to X_2$ are \emph{equivalent} if there is a subset $\widetilde X_1$ contained in the intersection of the domains of $f$ and $g$ such that $\dim_\VF(X_1\setminus\widetilde X_1)<d$ and $f_{|\widetilde X_1}=g_{\widetilde X_1}$. We always consider essential bijections up to equivalence.
\end{dfn}
Sometimes we will say that $f$ is an essential bijection in dimension $d$.

\begin{dfn}
Let $M\models\cT'$ be a model, $A\subset\VF(M)$ be a set of parameters and $R,S\in\Ob\RV_{\cT'(A)}[d]$, and let $h:R\to S$ be an $A$-definable bijection, which we write as $h=(h_0,h_1,h_2)$. A \emph{jacobian lifting} of $h$ is an $A$-definable map $f:\fL R\to\fL S$ such that:
\begin{enumerate}
\item $f$ is bijective;
\item for each $(\xi,\zeta_1,\zeta_2)\in R$, $f$ induces a $\cC^1$ map
\[f_{\zeta_1,\zeta_2}:\rv^{-1}(\xi)\to\rv^{-1}(h_0(\xi,\zeta_1,\zeta_2))\,;\]
\item $\rv(\Jac f_{\zeta_1,\zeta_2})$ is constant on $\rv^{-1}(\xi)$ for all $(\xi,\zeta_1,\zeta_2)\in R$;
\item for $(\xi,\zeta_1,\zeta_2)\in R$ and $x\in\rv^{-1}(\xi)$, we have
\[|\Jac f_{\zeta_1,\zeta_2}(x)|\cdot |\xi|=|h_0(\xi,\zeta_1,\zeta_2)|\,.\]
\end{enumerate}
The map $\Jac_\RV h:R\to\RV^\times$ deduced from $\rv(\Jac f)$ is an \emph{$\RV$-jacobian} of $h$.

We say that $\cT'$ has the jacobian bijection lifting property if, for every model $M\models\cT'$, every set of parameters $A\subset\VF(M)$ and every $A$-definable bijection $h:R\to S$ between two objects of $\RV_{\cT'(A)}[*]$, there is a map $f:\fL R\to \fL S$ which is an $A$-definable jacobian lifting of $h$.
\end{dfn}
\begin{rmk}
Certainly the jacobian bijection lifting property implies the bijection lifting propery. In all the examples we will give, the first one will hold.
\end{rmk}

Now we prove that the differentiability condition is actually quite weak:

\begin{lem}\label{lem_differentiable}
Let $d,k\in\bN$ be two integers. Let $f:X\subset\VF^d\times\RV^m\to X'\subset\VF^{d'}\times\RV^{m'}$ be a definable map with open domain $X$, and let $U$ be the set of $(x,r)\in X$ such that there is $B\subset\VF^d$ with:
\begin{enumerate}
\item $B$ is a nonempty open ball containing $x$;
\item $B\times\{r\}\subset X$ and $f(B\times\{r\})\subset \VF^{d'}\times\{s\}$ for some $s\in\RV^{m'}$;
\item the map $y\in B\mapsto \mathrm{pr}_\VF f(y,r)$ is of class $\cC^k$, where $\mathrm{pr}_\VF$ is the projection $X'\to\VF^d$.
\end{enumerate}
Then $U$ is a dense open definable subset of $X$ and $\dim_{\VF}X\setminus U<d$.
\end{lem}
\begin{proof}
Let $\chi:\Gamma_f\to Y$ be an $\RV$-partition of the graph of $f$. For all $y\in Y$, let $X_y$ be the interior of $\pi(\chi^{-1}(y))$, with $\pi:\Gamma_f\to \VF^d$ the projection map, so that $\dim_\VF \VF^d\setminus (\bigcup_{y\in Y} X_y) < d$. By \cite{cluckers_hensel_2022}, Theorem 5.1.5, the open subset $U_y\subset X_y$ of points near which $f$ has class $\cC^k$ is dense in $X_y$, so that $\dim_\VF (X_y\setminus U_y)<d$. The result follows. 
\end{proof}
\begin{dfn}
With the notations from Lemma \ref{lem_differentiable}, if $X=U$, we say that $f$ has class $\cC^k$.
\end{dfn}
By Lemmae \ref{lem_differentiable}, if $f:X_1\to X_2$ is a definable map with $X_i\subset\VF^d\times\RV^{m_i}$ of dimension $d$, we may always define its jacobian $\Jac(f):X_1\to \VF$, which is well-defined up to a $(d-1)$-dimensional set.

\begin{dfn}
\begin{enumerate}
\item Let $\mu\VF_{\cT'}[d]$ (resp. $\mu_\Gamma\VF_{\cT'}[d]$) be the category of pairs $(X,\omega)$ where $X\subset\VF^d\times Y\times\RV^m$ is an object of $\VF_{\cT'}[d]$ and $\omega:X\to\RV^\times$ (resp. $\omega:X\to\Gamma^\times$) is an $\emptyset$-definable map.

Morphisms $(X,\omega)\to (X',\omega')$ in this category are essential $\emptyset$-definable bijections $f:X\to X'$ such that, for almost all $x\in X$, we have
\[\omega(x) = \omega'(f(x))\rv(\Jac f(x)) \text{ (resp. }\omega(x) = \omega'(f(x))|\Jac f(x)|\text{ )}\,.\]

\item Let $\mu\RV_{\cT'}[d]$ be the category of pairs $(R,\omega)$ where $R\subset\RV^d\times Y\times\RV^m$ is an object of $\RV_{\cT'}[d]$ and where $\omega:R\to\RV^\times$ is an $\emptyset$-definable map.

The morphisms $(R,\omega)\to (S,\rho)$ in this category are the $\emptyset$-definable bijections $h:R\to S$ such that there is an $\RV$-jacobian $\Jac_\RV h$ of $h$ with
\[\omega(r) = \rho(h(r))\Jac_\RV h(r)\]
for every $r\in R$.

\item Let $\mu_\Gamma\RV_{\cT'}[d]$ be the category of pairs $(R,\omega)$ where $R\subset\RV^d\times Y\times\RV^m$ is an object of $\RV_{\cT'}[d]$ and where $\omega:R\to\Gamma$ is an $\emptyset$-definable map.

The morphisms $(R,\omega)\to (S,\rho)$ in this category are the $\emptyset$-definable bijections $h:R\to S$ such that
\[\omega(\xi)\cdot |\xi| = \rho(h(\xi))\cdot |h(\xi)|\]
for every $\xi\in R$, where if $\xi = (\xi_1,\dots,\xi_d,\zeta,\psi)$, then $|\xi| = \prod_{i=1}^d|\xi_i|$.
\end{enumerate}
If $(R,\omega)$ is an object of $\mu\RV_{\cT'}[d]$ (resp. of $\mu_\Gamma\RV_{\cT'}[d]$), we denote by $\fL(R,\omega)$ its lift, which is an object of $\mu\VF_{\cT'}[d]$ (resp. of $\mu_\Gamma\VF_{\cT'}[d]$).

We also denote by $\mu\VF_{\cT'}[*],\mu_\Gamma\VF_{\cT'}[*],\mu\RV_{\cT'}[*],\mu_\Gamma\RV_{\cT'}[*]$ the direct sums of these categories.
\end{dfn}
Unlike the case without volume forms, we now take direct sums over $\VF$, because we make a distinction between the different dimensions.

\begin{lem}\label{lem_relevement_mesure}
For all $d\in\bN\cup\{*\}$, the previous construction induces a morphism of semigroups (and even of semirings if $d=*$)
\[\fL:\Kp\mu\RV_{\cT'}[d]\to\Kp\mu\VF_{\cT'}[d]\,.\]
If $\cT'$ has the jacobian bijection lifting property, it also induces a morphism
\[\fL:\Kp\mu_\Gamma\RV_{\cT'}[d]\to\Kp\mu_\Gamma\VF_{\cT'}[d]\,.\]
\end{lem}
\begin{proof}
This follows at once from the definitions.
\end{proof}

Now we compute its image and its kernel. Actually, we will even define a converse (without any hypothesis on the jacobian bijection lifting property), as in the case without volume forms.

\begin{lem}
The morphisms of Lemma \ref{lem_relevement_mesure} are onto (if they are defined).

More precisely, let $(X,\omega)$ be an object of $\mu\VF_{\cT'}[d]$  (resp. $\mu_\Gamma\VF_{\cT'}[d]$), with $X\subset \VF^d\times Y\times\RV^m$, and let $\chi:\VF^d\times Y\to \RV^N$ be a cell decomposition adapted to $(X,\omega)$ and which only comprises dimension $d$ cells. Then there is $(R,\rho)$ an object of $\mu\RV_{\cT'}[d]$ (resp. $\mu_\Gamma\RV_{\cT'}[d]$) and a definable bijection $f:X\to \fL R$ which is also an isomorphism $f:(X,\omega)\to\fL(R,\rho)$ and descends to a bijection
\[\{h:(\chi(x,y),y,\xi):(x,y,\xi)\in X\}\to R\,,\]
preserving volume forms.
\end{lem}
\begin{proof}
The second part implies the first one: take a cell decomposition adapted to $(X,\omega)$, then restrict to $X'$ the union of dimension $d$ cells. The complementary subset $X\setminus X'$ has dimension strictly less than $d$, so the inclusion $(X',\omega)\subset (X,\omega)$ is an isomorphism in our category.

For the second point, proceed just as in Lemma \ref{lem_fL_surj}, noticing after \cite{stout_integration_2025} that the construction in this lemma preserves volume forms.
\end{proof}

Write once more $I_\chi X' = \{(\chi(x,y),y,\xi):(x,y,\xi)\in X'\}$ and $\omega:I_\chi X'\to \RV^\times$ (resp. $\omega:I_\chi X'\to \Gamma$) well defined, with $X'$ as in the proof of the lemma. The class $[I_\chi X', \chi]\in\Kp\RV_{\cT'}[d]$ is well defined.

To compute the kernel, we proceed as in the case without volume forms: we again have to do what Stout and Vermeulen do \cite{stout_integration_2025}.

\begin{dfn}
Let $\Isp^{\cT',\mu}$ be the congruence on $\Kp\mu\RV_{\cT'}[*]$ ou $\Kp\mu_\Gamma\RV_{\cT'}[*]$ generated by $([1]_1,[\RV_{<1}^\times]_1)$. Denote by $\Isp^{\cT',\mu}[d]$ (or $I^{\cT',\mu}_\sp$ if it is clear from the context) the semigroup congruence it induces on $\Kp\mu\RV_{\cT'}[d]$ ou $\Kp\mu_\Gamma\RV_{\cT'}[d]$.
\end{dfn}

We first deal with dimension $1$:

\begin{lem}
The kernel of the morphism $\fL:\Kp\mu\RV_{\cT'}[1]\to\Kp\mu\VF_{\cT'}[1]$ is the congruence $I^\mu_\sp[1]$. The same assertion holds for $\mu_\Gamma$ if the jacobian bijection lifting property is satisfied.

More precisely, the map $[X,\omega]\in\Kp\mu\VF_{\cT'}[1]\mapsto [I_\chi X,\omega]\in\Kp\mu\RV_{\cT'}[1]$ is well defined and depends neither on the choice of the cell decomposition $\chi$, nor on that of the representative $(X,\omega)$. The same assertion holds for $\mu_\Gamma$.
\end{lem}
\begin{proof}
As explained in \cite{stout_integration_2025}, Lemma 6.3.3, the proof is essentially the same as in the case without a volume forms, the only difference being that we now have to choose cell decompositions so that they are adapted to volume forms.
\end{proof}

We continue on the same way. We do not give details, since it is enough to do the same as before, with only some small attention to the subtlety pointed out in \cite{stout_integration_2025}, Remark 6.3.7.

We eventually get:
\begin{thm}\label{thm_int_avec}
There is an isomorphism of graded semirings
\[\int:\Kp\mu\VF_{\cT'}[*]\to\Kp\mu\RV_{\cT'}[*]\,,\]
such that $\int(\fL (R,\omega)) = [R,\omega]$ for every object $(R,\omega)$ in $\mu\RV_{\cT'}[d]$.

There is also a morphism of graded semirings
\[\int:\Kp\mu_\Gamma\VF_{\cT'}[*]\to\Kp\mu_\Gamma\RV_{\cT'}[*]\,,\]
such that $\int(\fL(R,\omega))=(R,\omega)$ for every object $(R,\omega)$ in $\mu_\Gamma\RV_{\cT'}[d]$, and which is an isomorphism if $\cT'$ satisfies the jacobian bijection lifting property.
\end{thm}

\subsection{Lifting bijections}
Here we give examples of some situations in which the (jacobian) bijection lifting property holds. Let us start with a first case, which is more or less trivial:

\subsubsection{Section of \texorpdfstring{$\RV^\times$}{RV}}

\begin{prop}\label{prop_rel_bij_RV}
Assume there is an $\emptyset$-definable map $\sn:\RV^\times\to\VF$ such that $\rv\circ\sn=\id_\RV$. Then $\cT'$ has the jacobian bijection lifting property.
\end{prop}
\begin{proof}
Let $R,S\in\Ob\RV_{\cT'(A)}[d]$ be two objects for $A$ some set of parameters in $\VF$, and let $h:R\to S$ be an $A$-definable bijection. We define a lifting $f:\fL R\to \fL S$ of $h$ in the following manner:
\[\begin{array}{rrcl}
f : & \fL R & \to & \fL S \\
& (x,y,\xi) & \mapsto & (x\cdot \sn(\rv(x)^{-1}h(\rv(x))), y, \xi)\,.
\end{array}\]
This clearly proves the result.
\end{proof}
\begin{rmk}
We did not use that the parameter set $A$ was in $\VF$; here, $A$ could have been anywhere inside $\VF^{\mathrm{eq}}$.
\end{rmk}

\subsubsection{Effective theory}

We now turn to a second case, inspired in Stout and Vermeulen's work \cite{stout_integration_2025}. We will need the following definition:
\begin{dfn}
Say that $\cT'$ is \emph{$\RV$-minimal} (relative to $\cT$) if, for every model $M\models\cT'$, every set of parameters $A\subset \VF(M)$ such that $A=\dcl_{\cL'}(A)\cap\VF$, and every $\cL'(A)$-definable $X\subset\RV^r$, then $X$ is already $\cL(A)$-definable.
\end{dfn}
In other words, besides the obvious effect on parameters, going from $\cL$ to $\cL'$ has no effet on the structure of $\RV$. Then:
\begin{prop}\label{prop_rel_bij_eff}
Assume the theory $\cT$ to be effective (in the sense of \cite{stout_integration_2025}) and the theory $\cT'$ to be $\RV$-minimal relative to $\cT$. Then $\cT'$ has the jacobian bijection lifting property.
\end{prop}
\begin{proof}
Let $R,S\in\Ob\RV_{\cT'(A)}[d]$ and $h:R\to S$ be an $A$-definable bijection, with $A$ a set of parameters in $\VF$. Write $R\subset(\RV^\times)^d\times Y_1\times\RV^{m_1}, S\subset(\RV^\times)^d\times Y_2\times\RV^{m_2}$. For $u\in Y_1,v\in Y_2$, set $R_u,S_v$ to be the respective fibres of $R$ and $S$ over $u$ and $v$, and $R_{u,v}=R_u\cap h^{-1}(S_v), S_{u,v}=h(R_u)\cap S_v$, as well as $h_{u,v}:R_{u,v}\to S_{u,v}$ the restriction of $h$. By doing so, $R_{u,v}, S_{u,v}$ and $h_{u,v}$ are $\cL'(A,u,v)$-definable, hence $\cL(w)$-definable for some $w\in\dcl_{\cL'}(A,u,v)\cap\VF$ by $\RV$-minimality. But the theory $\cT(w)$ is 1-h-minimal and effective, because $\cT$ is and $w$ is in $\VF$. We may use effectivity to find a jacobian lifting $f_{u,v}:\fL R_{u,v}\to \fL S_{u,v}$ of $h_{u,v}$, which is $\cL(w)$-definable -- note that the projection of $R_{u,v}, S_{u,v}$ on $\RV^d$ has finite fibres. Then $f_{u,v}$ is in particular $\cL'(A,u,v)$-definable, and we can glue by compactness.
\end{proof}

\subsubsection{Section of \texorpdfstring{$\Kb$}{K}}

The last case we are going to deal with is slightly more complex, it is some kind of mix of the first two properties.
\begin{prop}\label{prop_rel_bij_Kb}
Assume there is an $\emptyset$-definable section $\sn:\Kb\to\VF$ and we have on $\RV^\times$ an angular component $\ac:\RV^\times\to\Kb^\times$, such that:
\begin{enumerate}
\item $\rv\circ\sn = \id_{\Kb}$;
\item $\ac$ is a group morphism and its restriction to $\Kb^\times$ is the identity;
\item\label{item_prop_rel_bij_Kb_3} the definable subsets $\Gamma$ and $\Kb^\times$ are orthogonal, in the sense that for every model $M\models\cT'$, every set of parameters $A$ and every $A$-definable subset $X\subset\Kb^n\times\Gamma^m$, the set $X$ is a finite union of $A$-definable sets of the form $Y\times Z$, with $Y\subset\Kb^n$ and $Z\subset\Gamma^m$;
\item the induced structure on $\Gamma$ is that of an orderd abelian group, that is for every $M\models\cT'$, every set of parameters $A$ such that $A=\dcl_{\cL'}(A)$ and every $A$-definable subset $X\subset\Gamma^m$, the set $X$ is already definable in the language of ordered abelian groups with parameters in $A\cap\Gamma$.
\item the theory $\cT'$ is $\Gamma$-effective, in the sense that for every $M\models\cT'$ and every set of parameters $A\subset\VF(M)$, for every $\gamma\in\dcl_{\cL'}(A)\cap\Gamma$, there is $x\in\dcl_{\cL'}(A)\cap\VF$ such that $|x|=\gamma$.
\end{enumerate}
Then the theory $\cT'$ has the jacobian bijection lifting property.
\end{prop}
\begin{rmk}
We used the fact that, in presence of an angular component, we may see $\Gamma$ as a subset of $\RV^\times$, which is the set of points with angular component $1$. This is necessary for the statement of Item \ref{item_prop_rel_bij_Kb_3} to make sense.
\end{rmk}
\begin{proof}
Let $h:R\to S$ be an isomorphism in $\RV_{\cT'(A)}[d]$, with $A$ some parameters in $\VF$. Let us write $R\subset (\RV^\times)^d\times\RV^{m_1}\times Y_1, S\subset (\RV^\times)^d\times\RV^{m_2}\times Y_2$, where the projection maps $R\to(\RV^\times)^d\times Y_1, S\to(\RV^\times)^d\times Y_2$ have finite fibres and $Y_i\subset\VF^{l_i}$ maps into $\RV$. First we reduce to the case $Y_1=Y_2$ with $h$ relative to these coordinates. To do so, write $h=(h^\RV, h^Y)$, then set $\eta_R : R\to (\RV^\times)^d\times\RV^{m_1}\times Y_1\times Y_2, (x,y)\mapsto (x,y,h^Y(x,y))$, whose image we denote by $\widetilde R$. Define analogously $\eta_S:S\to \widetilde S$. Then set $\widetilde h:\widetilde R\to\widetilde S, (x,y_1,y_2)\mapsto (h^{\RV}(x,y_1),y_1,y_2)$. We have $h=\eta_S^{-1}\circ\widetilde h\circ\eta_R$, and it is straightforward to lift $\eta_R,\eta_S$. We have reduced to the desired case. Now, by compactness, it is enough to prove the jacobian bijection lifting property in every fibre of the projection to $Y_1=Y_2$. Note that here we did not assume that the projection to $Y_i$ was an $\RV$-partition.

Assume that $R\subset(\RV^\times)^d\times\RV^{m_1}, S\subset(\RV^\times)^d\times\RV^{m_2}$. The graph of $h$ is an $\emptyset$-definable subset $\Gamma_h\subset (\RV^\times)^{2d}\times\RV^{m_1+m_2}$. Without loss, assume that $\Gamma_h\subset(\RV^\times)^{2d+m_1+m_2}$. By orthogonality, we find $A$-definable partitions $R=\bigsqcup_{i=1}^N R_i^{\Kb}R_i^\Gamma, S=\bigsqcup_{i=1}^N S_i^{\Kb}S_i^{\Gamma}$, with $R_i^{\Kb}\subset(\Kb^\times)^{d+m_1}, R_i^\Gamma\subset\Gamma^{d+m_1}$, and similarly for $S_i^{\Kb}, S_i^\Gamma$. We also get $A$-definable bijections $h_i^{\Kb}:R_i^{\Kb}\to S_i^{\Kb}$, $h_i^\Gamma:R_i^\Gamma\to S_i^\Gamma$ such that, if $yz\in R_i^{\Kb}R_i^\Gamma$, then $h(yz) = h_i^{\Kb}(y)h_i^\Gamma(z)$. We restrict ourselves without loss to the case $N=1$, and stop writing the indices $i$. It is enough to deal with the two cases $h^{\Gamma}=\mathrm{id}$ and $h^{\Kb}=\mathrm{id}$ separately. We keep denoting by $x=yz$ the elements of $R$, with $y\in R^{\Kb}$ and $z\in R^\Gamma$; in other words, $y=\ac(x)$ and $z=|x|=xy^{-1}$.

If $h^{\Gamma}=\mathrm{id}$, just take

\[\begin{array}{rrcl}
f:&\fL R &\to & \fL S\\
& x & \mapsto & x\cdot \sn\left(\frac{h^{\Kb}\circ\ac\circ\rv(x)}{\ac\circ\rv(x)}\right)
\end{array}\,.\]

If $h^{\Kb}=\mathrm{id}$, then after a finite partition $h^\Gamma$ is given by linear combinations with coefficients in $\bQ$ and translations by definable valuation, according to Cluckers et Halupczok \cite{cluckers_quantifier_2011}, Corollary 1.10. Therefore it suffices to treat those cases separately.

If $\gamma$ is a definable valuation, then by assumption there is $x\in\dcl_{\cL'}(A)$ with valuation $\gamma$. Now lift the translation by $\gamma$ to multiplication by $x$.

Now we deal with the case of linear combinations with coefficients in $\bQ$. To simplify again, we limit ourselves to the cases of linear combinations with integer coefficients and homotheties of rational factor, which is enough. We start with $h^\Gamma$ a linear combination with integer coefficients. We lift by taking
\[x\mapsto \left (\frac{x}{\sn\circ\ac\circ\rv (x)}\right )^{k-1}\cdot x\,,\]
where we used multi-index notations.

Finally, when $h^\Gamma$ is a homothety of rational factor $r=\frac pq$, let $x\in\VF^\times$ have angular component $1$, \emph{i.e.} $\ac\circ\rv(x)=1$. Then there is a unique $y\in\VF^\times$ with angular component $1$ such that $x^p=y^q$; denote it by $y=x^r$. Set
\[x\mapsto \left(\frac{x}{\sn\circ\ac\circ\rv(x)}\right)^r\cdot x\,\]
which is clearly $\emptyset$-definable and lifts $h$.

For the jacobian version, just note that this construction does provide a jacobian lift.
\end{proof}

\begin{rmk}
Later on we will deal with a last case, very similar to this one but rather assuming that we have a section of $\Gamma$ and that $\Kb$ is algebraically bounded (so that we may lift bijections by effectivity).
\end{rmk}

\section{Example : henselian valued fields}\label{sec_henselien}
Now we show that the above results indeed holds in the case of residue characteristic zero henselian valued fields with a section of the residue field, of $\RV$ (Corollary \ref{cor_int_hens_1}) or of the value groupe (Corollary \ref{cor_int_hens_2}), with a tameness hypothesis on $\Kb$ in that last case.

\subsection{Section of \texorpdfstring{$\RV$}{RV} or the residue field}\label{ssec_hens_RV}
Let $\cL$ be a language obtained from $\cL_{val}$ by possibly adding structure on $\RV$ and let $\cT$ be an $\cL$-theory of henselian valued fields with characteristic zero residue field. In particular, $\cL$ may be $\cL_{val}$ or $\cL_{\mathrm{DP}}$ the Denef-Pas language. Let $\Gamma_0\subset\Gamma$ be a definable subgroup of $\Gamma$, and set $R$ to be the set of points of $\RV^\times$ with valuation in $\Gamma_0$. Let $\cL'$ be the language obtained from $\cL$ by adding a (partial) map $\sn:R\to\VF$, and let $\cT'$ be the theory asserting that $\sn$ is a section from $R$ to the valued field $\VF$, \emph{i.e.} a group morphism (for multiplication), so that moreover if $x,y\in R$ have the same valuation and $y\neq -x$, we have $\sn(x+y)=\sn(x)+\sn(y)$. We denote by $R_\VF$ the image of $R$ in $\VF$ via $\sn$.

\begin{rmk}
This level of generality allows for simultaneous proofs for the two particular cases most likely interesting in concrete situations: the case $\Gamma_0=1$, which corresponds to adding a section for the residue field, and the case $\Gamma_0=\Gamma$, which corresponds to adding a section of all $\RV^\times$.
\end{rmk}

We will prove that the axioms of Section \ref{sec_axiomes} do hold in that setting, in Corollary \ref{cor_section} for most of them and in Proposition \ref{prop_theorie_raffinee} for the refined character.

\begin{prop}\label{prop_section}
Let $M,N\models\cT'$ be two models such that $N$ is $|M|^+$-saturated, and let $A\subset M$ be a sub-$\cL'$-structure. Let $f:A\to N$ be a partial $\cL'$-embedding, which we additionally assume to be $\cL$-elementary. Then $f$ is a partial $\cL'$-elementary embedding. Actually, for any extension of $f$ to an $\cL$-elementary embedding  $g:A\cup\RV(M)\to N$, we may extend $g$ to the whole of $M$ in an $\cL'$-elementary way.
\end{prop}
\begin{proof}
The existence of $g$ follows immediatly from the fact that $f$ is $\cL$-elementary and saturation of $N$, and clearly the second part implies the first one. So let us take such a map $g$. We may extend $g$ to $\mathrm{Frac} \VF(A)$; this is again an $\cL$-elementary embedding. In the following we replace $\VF(A)$ by $\mathrm{Frac}\VF(A)$. We denote again by $g$ the map from $A\cup\RV(M)\cup R_\VF(M)$ obtained by setting $g(\sn(x))=\sn(g(x))$ for all $x\in R(M)$; there is no conflict between those notations.

\begin{fact}
This map $g$ can be extended to an $\cL'$-partial embedding on the generated substructure.
\end{fact}
\begin{proof}
The map $g$ is already defined on all $\RV(M)$, its compatibility with $\sn$ is automatic and the sort $\VF$ is only endowed with the language of rings and $\rv$. Thus it only remains to check that, for every polynomial expression $P(a_1,\dots,a_r,x_1,\dots,x_s)$ with integer coefficients and variables inside $A\cup R_\VF(M)$, it holds that
\[\rv(P(g(a_1),\dots,g(a_r),g(x_1),\dots,g(x_s))) = g(\rv(P(a_1,\dots,a_r,x_1,x_s)))\,.\]
Since $A$ and $R_\VF$ are stable under product and we already know that $g$ is compatible to it, we see by developing the expression that we are only left with linear combinations of elements of $R_\VF(M)$ with coefficients in $A$. So let $a_1,\dots,a_r\in A, x_1,\dots,x_r\in R_\VF(M)$ be elements. Without loss, assume $|a_1x_1|=\max_i|a_ix_i|$. We may replace each $a_i$ by $a_ia_1^{-1}$ and each $x_i$ by $x_ix_1^{-1}$; so it is enough to show the equality when all the $a_ix_i$ are integers, and $a_1x_1$ is an invertible integer. Let $I$ be the set of indices $1\leq i\leq r$ such that $|a_ix_i|=1$. For $i\in I$, we have $|a_i|=|x_i|^{-1}\in\Gamma_0$; if we set $b_i=\sn(\rv(a_i))$, then $\sum_{i=1}^r a_ix_i=\sum_{i\in I}b_ix_i$ modulo the maximal ideal. Moreover, $I$ is also the set of $i$'s such that $|g(a_i)g(x_i)|=1$, and since $f$ was compatible to $\sn$, we find that $g(b_i)=\sn(\rv(g(a_i)))$. Consequently, it also holds that $\sum_{i=1}^rg(a_i)g(x_i)=\sum_{i\in I}g(b_i)g(x_i)$ modulo the maximal ideal.

If $\sum_{i\in I}b_ix_i\neq 0$, then as $b_i\in R_\VF(M)$, by compatibility of $\sn$ to the sum, we actually obtain the following equality:
\[g(\rv(\sum_{i=1}^r a_ix_i))=g(\rv(\sum_{i\in I}b_ix_i))=\rv(g(\sum_{i\in I}b_ix_i))=\rv(\sum_{i\in I}g(b_i)g(x_i))=\rv(\sum_{i=1}^r g(a_i)g(x_i))\]
where the second equality uses the fact that $g$ is already defined on $\sum_{i\in I}b_ix_i\in R_\VF(M)$.

If not, then $\sum_{i\in I}b_ix_i=0$, so we can write one of the $x_i$'s in terms of the others, with coefficients in $R_\VF(A)$ -- guaranteeing that the equality also holds after applying $g$. Thus we can re-write the quantities we are studying with strictly less terms, and conclude by induction.
\end{proof}

Let us keep writing this extension $g$. It immediatly follows from Ax-Kochen-Ershov principle that $g$ is again $\cL$-elementary, because every $\cL$-formula is equivalent to an $\RV$-formula evaluated in the $\RV$-variables and the image via $\rv$ of some polynomial expressions in the $\VF$-variables and we have already defined $g$ over all $\RV(M)$. So we may extend $g$ to an $\cL$-elementary embedding, which is now in addition an embedding for $\cL'$. This extension is $\cL'$-elementary, by classical arguments: for instance, let $L$ be an elementary extension of  $M$, which is moreover $|N|^+$-saturated. By the preceding argument, there is an $\cL'$-embedding of $N$ into $L$ which extends the embedding of $M$ inside $L$. This last embedding is $\cL'$-elementary, and it is now well-known that in such a situation the embedding of $M$ into $N$ is $\cL'$-elementary.
\end{proof}

\begin{cor}\label{cor_section}
The theory $\cT'$ is complete if $\cT$ is, has quantifier elimination relative to $\cL$, is $\RV$-minimal and has $\RV$-partitions relative to $\cT$, and is $\RV$-thin.
\end{cor}
\begin{proof}
Completeness is an straightforward consequence of the previous proposition. Its statement also clearly implies quantifier elimination relative to $\cL$ (recall that "quantifier elimination" exactly means that every embedding is elementary, and here the "relative to $\cL$" part means that we have to assume that the embedding is already $\cL$-elementary). As for $\RV$-minimality, it corresponds to us being able to choose $g$ freely -- compatibility to $\cL$ on $\RV$ and being $\cL'$-elementary on the parameters is enough to find that the embedding is $\cL'$-elementary on $\RV$.

Now we show that $\cT'$ has $\RV$-partitions. Let $X$ be an $A$-definable set. Since $\cT'$ has quantifier elimination relative to $\cL$, the set $X$ is definable by a boolean combination of $\cL$-formulas and quantifier free $\cL'$-formulas. Then it is enough to show the existence of $\RV$-partitions for subsets defined by a single $\cL$-formula (in which case it is obvious) and those defined by a quantifier free $\cL'$-formula. In that case, we just take the partition mapping the variable to the tuple made of all terms starting with $\sn$, which gives the desired result.

Finally we prove that $\cT'$ is $\RV$-thin. Let $K\models\cT'$ be a model such that $|\VF(K)|=|\Kb(K)| = |\Gamma(K)|=\kappa$ for some cardinal $\kappa$. As a valued field, $K$ embeds into a spherically complete extension $\widetilde K$ by \cite{kaplansky_maximal_1942}. Then $\RV(\widetilde K)=\RV(K)$, with $\kappa$. On the other hand, $|\VF(\widetilde K)|=2^\kappa>\kappa$. In particular, the cardinal of a ball of $\widetilde K$ is also $2^\kappa$, which shows that there can be no surjection from $\RV^m$ to a ball of $\VF$. But $K\prec \widetilde K$ by Ax-Kochen-Ershov theorem, so there cannot be such a definable surjection in $K$ either. 
\end{proof}

\begin{rmk}
In fact, we have shown that, in that language, quantifier elimination relative to $\cL$ implies that the theory has $\RV$-partitions.
\end{rmk}

We put next point apart, since its proof is more technical.

\begin{prop}\label{prop_theorie_raffinee}
The theory $\cT'$ is refined.
\end{prop}
We split the proof into several lemmas.

\begin{lem}\label{lem_raff_commun}
Let $\cP_1,\cP_2$ be two cell partitions of finite length of $\VF$. There exists a cell partition of finite length $\cP$ of $\VF$ which refines $\cP_1$ and $\cP_2$.
\end{lem}
\begin{proof}
This fact is analogous to Lemma \ref{lem_raff_fini}: take the sequence of data $(I_0^{(i)},c^{(i)}$ defining $\cP_2$, then refine $\cP_1$ using these data following the procedure described in the proof of that lemma. Clearly this yields the result.
\end{proof}

The following lemma holds for any 1-h-minimal theory.

\begin{lem}\label{lem_cL_raff}
For any $\cL(\emptyset)$-definable subset $\widetilde X\subset\VF\times\RV^N$, there is a cell partition of finite length $\cP$ of $\VF$ such that, for every cell $B_\sigma$ of $\cP$, the fibre $\widetilde X_x \subset \RV^N$ is  constant when $x\in B_\sigma$.
\end{lem}
\begin{proof}
This is a consequence of 1-h-minimality, with the technique of the proof of \cite{stout_integration_2025}, Lemma 5.4.5. More precisely, let $\chi:\VF\to\RV^m$ be an $\cL(\emptyset)$-definable cell decomposition adapted to $\widetilde X$. Take $\cP_0$ to be the elementary cell partition of $\VF$. For every ball $B_\sigma$ in $\cP_0$, let $D_\sigma\subset B_\sigma$ be the set of the centers of $\chi$ which are inside $B_\sigma$. It must be finite by 1-h-minimality, and if it is nonempty then its barycenter $c_\sigma$ gives a definable point of $B_\sigma$. So we define $I_0$ as the set of $\sigma$ for which $D_\sigma\neq\emptyset$ and $r(\sigma)>0$, with $c:I_0\to\VF, \sigma\mapsto c_\sigma$, and take $\cP_1$ to be the refinement of $\cP_0$ corresponding to these data. Now let $N_0 = \max_{r(\sigma)>0}|D_\sigma|$ and $N_1$ be the analogous quantity for partition $\cP_1$. If $N_0>0$, we must have $N_1<N_0$ by construction. In particular, this process stops after a finite number of steps. But if this stops at $\cP_n$, it means that the centers of $\chi$ are all contained in points of $\cP_n$. As a consequence, $\cP_n$ refines the partition given by $\chi$, thus satisfies the requirements.
\end{proof}

The next result is at the heart of the proof. There lie the syntactical arguments.

\begin{lem}\label{lem_raff_sans_RV}
For every $\cL'$-term $\tau(x):\VF\to\RV$, there is a cell partition of finite length $\cP$ of $\VF$ such that, for any cell $B_\sigma$ of $\cP$, the term $\tau$ is constant on $B_\sigma$.

In particular, for all $\cL'$-formulas without quantifier $\varphi(x)$, there is a cell partition of finite length $\cP$ of $\VF$ such that, for every cell $B_\sigma$ of $\cP$, the formula $\varphi$ is true on all $B_\sigma$ or false on $B_\sigma$.

Consequently, for all $\emptyset$-definable subsets $X\subset\VF$, there is a cell partition of finite length $\cP$ of $\VF$ such that $X$ is a union of cells of $\cP$.
\end{lem}
\begin{proof}
We prove this in a similar fashion to the existence of $\RV$-partitions. Let $\tau_i(x),1\leq i\leq n,$ be the terms with values in $\RV$ occurring in $\tau$. We call \emph{width} of $\tau_i(x)$ the number of occurrences of the function symbol $\sn$ inside $\tau_i(x)$. Without loss, we assume that the $\tau_i$ are orderded by increasing width. Then we proceed as following. First, $\tau_1$ must have width zero, meaning that it is an $\cL$-term. Then apply Lemma \ref{lem_cL_raff} to its graph, which gives a cell partition $\cP_1$ of finite length such that $\tau_1$ is constant on each cell. By compactness, now we can restrict ourselves to a single cell $B_\sigma$, and add the value $\xi_{1,\sigma}$ of $\tau_1$ on this cell to the parameters. Let us substitute to the terms $\tau_i(x), i\geq 2$, the term $\tau_i^{\xi_{1,\sigma}}(x)$ obtained by replacing $\sn(\tau_1(x))$ by $\sn(\xi_{1,\sigma})$ everywhere it appears. At this point the term $\tau_2^{\xi_{1,\sigma}}$ must have zero width, so that we are able to find a cell partition $\cP_{2,\xi_{1,\sigma}}$ of $\VF$ such that $\tau_2^{\xi_{1,\sigma}}$ is constant on the cells of $\cP_{2,\xi_{1,\sigma}}$. Using Lemma \ref{lem_raff_commun}, we even assume that $\cP_{2,\xi_{1,\sigma}}$ refines $\cP_1$, and in particular that it induces a cell partition of finite length of $B_\sigma$. We glue by compactness into a cell partition $\cP_2$ such that $\tau_1,\tau_2$ are both constant on every cell, and iterate the process.

Now let $\varphi(x)$ be a quantifier-free $\cL'$-formula. Assume that it is a boolean combination of atomic formulas, each of which being given by some relation on $\RV$. Let $\tau_i(x)$ be all the terms that occur in those relations, and $\cP$ be a cell partition of finite length adapted to these terms in the sense of the first part of the lemma (which exist, by Lemma \ref{lem_raff_commun}). Then $\cP$ is as desired.

Finally let $X$ be such a set. By relative quantifier elimination, it is a boolean combination of $\cL$-definable and quantifier-free $\cL'$-definable sets. We just take a cell partition of finite length adapted to each of these sets, by Lemmas \ref{lem_raff_commun} and \ref{lem_cL_raff} and the previous result.
\end{proof}

At last we are able to prove the result we are interested in:

\begin{lem}\label{lem_raff_avec_RV}
For every $\emptyset$-definable subset $X\subset\VF\times\RV^m$, there is a cell partition of finite length $\cP$ of $\VF$ such that, for every cell $B_\sigma$, the fibre $X_x\subset\RV^m$ is constant as $x$ varies in $B_\sigma$.
\end{lem}
\begin{proof}
Case $m=0$ corresponds to Lemma \ref{lem_raff_sans_RV}. Let $y\in\RV^m$ be a point; by applying that lemma to the fibre $X_y\subset\VF$, we find a cell partition of finite length $\cP_y$ of $\VF$. Let $(I_0^{(1,y)},c^{(1,y)}), \dots, (I_0^{(n,y)}, c^{(n,y)})$ be the $y$-definable sequence of data defining $\cP_y$. By compactness, we assume that these data are uniformly definable in variablee $y$. Then let $C^{(i)} = \bigsqcup_{y\in\RV^m}\mathrm{im}(c^{(i,y)})$ be the union of the images of $c^{(i,y)}$; it is an $\emptyset$-definable subset of $\VF$ of dimension $0$. By Lemma \ref{lem_raff_sans_RV}, there is a cell partition of finite length $\widetilde\cP_{i+1}$ such that $C^{(i)}$ is a union of cells of $\widetilde\cP_{i+1}$. Since $C^{(i)}$ has dimension $0$, this means that every $\{x\},x\in C^{(i)}$ is a cell of $\cP_{i+1}$; in particular, $\widetilde\cP_{i+1}$ refines $\cP_{i,y}$ for every $y\in\RV^m$, where $\cP_{1,y},\dots,\cP_{n+1,y}=\cP_y$ is the sequence of cell partitions corresponding to the above data. In particular, $\widetilde\cP_{n+1}$ is a cell partition of finite length which refines all the partitions $\cP_y,y\in\RV^m$, so it is as wanted.
\end{proof}

\begin{proof}[Proof of Proposition \ref{prop_theorie_raffinee}]
Let $X\subset\VF$ and $\cP_i=1,2$ be two cell partitions of $X$. We see them as maps $X\to\RV^{N_i}$; by putting them together, we obtain a map $\chi:X\to\RV^N$. Then we just apply Lemma \ref{lem_raff_avec_RV} to the graph of $\chi$.
\end{proof}
\begin{rmk}
Once more, we proved that, in that language, quantifier elimination relative to $\RV$ (together with being $\RV$-thin) yields the fact that the theory is refined.
\end{rmk}

\begin{cor}\label{cor_int_hens_1}
There is an isomorphism of motivic integration, with and without volume forms, in the following cases:
\begin{enumerate}
\item\label{cas_eff} if $\cL=\cL_{val}$ and the residue field of $\cT$ is algebraically bounded over $\emptyset$;
\item\label{cas_RV} if $R=\RV^\times$;
\item\label{cas_Kb}  if $R=\Kb^\times$ and $\cL$ is obtained from $\cL_{\mathrm{DP}}$ by possibly adding structure on $\Kb$.
\end{enumerate}
\end{cor}
\begin{proof}
That $\cT$ is 1-h-minimal is already known, \emph{cf.} \cite{cluckers_hensel_2022}. Jacobian bijection lifting property is given by Propositions \ref{prop_rel_bij_RV}, \ref{prop_rel_bij_eff} and \ref{prop_rel_bij_Kb}.
\end{proof}

\subsection{Rewriting the categories}\label{ssec_reecriture}
In each of the three previous cases, we will rewrite the definition of the $\VF$ and $\RV$-categories, both to make them more explicit and to connect some of them to known examples.

We have seen that, in any case, we could take $\RV$-partitions by subsets $Y\subset\sn(R)^l$. So let us introduce the following:

\begin{dfn}
\begin{enumerate}
\item In case \ref{cas_eff}, let $\VF_{\sn}[d]$ be the category whose objects are $\cL'\emptyset)$-definable subsets $X\subset \VF^d\times R^l\times \RV^m$ such that the projection map $X\to \VF^d\times R^l$ has finite fibres and the composition $X\to R^l\to \sn(R)^l$ is an $\RV$-partition.
\item In case \ref{cas_RV}, let $\VF[d,\cdot]$ be the category whose objects are $\cL'(\emptyset)$-definable subsets $X\subset \VF^d\times \RV^m$ such that the composition $X\to (\RV^\times)^m\to \sn(\RV)^m$ is an $\RV$-partition.
\item In case \ref{cas_Kb}, let $\VF_{\Kb}[d]$ be the category whose objects are $\cL'(\emptyset)$-definable subsets $X\subset \VF^d\times \Kb^l\times\times\Gamma^m$ such that the projection map $X\to\VF^d\times\Kb^l$ has finite fibres and that the composition $X\to\Kb^l\to \sn(\Kb)^l$ is an $\RV$-partition.
\end{enumerate}
Morphisms, in each of these categories, are the $\cL'(\emptyset)$-definable maps.
\end{dfn}
\begin{rmk}
We denote the categories $\VF[d,\cdot]$ like this in reference to Hrushovski-Kazhdan's coarse categories.
\end{rmk}
\begin{lem}
The categories $\VF_{\sn}[d], \VF[d,\cdot],\VF_{\Kb}[d]$ are equivalent to $\VF_{\cT'}[d]$ in each of the cases \ref{cas_eff}, \ref{cas_RV}, \ref{cas_Kb} respectively.
\end{lem}
\begin{proof}
Clear.
\end{proof}

We now turn to $\RV$-categories.

\begin{dfn}
\begin{enumerate}
\item In case \ref{cas_eff}, let $\RV_{\sn}[d]$ be the category whose objects are $\cL(\emptyset)$-definable subsets $S\subset (\RV^\times)^d\times R^l \times \RV^m$ such that the projection map $S\to (\RV^\times)^d\times R^l$ has finite fibres.
\item In case \ref{cas_RV}, let $\RV[d,\cdot]$ be the caategory whose objects are $\cL(\emptyset)$-definable subsets $S\subset (\RV^\times)^d\times \RV^m$.
\item In case \ref{cas_Kb}, let $\RV_{\Kb}[d]$ be the category whose objects are $\cL(\emptyset)$-definable subsets $S\subset(\RV^\times)^d\times\Kb^l\times\Gamma^m$ such that the projection map $S\to(\RV^\times)^d\times\Kb^l$ has finite fibres.
\end{enumerate}
Morphisms, in each of these categories, are once more the definable map.
\end{dfn}
\begin{rmk}
In the second case, for $\cT=\ACVF(0,0)$, the category $\RV[d,\cdot]$ really is the coarse category on $\RV$ defined by Hrushovski and Kazhdan.
\end{rmk}
\begin{lem}
The categories $\RV_{\sn}[d], \RV[d,\cdot], \RV_{\Kb}[d]$ are equivalent to $\RV_{\cT'}[d]$ in each of the cases \ref{cas_eff}, \ref{cas_RV} and \ref{cas_Kb} respectively.
\end{lem}
\begin{proof}
We use the $\RV$-minimality of $\cT'$ relative to $\cT$ in order to replace everywhere $\cL(\emptyset)$-definable by  $\cL'(\emptyset)$-definable. The rest of the proof is clear.
\end{proof}

There is a similar description in the case with volume forms, which we shall not give in further detail.

\begin{rmk}
In the categories $\VF_{\sn}[d],\VF[d,\cdot]$ and $\VF_{\Kb}[d]$, we could restrict our attention to $\cL(\emptyset)$-definable objects (more precisely, the \textbf{full} subcategories they define are equivalent). This comes from the fact that any object is isomorphic to the lift of some object in the corresponding category on $\RV$. Nevertheless, it is crucial to take $\cL'(\emptyset)$-definable bijections.
\end{rmk}

\subsection{Section of the value group}
Let $\cL$ be a language obtained from $\cL_{\mathrm{DP}}$ by possibly adding structure on the residue field and the value group (separately), and let $\cT$ be a theory of henselian valued fields of zero residue characteristic. In that case, $\RV^\times$ identifies with $\Kb^\times\times\Gamma$. Also let $\cL'$ be the language we get from $\cL$ by adding $\sn:\Gamma\to\VF^\times$ a function symbol. Let $\cT'$ be the $\cL'$-theory obtained from $\cT$ by asserting that $\sn$ is a multiplicative group morphism and that, for all $\gamma\in\Gamma$, we have $\rv\circ\sn(\gamma)=\gamma$, where $\Gamma$ is seen as a definable subset of $\RV$.

Note that, under these hypotheses, $\Kb$ and $\Gamma$ are orthogonal. So we will be able to use the same arguments as in Section \ref{ssec_hens_RV}, enventually showing the similar Corollary \ref{cor_axiomes_Gamma}. Denote by $\Gamma_\VF$ the image of $\sn$.

\begin{prop}
Let $M,N\models\cT'$ be two models such that $N$ is $|M|^+$-saturated, and let $A\subset M$ be a sub-$\cL'$-structure. Let $f:A\to N$ be a partial embedding for language $\cL'$, which we additionally assume to be $\cL$-elementary. Then $f$ is a partial $\cL'$-elementary embedding. Actually, for any extension of $f$ to a partial $\cL$-elementary embedding $g:A\cup\RV(M)\to N$, there is an extension of $g$ to all of $M$ which is $\cL'$-elementary.
\end{prop}
\begin{proof}
Once more, it is enough to prove the second point, and we may assume $\VF(A)$ to be a field. We still denote by $g$ the extension to $A\cup\RV(M)\cup \Gamma_\VF(M)$ which we obtain by setting $g(\sn(\gamma))=\sn(g(\gamma))$ for all $\gamma\in \Gamma(M)$.

\begin{fact}
The map $g$ defined in that way can be extended to a partial $\cL'$-embedding on the generated structure.
\end{fact}
\begin{proof}
Again, we check that, for every polynomial expression $P(a_1,\dots,a_r,\gamma_1,\dots,\gamma_s)$ with integer coefficients and variables in $A\cup \Gamma_\VF(M)$, it holds that
\[\rv(P(g(a_1),\dots,g(a_r),g(\gamma_1),\dots,g(\gamma_s))) = g(\rv(P(a_1,\dots,a_r,\gamma_1,\gamma_s)))\,.\]
As $A$ and $\Gamma_\VF$ are stable under product, we reduce ourselves to linear combinations of elements of $\Gamma_\VF(M)$ with coefficients in $A$. Let then be $a_1,\dots,a_r\in A, \gamma_1,\dots,\gamma_r\in \Gamma_\VF(M)$. Without loss of generality, assume $|a_1\gamma_1|=\max_i|a_i\gamma_i|$. We replace $a_i$ by $a_ia_1^{-1}$ and $\gamma_i$ by $\gamma_i\gamma_1^{-1}$, so that it suffices to show equality when all $a_i\gamma_i$ are integers, and $a_1\gamma_1$ is an invertible integer. Let $I$ be the set of $1\leq i\leq r$ such that $|a_i\gamma_i|=1$. For $i\in I$, we have $\gamma_i = |a_i|^{-1}\in\Gamma_\VF(A)$; hence $\sum_{i\in I}g(a_i)g(\gamma_i) = g(\sum_{i\in I}a_i\gamma_i)$. Moreover, $\sum_{i=1}^r a_i\gamma_i=\sum_{i\in I}a_i\gamma_i$ modulo the maximal ideal, and similarly $\sum_{i\in I}g(a_i)g(\gamma_i) = \sum_{i=1}^r g(a_i)g(\gamma_i)$ because $I$ is also the set of $i$'s such that $|g(a_i)g(x_i)|=1$. This gives the result.
\end{proof}
From that point we end the proof as in Proposition \ref{prop_section}.
\end{proof}

\begin{cor}\label{cor_axiomes_Gamma}
The theory $\cT'$ is complete if $\cT$ is, has quantifier elimination relative to $\cL$, is $\RV$-minimal and has $\RV$-partitions relative to $\cT'$, and is $\RV$-thin. It is also refined.
\end{cor}
\begin{proof}
This is exactly analogous to Corollary \ref{cor_section} (and Proposition \ref{prop_theorie_raffinee}), and the same proof works \emph{verbatim}.
\end{proof}

\begin{lem}
Assume that $\cT$ is a theory of valued fields whose residue field is algebraically bounded over $\emptyset$ and that $\cL$ is obtained from $\cL_{\mathrm{DP}}$ by possibly adding structure on $\Gamma$. Then $\cT'$ has the jacobian bijection lifting property.
\end{lem}
\begin{proof}
Let $A$ be a set of parameters in $\VF$ and $R,S$ be objects of $\RV_{\cT'(A)}[d], h:R\to S$ an $A$-definable bijection. Assume without loss $A=\dcl_{\cL'}(A)\cap\VF$. Write $R\subset(\RV^\times)^d\times Y_1\times\RV^{m_1}, S\subset(\RV^\times)^d\times Y_2\times\RV^{n_2}$. For $u\in Y_1, v\in Y_2$, let $X_u, Y_v$ be the respective fibres over $u$ and $v$, and $X_{u,v}=X_u\cap f^{-1}(Y_v), Y_{u,v}=f(X_u)\cap Y_v$, as well as $f_{u,v}:X_{u,v}\to Y_{u,v}$ the restriction of $f$. We may assume that $|\cdot|:Y_i\to\Gamma^{l_i}$ are injective, with inverse $\sn$. In that way, $X_{u,v}, Y_{u,v}$ and $h_{u,v}$ are $\cL'(A,|u|,|v|)$-definable, thus $\cL(A,|u|,|v|)$-definable by $\RV$-minimality. In particular, they are also $\cL(A,u,v)$-definable. We will show that $h_{u,v}$ can be lifted in an $\cL'(A,u,v)$-definable way; by compactness, we will then be able to glue. Consequently, up to adding $u,v$ to the language (which preserves effectivity), we can take them off from notations, and reduce to the $Y_i = \{1\}$ case.

By doing a similar reasoning for the projection maps $R\to \Gamma^d \times\Gamma^{m_1}, S\to \Gamma^d\times\Gamma^{m_2}$, and due to the presence of a section $\sn:\Gamma\to\VF$, we reduce to the case in which those projections have constant image, and even to $R\subset(\Kb^\times)^d\times\Kb^{m_1}, S\subset(\Kb^\times)^d\times\Kb^{m_2}$. But then $h$ is already definable in the language $\cL_{val}(A)$, for which the jacobian lifting property holds according to \cite{stout_integration_2025}.
\end{proof}

\begin{cor}\label{cor_int_hens_2}
If the residue field of $\cT$ is algebraically bounded over $\emptyset$, and $\cL=\cL_{\mathrm{DP}}$, then $\cT'$ has an isomorphism of motivic integration, with and without volume forms.
\end{cor}

It is possible rewrite the categories analogously to Section \ref{ssec_reecriture}; we do not give details.

\section{Formalism of Cluckers-Loeser motivic integration}\label{sec_formalisme_CL}
We want to connect the previous constructions to those of Cluckers-Loeser motivic integration. More precisely, starting from the above motivic integration morphisms \emph{à la} Hrushovski-Kazhdan, we are going to define a formalism of motivic integration \emph{à la} Cluckers-Loeser, which will be done in Proposition \ref{prop_formalisme_integration}. In the particular case in which $\cL'$ is an $\RV$-minimal extension of the Denef-Pas language $\cL_{\mathrm{DP}}$ and $\cT'$ is a theory of henselian valued fields of residue characteristic zero and discrete value group, in next section we will connect this new framework to that of Cluckers-Loeser (Theorem \ref{thm_cluckers_loeser}).

As for our assumptions, we always assume that $\cT'$ has $\RV$-partitions relative to $\cT$, which is 1-h-minimal, and that it is refined. Everywhere, when a statement includes for example at the same time results for the cases with and without volume forms, the results without further hypothesis, but for the results with volume forms we have to assume $\cT'$ to be $\RV$-thin (if only for them to make sens). We could also assume everywhere that $\cT'$ is $\RV$-thin; the other cases are pathological.

Unlike in the rest of the article, here it is quite often relevant to study the case of a trivial extension $\cL=\cL', \cT=\cT'$. The first subsection applies to all the theories of motivic integration we have seen so far, but after that it will be necessary to exclude the case of $\RV$-valued volume forms. The reason for this is the lack of analogue for Serre measure on these objects, which prevents us from defining general jacobians.

\subsection{Relative motivic integration}

First of all, note that, due to control of parameters and compactness theorem, all what we have done so far also holds in relative version -- with some subtleties in the volume forms case. We give the corresponding definitions and statements. They will come in three forms: without volume forms, a "rigid" version with volume forms and finally with volume forms and taking into account the geometry of the base. In any case, we find a motivic integration map in Proposition \ref{prop_HK_rel}.

\begin{dfn}
Let $Z\subset\VF^n\times\RV^m$ be an $\emptyset$-definable set. A set $Y\subset Z\times\VF^l$ \emph{maps into $\RV$ relative to $Z$} if, for every $z\in Z$, the fibre $Y_z$ $z$-definably maps into $\RV$.
\end{dfn}
By compactness, this is equivalent to the existence of an injective $\emptyset$-definable map $Y\to Z\times\RV^k$ (for some $k\in\bN$) compatible to projection to $Z$. Actually, those sets enjoy a far simpler description:

\begin{lem}\label{lem_inj_RV_rel}
Let $Z\subset\VF^n\times\RV^m$ be an $\emptyset$-definable set, and let $Y$ be an $\emptyset$-definable set which maps into $\RV$ relative to $Z$. Then there exists $Y_0$ an $\emptyset$-definable set mapping into $\RV$ (relative to $\{*\}$), a set $Y'\subset Y_0\times Z\times\RV^p$ such that the projection map $Y'\to Y_0\times Z$ has finite fibres and an $\emptyset$-definable bijection $Y\to Y'$ commuting to projections to $Z$.
\end{lem}
\begin{proof}
Let $\chi:Y\to Y_0$ be an $\RV$-partition of $Y$ and $\psi:Y\to Z\times\RV^p$ be an injection of $Y$ into $\RV$. We take $Y'$ to be the image of $Y$ under $(\chi,\psi)$. Clearly this commutes to projections to $Z$ and is $\emptyset$-definable. Let $(y,z)\in Y_0\times Z$ be a point. The fibre $\chi^{-1}(y)$ is $\cL(y)$-definable, so the fibre of $Y$ at $(y,z)$ is $\cL(y,z)$-definable. But it is contained in $\VF$, and it maps to $\RV$; thus it is finite.
\end{proof}

\subsubsection{Categories without volume forms}
In the case without volume forms, this gives motivation for the following definition:

\begin{dfn}\label{def_relatif_sans}
Let $Z\subset\VF^n\times\RV^m$ be an $\emptyset$-definable set. For $d\in\bN$, we define the \emph{categories relative to $Z$ without volume forms} in the following way:

\begin{enumerate}
\item the objects of $\VF_Z[d]$ are the $\emptyset$-definable subsets $X\subset Z\times \VF^d\times Y \times \RV^k$, where $Y$ maps into $\RV$ and the projection map $X\to Z\times \VF^d\times Y$ has finite fibres. If $X,X'$ are two objects of $\VF_Z[d]$, a morphism $f:X\to X'$ is an $\emptyset$-definable map $f:X\to X'$ commuting to the projections to $Z$.
\item the objects of $\RV_Z[d]$ are the $\emptyset$-definable subsets $R\subset Z\times(\RV^\times)^d\times Y \times \RV^k$, where $Y$ maps into $\RV$ and the projection map $R\to Z\times (\RV^\times)^d\times Y$ has finite fibres. If $R,S$ are two objects of $\RV_Z[d]$, a morphism $f:R\to S$ is an $\emptyset$-definable map $f:R\to S$ commuting to the projections to $Z$.
\end{enumerate}

Denote by $\VF_Z$ the colimit of the categories $\VF_Z[d]$ (for the natural inclusions $\VF_Z[d]\subset\VF_Z[d+1]$), as well as $\RV_Z[\leq d]$ the direct sum of the categories $\RV_Z[\delta], \delta\leq d$, and $\RV_Z[*]$ the colimit of the $\RV_Z[\leq d]$. We also define $\Isp^Z$ as the congruence of the semiring $\Kp\RV_Z[*]$ generated by the relation $([1]_1, [\RV_{<1}]_1 + 1)$.
\end{dfn}

When $Z=\{*\}$ is a definable point, we find back the previous definitions. We think to these objects as versions "in a definable family over $Z$" of the above. So the objects over $Z$ are relative versions, while those over the point are absolute versions. More precisely, for every $z\in Z$, there is an "evaluation at $z$" functor given by
\[\begin{array}{rcl}
\RV_Z[*]&\to&\RV_{\cT'(z)}[*]\\
R&\mapsto&R_z\,,
\end{array}\]
with $R_z$ the fibre of $R$ over $z$ (the functor being defined by taking restriction on the morphisms). Note that here we dealt with the example of $\RV_Z[*]$, but the other categories can be dealt with in the same way. Also note that, by taking Grothendieck semigroups, this evaluation morphism induces a semiring morphism.

The following lemma justifies this functional point of view by showing that "we can test conditions pointwise":

\begin{lem}\label{lem_sans_forme_fibre}
Let $Z$ be an $\emptyset$-definable set. For every $d\in\bN$, define the category $\widetilde \VF_Z[d]$ (resp. $\widetilde\RV_Z[d]$) as follows: an $\emptyset$-definable set $X\subset Z\times\VF^l\times\RV^k$ is an object of $\widetilde\VF_Z[d]$ (resp. of $\widetilde\RV_Z[d]$) if and only if, for all $z\in Z$, the fibre $X_z$ is in $\VF_{\cT'(z)}^*[d]$ (resp. in $\RV_{\cT'(z)}^*[d]$). Similarly, an $\emptyset$-definable map $f:X\to X'$ between two objects of $\widetilde\VF_Z[d]$ (resp. of $\widetilde\RV_Z[d]$) is a morphism if and only if it commutes to the projection maps towards $Z$ and, for all $z\in Z$, the induced map $f_z:X_z\to X'_z$ is a morphism in $\VF_{\cT'(z)}[d]$ (resp. in $\RV_{\cT'(z)}[d]$). Then there is a canonical inclusion $\VF_Z[d]\subset\widetilde\VF_Z[d]$ (resp. $\RV_Z[d]\subset\widetilde\RV_Z[d]$) which is an equivalence of categories.

If $X,X'$ are objects of $\VF_Z[*]$ (resp. of $\RV_Z[*]$), then $[X]=[X']$ in $\Kp\VF_Z[*]$ (resp. in $\Kp\RV_Z[*]$) if and only if $[X_z]=[X'_z]$ in $\Kp\VF_{\cT'(z)}$ (resp. in $\Kp\RV_{\cT'(z)}[*]$) for all $z\in Z$.

In a similar way, if $X, X'$ are two objects of $\RV_Z[*]$, then $([X],[X'])\in\Isp^Z$ if and only if, for all $z\in Z$, the fibre $([X_z],[X'_z])$ is in $\Isp^{\cT'(z)}$.
\end{lem}
\begin{proof}
Let us start by the second point. If $[X]=[X']$, then there is an $\emptyset$-definable bijective map $f:X\to X'$ commuting to the projection maps towards $Z$. For $z\in Z$, the fibre $f_z:X_z\to X'_z$ is a $z$-definable bijection, whence $[X_z]=[X'_z]$. Conversely, assume that for all $z\in Z$ it holds that $[X_z]=[X'_z]$, \emph{i.e.} there is a $z$-definable bijection $f_z:X_z\to X'_z$. By compactness, we glue these bijections together into an $\emptyset$-definable bijection $f:X\to X'$ commuting to the projection maps towards $Z$. Thus $[X]=[X']$.

Now we turn to the first point. The existence of such an inclusion follows from the definitions, as well as its full faithfulness. As for essential surjectivity, we only deal with $\VF_Z[d]$, since $\RV_Z[d]$ is similar. Let $X\subset Z\times\VF^l\times\RV^k$ be such that, for all $z\in Z$, the fibre $X_z$ is in $\VF_{\cT'(z)}^*[d]$. Then $l\geq d$, and moreover for all $z\in Z$ there is a subset $Y_z\subset\VF^{l-d}$ such that $X_z\subset\VF^d\times Y_z$, the set $Y_z$ $z$-definably maps into $\RV$ and the projection map $X_z\to\VF^l$ has finite fibres. By glueing (and permuting the coordinates), we find out that $X\subset Y\times\VF^d\times\RV^k$, with $Y\subset Z\times\VF^{l-d}$ an $\emptyset$-definable set that maps into $\RV$ relative to $Z$, and the projection map $X\to Y\times\VF^d$ has finite fibres. By Lemma \ref{lem_inj_RV_rel}, there is an $\emptyset$-definable bijection $Y\to Y'\subset Z\times Y_0\times\RV^m$ that commutes to the projection maps towards $Z$, with $Y_0$ mapping into $\RV$ and $Y'\to Z\times Y_0$ having finite fibres. In particular, up to applying a definable bijection relative to $Z$, we may replace $X$ by a subset $X'\subset Z\times Y_0\times\VF^d\times\RV^{k+m}$ which is indeed inside $\VF_Z[d]$ (up to permuting the coordinates).

Finally we prove the last point, by closely followinng \cite{stout_integration_2025}, Lemma 5.4.2. Assume that $X,X'\in\Ob\RV_Z[\leq d]$. Saying that $([X], [X'])\in\Isp^Z$ is equivalent to saying that there are $T_{i,j}, i,j\leq d$ in $\Ob\RV_Z[\leq d]$ such that
\[\begin{cases}
[X] = \sum_{i,j\leq d} [T_{i,j}]([\RV^\times_{<1}]_1+1)^i[1]_1^j\,, \\
[X'] = \sum_{i,j^leq d} [T_{i,j}]([\RV^\times_{<1}]_1+1)^j[1]_1^i\,.
\end{cases}\]
By taking the fibre at $z\in Z$, we now see that $([X_z],[X'_z]) \in\Isp^{\cT'(z)}$. Conversely, if $([X_z],[X'_z])\in\Isp^{\cT'(z)}$ for all $z$, we find such relations for some objects $T_{i,j,z}$. By compactness, we are able to glue them together; now using the first point, we find objects $T_{i,j}$ of $\RV_Z[\leq d]$. Since, according to the second points, equalities can be checked pointwise, we still have the desired equalities, which gives the desired conclusion.
\end{proof}

Now we define the support of an object:
\begin{dfn}
The support of an object $X\in\Ob\VF_Z[d]$ (resp. $\Ob\RV_Z[d]$) is the set of $z\in Z$ such that $X_z\neq\emptyset$. We denote it by $\Supp X$.
\end{dfn}
This definition can be extended to $\VF_Z, \RV_Z[\leq d]$ and $\RV_Z[*]$ in a natural way.

\begin{rmk}\label{rmk_pullback}
Let $f:Z'\to Z$ be an $\emptyset$-definable map. Then it naturally induces a pullback functor $f^*\RV_Z[*]\to\RV_{Z'}[*]$. Similar statements holds for $\VF_Z$, etc.
\end{rmk}

\subsubsection{Categories with volume forms}
With volume forms, we can follow exactly the same reasoning as in Lemma \ref{lem_sans_forme_fibre}. Let us give in particular the following definition:

\begin{dfn}\label{def_cat_rig}
Let $Z$ be an $\emptyset$-definable set. For every $d\in\bN$, let $\mu_\Gamma\VF_Z^{\mathrm{rig}}[d]$ be the category whose objects are pairs $(X,\omega)$, with $X\in\Ob\VF_Z[d]$ and $\omega:X\to\Gamma$ an $\emptyset$-definable map. A morphism $(X,\omega)\to (X',\omega')$ is a partial $\emptyset$-definable map $f:X\to X'$ commuting to the projections towards $Z$ and such that, for all $z\in Z$, the map $f_z:(X_z,\omega_z)\to (X'_z,\omega'_z)$ is a morphism in $\mu_\Gamma\VF_{\cT'(z)}[d]$.

Let also $\mu_\Gamma\RV_Z^{\mathrm{rig}}[d]$ be the category whose objects are pairs $(R,\omega)$, with $R\in\Ob\RV_Z[d]$ and $\omega:R\to\Gamma$ an $\emptyset$-definable map. A morphism $(R,\omega)\to (S,\rho)$ is an $\emptyset$-definable map $f:R\to S$ that commutes to the projection maps towards $Z$ and such that, for all $z\in Z$, the map $f_z:(R_z,\omega_z)\to (S_z,\rho_z)$ is a morphism in $\mu_\Gamma\RV_{\cT'(z)}[d]$.

We define as usual $\mu_\Gamma\VF_Z^{\mathrm{rig}}[*], \mu_\Gamma\RV_Z^{\mathrm{rig}}[\leq d]$ and $\mu_\Gamma\RV_Z^{\mathrm{rig}}[*]$.
\end{dfn}

The analogous statement to Lemma \ref{lem_sans_forme_fibre} clearly holds (by definition) for these categories, and from that fact we deduce the existence of pullback functors. Nevertheless, the description of the morphisms in this category is not satisfying for our purposes. In order to find an alternative definition, let us first give a definition of the $\VF$ categories that will better fit our desiderata:

\begin{dfn}\label{def_cat_avec_VF}
Let $Z$ be an $\emptyset$-definable set, and let $\delta\in\bN$ be an integer. For every $d\in\bN$, let $\mu_\Gamma\VF^\delta_Z[d]$ be the category whose objects are pairs $(X,\omega)\in\Ob\mu_\Gamma\VF_Z^{\mathrm{rig}}[d]$ such that $\dim_\VF\Supp X\leq\delta$. A morphism $(X,\omega)\to (X',\omega')$ is an essential bijection $f:X\to X'$ in dimension $d+\delta$ such that;
\begin{enumerate}
\item $f$ commutes to the projection maps towards $Z$;
\item for almost all $x\in X$ (\emph{i.e.} outisde of a set of dimension at most $d+\delta-1$), we have
\[\omega(x) = \omega'(f(x))|\Jac f(x)|\,.\]
\end{enumerate}
Let $\mu_\Gamma\VF_Z^\delta[*]$ be the direct sum of the categories $\mu_\Gamma\VF_Z^\delta[d], d\in\bN$, and let $\mu_\Gamma\VF_Z[*]$ be the direct sum of the categories $\mu_\Gamma\VF_Z^\delta[*], \delta\in\bN$.
\end{dfn}
When $Z=\{*\}$, this indeed gives us back the category $\mu_\Gamma\VF_{\cT'}[*]$ with $\delta=0$, and for $\delta>0$ the categories here defined are trivial (equivalent to the category with one object and one morphism). In particular $\mu_\Gamma\VF_{\cT'}[*]$ is equivalent $\mu_\Gamma\VF_{\{0\}}[*]$. More generally, if $\delta>\dim_\VF Z$, then $\mu_\Gamma\VF_Z^\delta[*]$ is trivial.

For $Z$ any base set, if we see these categories as functions over $Z$, then what we actually did is that we introduced functions defined almost everuwhere, with this word holding a different meaning in each dimension: if we are interested in functions supported in dimension $1$, then we look them up to a set of dimension $0$. The most interesting case is $\delta=\dim_\VF Z$, where this "almost everywhere" has the usual meaning. Indeed, if $\delta<\dim_\VF Z$, then for every $(X,\omega)\in\Ob\mu_\Gamma\VF_Z^\delta[*]$ there is $Z'\subset Z$ an $\emptyset$-definable set of dimension $\delta$ such that $(X,\omega)$ can be interpreted as an object of $\mu_\Gamma\VF_{Z'}^\delta[*]$, which is just $Z'=\Supp X$. However, the categories with $\delta<\dim_\VF Z$ will play an important technical role.

Now we introduce the analogous categories over $\RV$:

\begin{dfn}\label{def_cat_avec_RV}
Let $Z$ be an $\emptyset$-definable set, and let $\delta\in\bN$ be an integer. For every $d\in\bN$, let $\mu_\Gamma\RV_Z[*]$ be the category whose objects are the pairs $(R,\omega)\in\Ob\mu_\Gamma\RV_Z^{\mathrm{rig}}[d]$ such that $\dim_\VF(\Supp R)\leq\delta$. Let $(R,\omega),(S,\rho)$ be two objects, and denote by $\pi:R\to Z$ the structural morphism. A map $(R,\omega)\to (S,\rho)$ is a partial $\emptyset$-definable map $f:\pi^{-1}(Z_0)\subset R\to S$ such that:
\begin{enumerate}
\item $f$ commutes to the projections towards $Z$;
\item $Z_0\subset\Supp R\cap\Supp S$ and $\dim_\VF (\Supp(R\cup S)\setminus Z_0)<\delta$;
\item for all $z\in Z_0$, the map $f_z:(R_z,\omega_z)\to (S_z,\rho_z)$ is a morphism of $\mu_\Gamma\RV_{\cT'(z)}[d]$.
\end{enumerate}
In addition we identify two morphisms $f,g:(R,\omega)\to (S,\rho)$ if there is an $\emptyset$-definable subset $Z_0\subset \Supp R$ such that $\dim_\VF (\Supp(R)\setminus Z_0)<\delta$ and $f_{|\pi^{-1}(Z_0)} = g_{|\pi^{-1}(Z_0)}$.

We also define $\mu_\Gamma\RV_Z^\delta[\leq d]$ as the direct sum of the $\mu_\Gamma\RV_Z^\delta[d'], d'\leq d$, and $\mu_\Gamma\RV_Z^\delta[*]$ as the colimit of the $\mu_\Gamma\RV_Z^\delta[\leq d], d\in\bN$, and finally $\mu_\Gamma\RV_Z[*]$ as the direct sum of the $\mu_\Gamma\RV_Z^\delta[*]$.
\end{dfn}

From now on, the categories defined in Definition \ref{def_cat_rig} will be called the \emph{rigid} versions of the categories defined in \ref{def_cat_avec_VF}, \ref{def_cat_avec_RV}.

An important thing is that $\Kp\mu_\Gamma\VF_Z[*],\Kp\mu_\Gamma\RV_Z[*]$ are bigraded semigroups, but not semirings, unlike $\Kp\mu_\Gamma\VF_Z^{\mathrm{rig}}[*], \Kp\mu_\Gamma\RV_Z^{\mathrm{rig}}[*]$ which are graded semirings. However, if $\delta=\dim_\VF Z$, then indeed $\Kp\mu_\Gamma\VF_Z^\delta[*],\Kp\mu_\Gamma\RV_Z^\delta[*]$ are semirings. More interestingly, there is an action by multiplication of the semiring $\Kp\mu_\Gamma\RV_Z^{\mathrm{rig}}[*]$ on each of the semigroups $\Kp\mu_\Gamma\RV_Z^\delta[*],\delta\in\bN$, turning $\Kp\mu_\Gamma\RV_Z[*]$ into a bigraded semi-$\Kp\mu_\Gamma\RV_Z^{\mathrm{rig}}[*]$-module, with natural action of the semiring on the grading by $d$ and trivial action on the grading by $\delta$. Similarly, $\Kp\mu_\Gamma\VF_Z[*]$ is naturally a bigraded semi-$\Kp\mu_\Gamma\VF_Z^{\mathrm{rig}}[*]$-module.

\begin{dfn}
Let $Z$ be an $\emptyset$-definable set. Let $\Isp^{Z,\mu,\mathrm{rig}}$ be the congruence over $\Kp\mu_\Gamma\RV_Z^{\mathrm{rig}}[*]$ generated by the relation $([1]_1, [\RV_{<1}]_1)$. We denote by $\Isp^{Z,\mu}$ the semigroup congruence on $\Kp\mu_\Gamma\RV_Z[*]$ induced by $\Isp^{Z,\mu,\mathrm{rig}}$ via the semi-module structure.
\end{dfn}

As in the case without volume forms, for $(R,\omega),(S,\rho)$ two objects of $\mu_\Gamma\RV_Z^{\mathrm{rig}}[*]$, we have $([R,\omega],[S,\omega])\in\Isp^{Z,\mu,\mathrm{rig}}$ if and only if for all $z\in Z$ we have $([R_z,\omega_z],[S_z,\rho_z])\in\Isp^{\cT'(z),\mu}$. In $\mu_\Gamma\RV_Z[*]$, there is a similar description: if $(R,\omega),(S,\rho)$ are two objects in $\mu_\Gamma\RV_Z^\delta[*]$, with $Z_0=\Supp R\cup\Supp S$, then $([R,\omega],[S,\rho])\in\Isp^{Z,\mu}$ if and only if there is a subset $Z_1\subset Z_0$ such that $\dim_\VF(Z_0\setminus Z_1)<\delta$ and, for all $z\in Z_1$, we have $([R_z,\omega_z],[S_z,\rho_z])\in\Isp^{\cT'(z),\mu}$.

Let us justify these assertions. The first case is clear, and follows from the fibrewise description of the category $\mu_\Gamma\RV_Z^{\mathrm{rig}}[*]$. For the second case, first of all this is indeed a semigroup congruence, \emph{i.e.} an equivalence relation stable under sum. Moreover, it obviously contains the relation $\Isp^{Z,\mu}$. Conversely, we see that, if $([R,\omega],[S,\rho])$ lies in this relation, then by taking the restriction to $Z_1$ (which is an isomorphism) we obtain a pair $([R_1,\omega_1],[S_1,\rho_1])$ which is in $\Isp^{Z,\mu,\mathrm{rig}}$, hence in $\Isp^{Z,\mu}$.

\subsubsection{Integration morphisms}
The above definitions are designed so that we have lifting maps
\[\fL:\Ob\RV_Z[\leq d]\to\Ob\VF_Z[d]\text{ et }\fL:\Ob\mu_\Gamma\RV_Z^\delta[d]\to\Ob\mu_\Gamma\VF_Z^\delta[d]\]
for every $\delta,d\in\bN$. Note that, since the rigid categories have the same objects as non-rigid categories, these maps are also defined for them.

Now we may state the main theorem of motivic integration:

\begin{prop}\label{prop_HK_rel}
Let $Z\subset\VF^n\times\RV^m$ be a $\emptyset$-definable set. There is a semiring morphism
\[\int:\Kp\VF_Z\to \Kp\RV_Z[*]/\Isp^Z\,,\]
which is Hrushovski-Kazhdan motivic integration without volume forms relative to $Z$, and a morphism of graded semirings
\[\int^{\mathrm{rig}}:\Kp\mu_\Gamma\VF^{\mathrm{rig}}_Z[*]\to\Kp\mu_\Gamma\RV_Z^{\mathrm{rig}}[*]/\Isp^{Z,\mu,\mathrm{rig}}\,,\]
which is Hrushovski-Kazhdan rigid motivic integration with volume forms relative to $Z$.

Finally there is also a morphism of bigraded semigroups
\[\int:\Kp\mu_\Gamma\VF_Z[*]\to\Kp\mu_\Gamma\RV_Z[*]/\Isp^{Z,\mu}\,,\]
which is Hrushovski-Kazhdan motivic integration with volume forms relative to $Z$, and which is equivariant with respect to $\int^{\mathrm{rig}}$.

Those morphisms are characterised by the fact that, for all $[R]$ in the codomain, we have $\int[\fL R]=[R]$.

Moreover, if $\cT'$ has the bijection lifting property (resp. the jacobian bijection lifting property) and $Z\subset\VF^n$, then the first (resp. the last two) of those morphisms is an isomorphism.
\end{prop}
\begin{rmk}
If $\cT'$ enjoys a section $\sn$ of $\rv$, then the jacobian bijection lifting property holds. In addition, by using $\sn$, we can always reduce ourselves to the case $Z\subset\VF^n$. As a consequence, in that situation, there are such isomorphisms for any $Z$.
\end{rmk}
\begin{proof}
As mentioned above, this is merely a consequence of compactness and the absolute case. More precisely, for the first two morphisms:
\begin{itemize}
\item to define the morphism, we proceed as in the absolute case but taking cell decompositions relative to $Z$, as is done for instance in Section \ref{ssec_dec_cell} relative to a set $Y$ mapping into $\RV$. Note that the construction in the absolute cases already makes use of the fact that we can actually define it in the relative case, this is how the categories $\VFR_{\cT'}[k,l]$ work;
\item to show that every object on $\VF$ is isomorphic to a lift, we use once more cell decompositions relative to $Z$;
\item to lift bijections, if $Z\subset\VF^n$ and $\cT'$ has the (jacobian) bijection lifting property, then for every $z\in Z$ the theory $\cT'(z)$ still has the (jacobian) bijection lifting property) by definition (and thanks to $Z\subset\VF^n$); so we lift the induced bijection in each fibre, and glue by compactness.
\end{itemize}
For the third morphism, just note that $\int^{\mathrm{rig}}$ preserves support, hence an equality almost everywhere before integration becomes an equality almost everywhere after integration.
\end{proof}

\begin{rmk}
Throughout this section, we have only worked with the categories without volume forms or with $\Gamma$-valued volume forms. Nevertheless, everything we did in that second case works \emph{mutatis mutandis} with $\RV$-valued volume forms. 
\end{rmk}

\subsection{Constructible motivic functions}\label{ssec_fonctions_constructibles_motiviques}
From now on, we will only study motivic integration without volume forms or with $\Gamma$-valued volume forms.

Now we may define the semigroup of (positive) constructible motivic Functions:

\begin{dfn}
Let $Z\subset\VF^n\times\RV^m$ be an $\emptyset$-definable set. The semigroup of constructible motivic Functions over $Z$ without volume forms (resp. with volume forms) is the Grothendieck semigroup $\Kp\RV_Z[*]/\Isp^Z$ (resp. $\Kp\mu_\Gamma\RV_Z[*]/\Isp^{Z,\mu}$). We denote it by $C_+^\chi(Z)$ (resp. $C_+^{\mu_\Gamma}(Z)$).

We will also denote by $\sC_+^{\mu_\Gamma}(Z)$ the semiring $\Kp\RV^{\mathrm{rig}}_Z[*]/\Isp^{Z,\mu,\mathrm{rig}}$, which we call the semiring of constructible motivic functions with volume forms over $Z$.
\end{dfn}
Although $\sC_+^{\mu_\Gamma}(Z)$ is still a graded semiring and $C_+^{\mu_\Gamma}(Z)$ a bigraded semigroup, we will no longer be interested in the grading by $d$. One of the main reasons for this is thas, in \cite{cluckers_constructible_2008}, Cluckers and Loeser normalise the motivic measure so that $\int[\cO]=1$, which cancels this grading. So from now on, whenever we see $C_+^{\mu_\Gamma}(Z)$ as a graded semigroup, this is in terms of the $\delta$ grading.

Sometimes it will be useful to let $\sC_+^\chi(Z) = C_+^\chi(Z)$, which we then call semiring  of constructible motivic functions without volume forms over $Z$.

There is another construction of $C_+^{\mu_\Gamma}(Z)$. For $\delta\in\bN$, let $\sC_+^{\mu_\Gamma,\leq\delta}$ be the subsemigroup of $\sC_+^{\mu_\Gamma}(Z)$ composed of the functions whose support has dimension at most $\delta$. Then
\[C_+^{\mu_\Gamma,\delta}(Z) = \sC_+^{\mu_\Gamma,\leq\delta}(Z)/\sC_+^{\mu_\Gamma,\leq\delta-1}(Z)\,.\]

There are pullback morphisms between those semirings. Indeed, if $f:Z'\to Z$ is an $\emptyset$-definable map, we immediatly find by definition of the relative categories a functor $f^*:\RV_Z[*]\to\RV_{Z'}[*]$ (and $f^*:\mu_\Gamma\RV_Z^{\mathrm{rig}}[*]\to\mu_\Gamma\RV_{Z'}^{\mathrm{rig}}[*]$), see Remark \ref{rmk_pullback}. By taking semigroups, this functore induces a semiring morphism $f^*:C_+^\chi(Z)\to C_+^\chi(Z')$ (resp. $f^*:\sC_+^{\mu_\Gamma}(Z)\to \sC_+^{\mu_\Gamma}(Z')$).

Moreover, in the case with volume forms, if $f$ has $0$-dimensional fibres, the functor $f^*$ preserves the dimension of the support, hence induces a functor $f^*:\mu_\Gamma\RV_Z[*]\to\mu_\Gamma\RV_{Z'}[*]$. By taking semigroups again, we find a morphism of graded semigroups $f^*:C_+^{\mu_\Gamma,d}(Z)\to C_+^{\mu_\Gamma,d}(Z')$, which is equivariant relative to the morphism between the semirings of functions. All these morphisms clearly enjoy functoriality.

\subsection{Differentiation and jacobian}
Here we show some results around differentiation and construction of the jacobian. These results are not needed to define motivic integration and they use the notion of constructible motivic functions, so we rather present them  here.

We will try and define a canonical volume form \emph{à la} Serre on $X\subset\VF^n\times\RV^m$ definable of dimension $d$, and use it to get a notion of jacobian in Definition \ref{def_jacobien}. To do so, we take our inspiration from \cite{cluckers_constructible_2008}, Section 8.

\begin{dfn}
Let $Z\subset\VF^\delta\times\RV^m$ be an $\emptyset$-definable set of dimension $\delta$. We define the set $|\widetilde\Omega|_+(Z)$ of positive definable volume forms on $Z$ as the rank one free semimodule on $C_+^{\mu_\Gamma,\delta}(Z)$ with generator $|dx_1\wedge\dots\wedge dx_\delta|$ (formal notation).

If $h:Z_1\to Z_2$ is an $\emptyset$-definable map whose fibres with $0$-dimensional fibres, with $Z_i\subset\VF^\delta\times\RV^{m_i},i=1,2$, of dimension $\delta$, let $h^*:|\widetilde\Omega|_+(Z_2)\to|\widetilde\Omega|_+(Z_1)$ be the pullback morphism defined by
\[h^*(\varphi|dx_1\wedge\dots\wedge dx_\delta|) = [|\Jac(h)|](h^*\varphi)|dx_1\wedge\dots\wedge dx_\delta|\,,\]
where $[|\Jac(h)|]$ is by definition the class of $(Z_1, |\Jac(h)|)$ inside $C_+^{\mu_\Gamma,\delta}(Z_1)$.
\end{dfn}

This is well defined because $\Jac(h)$ is defined up to a subset of dimension at most $\delta-1$.

\begin{lem}\label{lem_mesure_fonctorielle}
This construction is functorial: if $Z_i\subset\VF^\delta\times\RV^{m_i}, 1\leq i\leq 3$, are $\emptyset$-definable sets of dimension $\delta$, and $f_1:Z_1\to Z_2, f_2:Z_2\to Z_3$ are $\emptyset$-definable maps with $0$-dimensional fibres, then $(f_2\circ f_1)^* = f_1^*\circ f_2^*$.
\end{lem}
\begin{proof}
This follows from chain rule and the definitions
\end{proof}

We deduce the following definition:

\begin{dfn}
Let $Z$ be an $\emptyset$-definable set of dimension $\delta$. By Lemma \ref{lem_eq_cat}, there is an $\emptyset$-definable bijection $h:Z\to Z'\subset\VF^\delta\times\RV^m$. We define $|\widetilde\Omega|_+(Z)$ as the rank one free $C_+^{\mu_\Gamma,\delta}(Z)$-semimodule with generator $h^*|dx_1\wedge dx_\delta|$, and we denote by $h^*:|\widetilde\Omega|_+(Z')\to|\widetilde\Omega|_+(Z)$ the canonical isomorphism.

By Lemma \ref{lem_mesure_fonctorielle}, for every $\emptyset$-definable map $h:Z'\to Z$ with $0$-dimension fibres and between sets of dimension $\delta$, we may define a morphism $h^*:|\widetilde\Omega|_+(Z)\to|\widetilde\Omega|_+(Z')$ functorially.
\end{dfn}
Note that nothing depend on the choices.

Let us slightly generalise the morphisms $h^*$: let $h:Z'\to Z$ be an $\emptyset$-definable map, with $\dim_\VF Z=\dim_\VF Z'=\delta$, and let $Z_0\subset Z$ be the set of $z\in Z$ such that $\dim_\VF h^{-1}(z)=0$. We denote by $h_0$ the restriction of $h$ to $Z'_0=h^{-1}(Z_0)\to Z$. If $\dim_\VF h^{-1}(Z_0) = \delta$, there is a morphism $h_0^*:|\widetilde\Omega|_+(Z)\to |\widetilde\Omega|_+(Z'_0)$. By taking composition with the extension by $0$, we find a morphism $h^*:|\widetilde\Omega|_+(Z)\to|\widetilde\Omega|_+(Z')$. When $\dim_\VF Z'_0<\delta$, we just let $h^*=0$. This construction is, again, functorial.

Now we can define Serre canonical measure.
\begin{dfn}\label{prop_mesure_Serre}
Let $Z\subset\VF^n\times\RV^m$ be an $\emptyset$-definable set of dimension $\delta$. For $I=\{i_1<\dots<i_\delta\}\subset\llbracket 1,\delta\rrbracket$, let $\pi_I:Z\to\VF^I$ be the projection map, $|\omega_I| = |dx_{i_1}\wedge\dots\wedge dx_{i_\delta}|$ and $|\omega_I|_Z = \pi_I^*|\omega_I|$.

Then there is a unique measure $|\omega_0|_Z$, called the \emph{canonical measure on $Z$}, such that, for all $I$, there is an $\emptyset$-definable map $\alpha_I:Z\to\Gamma\cup\{0\}$ such that $|\omega_I|_Z = [\alpha_I] |\omega_0|_Z$ and furthermore:
\begin{enumerate}
\item $|\alpha_I| \leq 1$;
\item $\sup_I|\alpha_I|=1$.
\end{enumerate}

We also call $|\omega_0|_Z$ the Serre measure of $Z$.
\end{dfn}
\begin{proof}
We start by uniqueness. Assume that $|\omega_0|_Z$ and $|\omega'_0|_Z$ both satisfy the hypotheses, and denote by $\alpha_I,\alpha'_I$ the corresponding maps. Let $Z_I\subset Z$ be the locus given by $\alpha_I=1$. Then $|\omega_0|_{Z_I} = |\omega_I|_{Z_I} = [\alpha'_I]|\omega_0|_{Z_I}$. In particular, for every $I,J$, on $Z_I\cap Z_J$, we have $[\alpha'_I]=[\alpha'_J]$, thus $\alpha'_I=\alpha'_J$ outside of a subset of dimension at most $d-1$. In particular, $\alpha'_I=1$ outside of a subset of dimension at most $\delta-1$, hence $[\alpha'_I]=1$. So $|\omega_0|=|\omega'_0|$ on $Z_I$, then on all $Z$.

Now we show the existence. The statement is local, so we may just assume that $\pi_I:Z\to \VF^I$ has $0$-dimensional fibres. Then $\widetilde\Omega^\RV_+(Z) = C_+^\mu(Z)|\omega_I|_Z$. By definition of pullback, for every $J$ there is an $\emptyset$-definable map $\gamma_J:Z\to\Gamma\cup\{0\}$ such that $|\omega_J|_Z=[\gamma_J]|\omega_I|_Z$. Let $\beta = \sup_J \gamma_J$, then $|\omega_0|_Z=[\beta]|\omega_I|_Z$ and $\alpha_J = \gamma_J\beta^{-1}$.
\end{proof}

Now we may define the jacobian of a map.

\begin{dfn}\label{def_jacobien}
Let $Z,Z'$ be two $\emptyset$-definable sets of dimension $\delta$, and let $f:Z'\to Z$ be an $\emptyset$-definable map with $0$-dimensional fibres. Then there is an $\emptyset$-definable map $|\Jac(f)|:Z'\to \Gamma$, unique up to a subset of dimension $\delta-1$, such that
\[f^*|\omega_0|_Z = [|\Jac(f)|]|\omega_0|_{Z'}\,.\]
\end{dfn}
The existence of such a factor in $C_+^{\mu_\Gamma, \delta}(Z')$ follows from the definition of the Serre measure. It being of the form $[\alpha]$ for a map $\alpha:Z'\to\Gamma$ is a consequence of the change of variable formulas defining pullback morphisms.

Now we have this very general notion of jacobian, we are at last able to give a new description of the category $\mu_\Gamma\VF_{\cT'}[d]$ (and of $\mu_\Gamma\VF_Z[*]$), analogous to $\VF_{\cT'}^\circ[d]$ in the case without volume forms.

\begin{dfn}\label{def_mu_Gamma_VF}
For $d\in\bN$, let $\mu_\Gamma\VF_{\cT'}^\circ[d]$ be the category whose objects are pairs $(X,\omega)$, with $X\in\Ob\VF_{\cT'}^\circ[d]$ and $\omega:X\to \Gamma$. A morphism $(X,\omega)\to (X',\omega')$ is an essential $\emptyset$-definable bijection $f:X\to X'$, such that for almost every $x\in X$ we have
\[\omega(x)=\omega'(f(x))|\Jac(f)(x)|\,.\]

Similarly, for $Z$ an $\emptyset$-definable set and $d,\delta\in\bN$ natural numbers, let $\mu_\Gamma\VF_Z^{\delta,\circ}[d]$ be the category whose objects are pairs $(X,\omega)$ with $X\subset Z\times\VF^n\times\RV^m$ and $\omega:X\to\Gamma$ both $\emptyset$-definable such that:
\begin{enumerate}
\item for all $z\in Z$, we have $\dim_\VF X_z \leq d$;
\item the projection map $X\to Z\times\VF^n$ has finite fibres;
\item $\dim_\VF \Supp X \leq\delta$.
\end{enumerate}
A morphism $(X,\omega)\to (X',\omega')$ in this category is then an essential $\emptyset$-definable bijection in dimension $d+\delta$ $f:X\to X'$ commuting to the projection maps towards $Z$ and such that, for almost every $x\in X$ (\emph{i.e.} outside of a set of dimension at most $d+\delta-1$), we have
\[\omega(x) = \omega'(f(x))|\Jac(f)(x)|\,.\]
\end{dfn}
\begin{rmk}
The category $\mu_\Gamma\VF_Z^\circ[*]$ has no rigid version: the presence of a jacobian defined almost everywhere does not allow for it.
\end{rmk}

\begin{lem}\label{lem_eq_cat_mes}
Let $d\in\bN$. Then the inclusion $\mu_\Gamma\VF_{\cT'}[d]\subset\mu_\Gamma\VF_{\cT'}^\circ[d]$ is an equivalence of categories.

Similarly, for $Z$ an $\emptyset$-definable set, $d,\delta\in\bN$ natural numbers, the inclusion $\mu_\Gamma\VF_Z^\delta[d]\subset\mu_\Gamma\VF_Z^{\delta,\circ}[d]$ is an equivalence of categories.
\end{lem}
\begin{proof}
In the first case, full faithfulness is clear; let us show that this functor is essentially surjective. Let $(X,\omega)$ be an object of $\mu_\Gamma\VF_{\cT'}^\circ[d]$. According to the equivalence of categories of Lemma \ref{lem_eq_cat}, there exists $f:X\to X'$ an $\emptyset$-definable bijection with $X'$ an object of $\VF_{\cT'}[d]$. Then let $\omega'(x) = \omega(f^{-1}(x))|\Jac(f^{-1})(x)|$, which is an $\emptyset$-definable map $X'\to\Gamma$. Clearly $f$ gives in $\mu_\Gamma\VF_{\cT'}^\circ[d]$ an isomorphism $(X,\omega)\to (X',\omega')$.

The proof of essential surjectivity still holds in the relative case, and full faithfulness is clear.
\end{proof}

\subsection{Construction of the integral}
We have now defined the main objects, so we turn to the construction of the integration morphisms -- the construction will be achieved in Proposition \ref{prop_formalisme_integration}. Set $\eta\in\{\chi,\mu_\Gamma\}$.

\subsubsection{Projection along \texorpdfstring{$\RV$}{RV}}
Let $Z\subset\VF^n\times\RV^m$ be an $\emptyset$-definable set, and let $l\in\bN, Z'\subset Z\times\RV^l$ be an $\emptyset$-definable set. Denote by $\pi:Z'\to Z$ the projection map, we will define the morphism $\pi_!$ of integration along $\pi$.

We start by the case without volume forms. Let $\varphi\in C_+^\chi(Z')$. Then there is $X\in\Ob\RV_{Z'}[*]$ such that $\varphi=[X]$. Assume for instance that $X\in\Ob\RV_{Z'}[d]$. In that case, $X\subset Z'\times(\RV^\times)^d\times Y\times\RV^k$ where $Y$ maps into $\RV$, and the projection map $X\to Z'\times(\RV^\times)^d\times Y$ has finite fibres. Since $Z'\subset Z\times\RV^l$, we also have $X\subset Z\times\RV^{d+l}\times Y\times\RV^k$, still with a finite projection towards $Z\times\RV^{d+l}\times Y$. In that way, we are able to see $X$ as an object of $\RV_Z[\leq d+l]$, which we denote by $X^{(Z)}$. Let $\pi_!\varphi = [X^{(Z)}]$.

We can restate this construction in a more categorical way: there is a forgetful functor of the structure relative to $Z'$ that only keeps structure relative to $Z$, denoted by $\pi_!:\RV_{Z'}[*]\to\RV_Z[*]$. This functore commutes with disjoint unions, thus induces a morphism between the underlying Grothendieck semigroups. Moreover, it sends $([1]_1,[\RV_{<1}]_1+1)$ to $[Z']([1]_1,[\RV_{<1}]_1+1)$, thus descends to the quotient giving a morphism between the semigroups of constructible motivic Functions without volume forms.

In the rigid case with volume forms, the same construction works up to one detail: since we are thinking of that construction as a sum over $\RV^l$, seen as a discrete set, we do not wish to interpret the $l$ $\RV$-coordinates we add as having a volume. To do so, we modify $\omega$ as follows: if $(X,\omega)$ is the original object, we replace $\omega$ bt $\omega'(x) = |\mathrm{pr}_\RV(x)|^{-1}\omega(x)$, where $\mathrm{pr}_\RV:X\to Z'\to\RV^l$ is the projection map.

The functor $\pi_!$ defined above commutes with support: clearly
\[\forall (X,\omega)\in\Ob\mu_\Gamma\RV_Z^{\mathrm{rig}}[*], \Supp \pi_!(X,\omega) = \pi(\Supp (X_\omega))\,.\]
Since $\pi$ preserves dimension, we deduce a functor $\pi_!:\mu_\Gamma\RV_{Z'}^\delta[*]\to\mu_\Gamma\RV^\delta_Z[*]$ for all $\delta\in\bN$. As in the case without volume forms, those functors induce morphisms of Grothendieck semigroups, and descend to the quotient by $\Isp^\mu$.

We have proved:

\begin{lem}\label{lem_int_RV_proj}
The above construction gives a well-defined semigroup morphism $\pi_!:C_+^\eta(Z')\to C_+^\eta(Z)$, graded in the case with volume forms $\eta=\mu_\Gamma$, as well as a semigroup morphism $\pi_!:\sC_+^\eta(Z')\to\sC_+^\eta(Z)$. Additionally, for all $\varphi\in\sC_+^\eta(Z')$, we have $\Supp(\pi_!\varphi) = \pi(\Supp\varphi)$.
\end{lem}

\subsubsection{Projection along \texorpdfstring{$\VF$}{VF}}
Let $Z\subset\VF^n\times\RV^m$ be an $\emptyset$-definable set, and let $l\in\bN,Z'\subset Z\times\VF^l$ be an $\emptyset$-definable subset. Denote by $\pi:Z'\to Z$ the projection map, and define once more the morphism $\pi_!$ of  integration along $\pi$.

We deal separately with the cases with and without volume forms. The case without volume forms is simpler, whereas is the case with volume forms we will have to be careful about grading. Notably, $\pi_!$ will no longer be a morphism of graded semigroups.

Let $\varphi\in C_+^\chi(Z')$. Then there exists $X\in\Ob\RV_{Z'}[\leq d]$ (for some $d\in\bN$) such that $[X]=\varphi$. The set $X$ is an $\emptyset$-definable subset of $Z'\times\RV^d\times Y\times\RV^m\subset Z\times\VF^l\times\RV^d\times Y\times\RV^m$, with $Y$ mapping into $\RV$, such that the projection map $X\to Z\times\VF^l\times\RV^d\times Y$ has finite fibres. Let us change our point of view and see $X$ as an object $\widetilde X$ of $\VF_{Z\times\RV^d}[l]$. We may take its motivic integral $\int [\widetilde X]\in \RV_{Z\times\RV^d}[\leq l]/\Isp^{Z\times\RV^d}$. Denote by $\pi_\RV^d:Z\times\RV^d\to Z$ the projection map; the morphism $\pi_{\RV!}^d$ defined in the above section now allows us to define $\pi_!\varphi \coloneqq \pi_{\RV}^d\int[\widetilde X]$.

In other words, when $\varphi = 1_{Z'}$, we define the integral by means of Hrushovski-Kazhdan motivic integration relative to $Z$. In general, we reduce to this situation using the integration morphisms along an $\RV$-projection and the fact that we want integration to be functorial.

\begin{lem}\label{lem_int_VF_proj}
The morphism $\pi_!$ is well-defined. Moreover, if there is a commutative diagram
\[\begin{tikzcd}
&Z_0 \arrow[ld, "\pi_1"] \arrow[rd, "\pi'_1"] & \\ Z_1 \arrow[rd, "\pi_2"] & & Z_2 \arrow[ld, "\pi'_2"] \\ & Z &
\end{tikzcd}\]
where every arrow $\pi_1,\pi'_1,\pi_2,\pi'_2$ is a projection either along $\RV$-coordinates or $\VF$-coordinates, then for all $\varphi\in C_+^\emptyset(Z_0)$ we have an equality
\[\pi_{2!}\pi_{1!}\varphi = \pi'_{2!}\pi'_{1!}\varphi\,.\]
\end{lem}
\begin{proof}
Let us first check that $\pi_!$ is well-defined. We made one choice, that of a representative $X$ of the isomorphism class $\varphi$.

We show that the construction is independent of that choice; to do so, let $X_i,i=1,2,$ have the same class in $\Kp\RV_{Z'}[*]/\Isp^{Z'}$. Assume that $X_1,X_2$ are objects of $\RV_{Z'}[\leq d]$. We may also see them as objects of $\VFR_Z[l,d]$. By assumption, the equality $[X_1]=[X_2]$ holds in $\Kp\VFR_Z[l,d]/\Isp^Z$, so by Lemma \ref{lem_int_VFR} (or more precisely its relative variant) it still holds that $\int[X_1]=\int[X_2]$ in $\Kp\RV_Z[d+l]/\Isp^Z$, as we wanted. Note that, by using the categories $\VFR$, we could spare the use of the integration morphism along $\RV$-coordinates; nonetheless, the construction we just did is clearly the same as the one described above.

Assume we have such a commutative diagram. There are three cases: either $\pi_1,\pi'_2$ are $\VF$-projections and $\pi_2,\pi'_1$ are $\RV$-projections -- which implies that each of the two pairs projects along the same set of coordinates -- or all the projections are $\VF$-projections, or all the projections are $\RV$-projections. This last case is straightforward by composition of forgetful functors. The second case is just a second case of Fubini theorem for Hrushovski-Kazhdan motivic integration (or at least its relative version), which we did not explicitly mentioned but immediately follows from the description of the integral in terms of liftings $\fL$. As for the first case, the assertion holds by construction, as we explained just before the lemma.
\end{proof}

\begin{rmk}
We did not use anything deep about the $\VFR$ categories. We just used them as tools to show that integration commutes with $\RV$, in a way compatible with quotient by $\Isp$ (\emph{cf.} \cite{stout_integration_2025}, Lemmas 5.5.12 \& 5.5.13).
\end{rmk}

Now let $\varphi\in C_+^{\mu_\Gamma,\delta}(Z')$ be a Function. Then take $(X,\omega)\in\Ob\mu_\Gamma\RV^\delta_{Z'}[*]$ such that $\varphi$ is the class of $(X,\omega)$ in $\Kp\mu_\Gamma\RV^\delta_{Z'}[*]/\Isp^{\mu,Z'}$. Assume for instance that $(X,\omega)$ is in $\mu_\Gamma\RV^\delta_{Z'}[d]$. For $z\in Z$, let $\varepsilon = \dim_{\VF} X_z$ be the dimension of the fibre of $X$ over $z$. Let us partition $Z$ into $\bigsqcup_{\varepsilon=0}^l Z_{\varepsilon}$.

Over $Z_{\varepsilon}$, we may see $(X,\omega)$ as an object $(X_{\varepsilon}, \omega)$ of $\mu_\Gamma\VF^{\delta-\varepsilon}_{Z\times(\RV^\times)^d}[\varepsilon]$ or of $\mu_\Gamma\VF^{\delta-\varepsilon}_{Z\times(\RV^\times)^d}[\varepsilon]$, by expansion by $0$. Indeed, $\delta\geq \dim_\VF\Supp_{Z'} X \geq \varepsilon+\dim_\VF\Supp_Z X_{\varepsilon}$, where $\Supp_Z$ is the support relative to $Z$, and consequently $\Supp_Z X_\varepsilon \leq \delta-\varepsilon$. Actually, we apply to $\omega$ the inverse transformation to that of Lemma \ref{lem_int_RV_proj}: if we denote by $(z',r,y,\rho)$ the coordinates on $X\subset Z'\times(\RV^\times)^d\times Y\times\RV^m$, we let $\omega' = |r|\omega$. Let $\pi_!\varphi = \sum_{\varepsilon=0}^l \pi_{\RV!}^d\int[X_{\varepsilon},\omega']$.

We expand by linearity the morphism $\pi_!$ to $C_+^{\mu_\Gamma}(Z')$.

\begin{lem}\label{lem_int_proj_VF}
The morphism $\pi_!$ is well defined. Moreover, if we have a commutative diagram
\[\begin{tikzcd}
&Z_0 \arrow[ld, "\pi_1"] \arrow[rd, "\pi'_1"] & \\ Z_1 \arrow[rd, "\pi_2"] & & Z_2 \arrow[ld, "\pi'_2"] \\ & Z &
\end{tikzcd}\] where each of the arrows $\pi_1,\pi'_1,\pi_2,\pi'_2$ is a projection along $\RV$-coordinates or $\VF$-coordinates, then for all $\varphi\in C_+^{\mu_\Gamma}(Z_0)$ it holds that
\[\pi_{2!}\pi_{1!}\varphi = \pi'_{2!}\pi'_{1!}\varphi\,.\]
\end{lem}
\begin{proof}
Let us start by proving good definition. We made one choice, that of $(X,\omega)$. So let $(X',\omega')$ be such that $[X,\omega]=[X',\omega']$ in $\Kp\mu_\Gamma\RV^\delta_{Z'}[*]/\Isp^{\mu,Z'}$. We may first assume that $[X,\omega]=[X',\omega']$ in $\Kp\mu_\Gamma\RV^{\mathrm{rig}}_{Z'}[*]/\Isp^{\mu,Z',\mathrm{rig}}$. In this situation, the proof is dealt with analogously to the case without volume forms, and we get a morphism $\pi_!^{(\delta)}:\sC_+^{\mu_\Gamma,\leq\delta}(Z')\to C_+(Z)$.

Now it is enough to show that, for all $\varphi\in\sC_+^{\mu_\Gamma,\leq\delta-1}(Z')$, we have $\pi_!^{(\delta)}\varphi=0$. Let $(X,\omega)\in\Ob\mu_\Gamma\RV_{Z'}^{\mathrm{rig}}[*]$ be such that $[X,\omega] = \varphi$ in $\sC_+^{\mu_\Gamma}(Z')$. We define as in the construction of $\pi_!^{(\delta)}$ a partition $Z = \bigsqcup_{\varepsilon=0}^l Z_{\varepsilon}$. For $\varepsilon=0,\dots, l$, we have $\dim_\VF Z_{\varepsilon} + \varepsilon \leq \delta-1$, so $\dim_\VF Z_{\varepsilon} \leq \delta-\varepsilon-1$. But the component of $\pi_!^{(\delta)}\varphi$ inside $C_+^{\delta-\varepsilon}(Z)$ has support inside $Z_{\varepsilon}$; thus it cancels, and $\pi_!^{(\delta)}\varphi=0$.

We show commutativity of the diagram exactly as in the case without volume forms. Passing from $\omega$ to $\omega'$ exactly compensates the analogous construction in the case of $\RV$-projections.
\end{proof}

\begin{rmk}
Actually, the object $(X_\varepsilon,\omega)$ should rather be in $\mu_\Gamma\VF_{Z\times\RV^d}^{\delta-\varepsilon, \circ}[\varepsilon]$. Thanks to the equivalence of Lemma \ref{lem_eq_cat_mes}, we may ignore this difference. However, it explains that it is impossible to obtain a rigid version of this morphism.
\end{rmk}

\subsubsection{Projections along any coordinates}\label{sssec_proj}
Let $Z$ be an $\emptyset$-definable set, $Z'\subset Z\times\VF^n\times\RV^m$ be an $\emptyset$-definable subset, and $\pi:Z'\to Z$ be the projection map. Using the two above constructions, we may define $\pi_!: C_+^*(Z')\to C_+^*(Z)$ by composing the projection maps along one coordinate at a time; by Lemmas \ref{lem_int_VF_proj} and \ref{lem_int_proj_VF}, the morphism thus defined is independent of the order of the projections. Moreover, if we have a commutative diagram \[\begin{tikzcd}
&Z_0 \arrow[ld, "\pi_1"] \arrow[rd, "\pi'_1"] & \\ Z_1 \arrow[rd, "\pi_2"] & & Z_2 \arrow[ld, "\pi'_2"] \\ & Z &
\end{tikzcd}\] where each of the arrows $\pi_1,\pi'_1,\pi_2,\pi'_2$ is a projection along some coordinates, then $\pi_{2!}\pi_{1!}=\pi'_{2!}\pi'_{1!}$.

\subsubsection{Bijective projections}
Let $Z$ be an $\emptyset$-definable set, $Z'\subset Z\times\VF^n\times\RV^m$ be an $\emptyset$-definable subset, and assume that the projection map $\pi:Z'\to Z$ is bijective. There is an equivalence of categories $\pi^*:\RV_{Z}[*]\to\RV_{Z'}[*]$ given by pullback. From this we deduce an isomorphism $\pi^*:C_+^\emptyset(Z)\to C_+^\emptyset(Z')$.

\begin{lem}\label{lem_proj_bij_sans}
In the case without volume forms, $\pi_!$ is an isomorphism of semirings, inverse to $\pi^*$.
\end{lem}
\begin{proof}
It is enough to deal separately with the case of a projection along $\VF$ and that of a projection along $\RV$; both are clear by construction.
\end{proof}

Now we deal with the case with volume forms. We also have a pullback morphism $\pi^*:C_+^{\mu_\Gamma}(Z)\to C_+^{\mu_\Gamma}(Z')$ which is an isomorphism of graded semigroups, and a semiring isomorphism at grading $\delta=\dim_\VF Z$. However here $\pi^*$ will not be inverse to $\pi_!$, due to the change of variables formula.

We prove the following lemma, thanks to which we will be able to define all the integration morphisms:

\begin{lem}\label{lem_proj_bij_vol}
Assume that the map $\pi:Z'\to Z$ is a bijective projection along some coordinates. Then $\pi_!$ is an isomorphism of graded semigroups. Moreover, its inverse is defined as follows: let $\varphi\in C_+^{\mu_\Gamma,\delta}(Z)$ be a Function and let $X=\Supp \varphi, X' = \pi^{-1}(X)\subset Z'$. Then 
\[(\pi_!)^{-1}\varphi = [|\Jac(\pi_{|X'})|^{-1}]\pi^*\varphi\,.\]
\end{lem}
\begin{proof}
By chain rule, it is enough to deal with the case where $\pi$ is a projection along one coordinate.

Assume $\pi$ to be an $\RV$-projection. Then $(\pi_!)^{-1}=\pi^*$ by construction; let us make sure that $|\Jac(\pi_{|X'})|=1$. The projections of $X$ and $X'$ to the $\VF^I$ are the same for every $I$; as a consequence, it follows from the definition of the canonical measure that $\pi^*|\omega_0|_X=|\omega_0|_{X'}$,, hence $|\Jac(\pi_{|X'})|=1$.

Now assume that $\pi$ is a projection along $\VF$, and write $\varphi = [R,\omega]$ with $(R,\omega)\in\Ob \mu_\Gamma\RV_Z[d]$. The result follows at once from the definition of the category $\mu_\Gamma\VF_{X\times\RV^d}^\circ[0]$ and of the morphism $\pi_!$.
\end{proof}
\begin{rmk}
The sets $X,X'$ are only defined up to a subset of dimension at most $\delta-1$, which is not a problem.
\end{rmk}

\subsubsection{Integration along any morphism}
\begin{dfn}
Let $f:Z'\to Z$ be an $\emptyset$-definable map and $\eta\in\{\chi,\mu_\Gamma\}$. Let $\Gamma_f$ be its graph, and let $p:\Gamma_f\to Z', \pi:\Gamma_f\to Z$ be the projections. We set
\[f_! = \pi_!\circ p^{-1}: C_+^\chi(Z')\to C_+^\chi(Z)\,.\]
\end{dfn}
This is well defined by Lemmas \ref{lem_proj_bij_sans} and \ref{lem_proj_bij_vol}.

\begin{prop}\label{prop_formalisme_integration}
Let $\begin{tikzcd} Z_0\arrow[r, "f"] & Z_1 \arrow[r,"g"] & Z_2\end{tikzcd}$ be $\emptyset$-definable maps and $\eta\in\{\chi,\mu_\Gamma\}$. Then $g_!\circ f_! = (g\circ f)_!$.

Moreover, if $f$ is a projection along some coordinates, then this definition is coherent with the above.
\end{prop}
\begin{proof}
The first fact is a straightforward consequence of the result of Section \ref{sssec_proj}. For the second one, in the case without volume forms it follows from Lemma \ref{lem_proj_bij_sans}. With volume forms, it is enough to see that, if $f$ is a coordinate projection and $\Gamma_f$ its graph, $p:\Gamma_f\to Z_0$ the projection map, then by definition of the canonical measure $|\omega_0|_{Z_0} = p^*|\omega_0|_{\Gamma_f}$. Now apply Lemma \ref{lem_proj_bij_vol}.
\end{proof}

\begin{cor}\label{cor_projection}
Let $f:Z'\to Z$ be an $\emptyset$-definable map, and let $\varphi\in\sC_+^\eta(Z),\psi\in C_+^\eta(Z)$. Then we have the projection formula
\[f_!(f^*\varphi\cdot\psi) = \varphi f_!\psi\,.\]
\end{cor}
\begin{proof}
Clear from construction.
\end{proof}

\begin{rmk}[Adding parameters]
Let $A\subset M\models\cT'$ be a set of parameters. We may do all the previous construction for $A$-definable objects instead of $\emptyset$-definable ones. By invariance of Hrushovski-Kazhdan motivic integration under adding parameters, the two constructions are compatible.

More precisely, denote by $\RV_Z^A[*],\mu_\Gamma\RV_Z^A[*], \dots$ the categories analogous to the above, but with parameters in $A$.We have functors $\RV_Z[*]\to\RV_Z^A[*], \dots$. Those functors induce morphisms between the Grothendieck semigroups, and these descend to the quotient down to constructible Functions, whence morphisms
\[C_+^\eta(Z)\to C_{+,A}^\eta(Z)\]
for every $\emptyset$-definable set $Z$, where $C_{+,A}^\eta(Z)$ is the semigroup of constructible motivic Functions over $Z$ with parameters in $A$ (and $\eta\in\{\chi,\mu_\Gamma\}$).

Then the above morphisms yield commutative diagrams
\[\begin{tikzcd}
C_+^\eta(Z') \arrow[r] \arrow[d, "f_!"] & C_{+,A}^\eta(Z') \arrow[d, "f_!"] \\
C_+^\eta(Z) \arrow[r] & C_{+,A}^\eta(Z')
\end{tikzcd}\,.\]
\end{rmk}

\subsection{Bounded objects}
In \cite{cluckers_constructible_2008}, Cluckers and Loeser consider a much smaller class of constructible motivic Functions. One of the main interests of this class is that its groupification is nonzero. We will take this restriction in our case as well, which will enable us to find back Cluckers-Loeser motivic integration.

We start by introducing a notion of bounded objects on $\RV$:
\begin{dfn}\label{dfn_borné}
Let $Z$ be an $\emptyset$-definable set and $d,\delta\in\bN$ natural numbers. An object $(R,\omega)$ of $\mu_\Gamma\RV^\delta_Z[d]$ is (relatively) \emph{bounded} if there is $Z_0\subset\Supp R$ such that $\dim_\VF \Supp(R)\setminus Z_0 < \delta$ and that, for all $z\in Z_0$:
\begin{enumerate}
\item the image of $|\omega_z|$ in $\Gamma$ is bounded from above, \emph{i.e.} there exists $\gamma_z\in\Gamma$ such that for all $r\in R_z$, we have $|\omega(r)|\cdot|r|<\gamma_z$;
\item for all $\gamma\in\Gamma$, the set $R_{z,\gamma}\coloneqq\{r\in R_z: |r|\cdot |\omega(r)|=\gamma\}$ is bounded in $\RV$, \emph{i.e.} if $R_z\subset\RV^n$, there exists $\delta_{z,\gamma}\in\Gamma$ such that, for all $r\in R_{z,\gamma}$, we have $|r|\in [\delta_{z,\gamma}^{-1},\delta_{z,\gamma}]^n$.
\end{enumerate}
We denote by $\mu_\Gamma\RV^{\delta,\bdd}_Z[d]$ the full subcategory of bounded objects, as well as $\mu_\Gamma\RV^{\delta,\bdd}_Z[*], \mu_\Gamma\RV^\bdd_Z[*]$ the associated direct sums.

We denote again by $\Isp^{Z,\mu}$ the restriction of this congruence to the Grothendieck semigroup of the category of (relatively) bounded objects.
\end{dfn}

Analogously, we can also define the rigid versions $\mu_\Gamma\RV_Z^{\bdd, \mathrm{rig}}[d], \mu_\Gamma\RV_Z^{\bdd,\mathrm{rig}}[*]$ of these categories. Now we introduce the semigroups of bounded constructible motivic Functions:

\begin{dfn}
Let $Z$ be an $\emptyset$-definable set. The semigroup of bounded constructible motivic Functions over $Z$ (with volume forms) is $C_+^\bdd(Z) = \Kp\mu_\Gamma\RV_Z^\bdd[*]/\Isp^{Z,\mu}$.

The semiring of bounded constructible motivic functions over $Z$ (with volume forms) is $\sC_+^\bdd(Z) = \Kp\mu_\Gamma\RV_Z^{\bdd,\mathrm{rig}}[*]/\Isp^{Z,\mu,\mathrm{rig}}$.
\end{dfn}
The semigroup $C_+^\bdd(Z)$ is naturally a (bi)graded subsemigroup of $C_+^{\mu_\Gamma}(Z)$, with the additional property:
\[\forall\varphi,\psi\in C_+^{\mu_\Gamma}(Z), \varphi+\psi\in C_+^\bdd(Z)\implies \varphi,\psi\in C_+^\bdd(Z)\,.\]
Similarly, $\sC_+^\bdd(Z)$ is a (graded) subsemiring of $\sC_+^{\mu_\Gamma}(Z)$, with the same property.

\begin{lem}
Let $f:Z'\to Z$ be an $\emptyset$-definable map, and let $\varphi\in C_+^{\mu_\Gamma}(Z')$ be such that $f_!\varphi\in C_+^{\mu_\Gamma}(Z)$. Then $\varphi\in C_+^\bdd(Z')$.
\end{lem}
\begin{proof}
If $f$ is bijective, it follows from Lemma \ref{lem_proj_bij_vol} that $f_!$ induces an isomorphism between the semigroups of bounded Functions. Using the construction of $f_!$, we are reduced to the case where $f$ is a projection. By induction, it is enough to deal with the case where $f$ is a projection along $\RV$ or along one $\VF$ coordinate.

If it is a projection along $\RV$, the result is clear by definition. So let us assume that $f$ is a projection along one $\VF$ coordinate, and let $\varphi\in C_+^{\mu_\Gamma}(Z')$ be unbounded. We want to show that $f_!\varphi$ is unbounded as well. To do so, assume that there is $\delta\in\bN$ such that $\varphi\in C_+^{\mu_\Gamma,\delta}(Z')$, and let $(X,\omega)\in\Ob\mu_\Gamma\RV_{Z'}^\delta[*]$ be a representative of $\varphi$. Taking a restriction if necessary, we may assume that $Z'=\Supp X$ and that the projection map $Z'\to Z$ has equidimensional fibres, with common dimension $\varepsilon\in\{0,1\}$. If $\varepsilon=0$, then there is a factorisation $Z'\to Z''\to Z$ of $f$ such that $Z''\to Z$ is an $\RV$-projection and that $Z'\to Z''$ is bijective; we have already treated this case. Assume then that $\varepsilon=1$. We write $X\subset Z'\times \RV^d\times Y \times \RV^m$, with $Y$ mapping into $\RV$, the projection $X\to Z'\times\RV^d\times Y$ having finite fibres and the projection map from the graph of $\omega$ towards $Y$ being an $\RV$-partition. For $y\in Y, z\in Z$, we take $C_{y,z}\subset\VF$ a finite $\cL(y,z)$-definable set preparing $\{(x,\omega(x)):x\in X_{y,z}\}$, where $X_{y,z}$ is the fibre of the projection $X\to Y\times Z$ at $(y,z)$. By compactness, we may assume that $C\subset\VF\times Y\times Z$ is $\emptyset$-definable. Then for all $z\in Z$ the projection $C_z\subset\VF\times Y\to \VF$ has an image $\pi(C_z)$ of dimension $0$; in particular, there exists $B_z\subset Z'_z\setminus \pi(C_z)$ a nonempty open ball. Over $B_z$, we have $X_{|B_z} = B_z\times X'_z$ and $\omega_{|X_{|B_z}}$ only depends on the variable along $X'_z$. Let us add some parameters to make $z$ and $B_z$ definable, and write $f_z:Z'_z\to \{z\}$ for the projection map. Then $f_{z!}[X_{|B_z},\omega] = [X'_z]\int[B_z] = [X'_z][r(B_z)]_1,$ where $r(B_z)\in\Gamma$ is the radius of the ball $B_z$ (seen as an open ball). Since $(X'_z,\omega)$ is unbounded by hypothesis (for some good choice of $z$), we deduce that $f_{z!}[X_{|B_z},\omega]$ neither. Hence, if we write $\varphi_z = [X_z,\omega]$, then $f_{z!}\varphi_z$ is unbounded, because $\varphi_z = [X_{|B_z},\omega]+\psi_z$ for some $\psi_z$. But $f_{z!}\varphi_z = (f_!\varphi)_z$, so $f_!\varphi$ is unbounded.
\end{proof}

This legitimates the following definition:

\begin{dfn}
Let $h:Z\to S$ be an $\emptyset$-definable map. The graded semigroup of constructible motivic Functions over $Z$ integrable relative to $S$ is
\[\I_S C_+(Z) = h_!^{-1}(C_+^\bdd(S))\subset C_+^\bdd(Z)\,.\]
\end{dfn}
We verify that this is indeed a graded semigroup. Let $\varphi = \sum_{\delta=0}^n \varphi_\delta \in \I_S C_+(Z)$ with $\varphi_\delta\in C_+^{\mu_\Gamma,\delta}(Z)$. Set some $\delta$, and write $\varphi = \varphi_\delta+\psi$. Then we have $h_!\varphi = h_!\varphi_\delta+h_!\psi$ which is bounded by assumption. So $h_!\varphi_\delta$ is indeed bounded.

\section{Connection with Cluckers-Loeser motivic integration}\label{sec_integration_CL}
We investigate more specifically the case of a discretely valued field, starting by an analogy with $\bQ_p$. Cluckers and Haskell \cite{cluckers_grothendieck_2001} showed that the Grothendieck group of definable subsets of $\bQ_p$ is trivial, and Hrushovski and Kazhdan \cite{ginzburg_integration_2006} interpreted this result in terms of their integration in their article. Later, Cluckers and Halupczok \cite{cluckers_p-adic_2021} showed that $p$-adic integration is universal, meaning that it induces an isomorphism between the Grothendieck group of definable subsets of $\bQ_p$ with volume forms and a subgroup of $\bQ$. They additionally interpreted this result as meaning that $p$-adic integration is an avatar of Hrushovski-Kazhdan motivic integration over $\bQ_p$. The language considered by Cluckers and Haskell, then Cluckers and Halupczok, is that of Denef-Pas, with an angular component. It is in particular the same as the one used by Cluckers and Loeser \cite{cluckers_constructible_2008} for their motivic integration. However, in the case of $\bQ_p$, finiteness of the residue field allows for the construction of a definable section $\Kb\to\VF$ -- which is not a field morphism, but this does not matter.

Following the same intuition, we will show that the Grothendieck group of definable subsets of a discretely valued field is zero, and that Cluckers-Loeser motivic integration is universal in the correct language. However, this "good language" is not that of Denef-Pas: in the same way that Cluckers and Halupczok used the section of the residue field that was naturally present, we will have to add a section of $\RV$. Note that, by virtue of $\RV$-minimality, this does change the categories over $\VF$ but not over $\RV$. We may see this as a way of precisely quantifying the defect of injectivity of Cluckers-Loeser motivic integration, see also upcoming work by Stout and Vermeulen.

Here $\cL$ is an expansion of $\cL_{\mathrm{DP}}$ the language of Denef-Pas by a possible enrichment of the structure on the residue field $\Kb$. We assume that $\cT$ is a theory of henselian valued fields with zero residue characteristic with $\ac$ an angular component on $\RV$. Finally, we assume that $\cT'$ is an expansion of $\cT$ obtained by adding a section $\sn:\RV^\times\to\VF^\times$, as described in Section \ref{sec_henselien}, and possibly adding parameters.

\subsection{Decomposition of \texorpdfstring{$\RV$}{RV}}
Due to the presence of an angular component, the $\RV$-categories decompose into categories along $\Kb$ and $\Gamma$. We start by writing these decompositions explicitly, first in the absolute case (Lemma \ref{lem_RV_dec_orth}) and then in the relative case (Lemma \ref{lem_RV_dec_orth_rel}), without forgetting to analyse what happens for bounded objects in Lemma \ref{lem_RV_dec_orth_bdd}.

\subsubsection{Absolute case}

\begin{dfn}
Let $\Def_{\Kb}[d]$ be the category whose objects are the $\emptyset$-definable subsets of $(\Kb^\times)^d\times\Kb^m$, and whose morphisms are the $\emptyset$-definable maps.

Let $\Def_\Gamma$ be the category whose objects are $\emptyset$-definable subsets of $\Gamma^m, m\in\bN$, and whose morphisms are $\emptyset$-definable maps.

Let $\mu\Def_\Gamma[d]$ be the category whose objects are pairs $(R,\omega)$, with $R\subset\Gamma^d\times\Gamma^m, m\in\bN$, and $\omega:R\to\Gamma$ being $\emptyset$-definable. The morphisms are $\emptyset$-definable bijections $h:(R,\omega)\to (S,\rho)$ such that, for all $r\in R$, we have $|r|\cdot \omega(r) = |h(r)|\cdot \rho(h(r))$, where $|r|=\prod_{i=1}^d r_i$.
\end{dfn}
\begin{rmk}
We assumed $\emptyset$-definability everywhere in the above. Note that, since all the sets of interest are in $\RV$ and $\cT'$ is $\RV$-minimal, assuming $\dcl_{\cL'}(\emptyset)\cap\VF = \dcl_\cL(\emptyset)\cap\VF$, these definitions are equivalent to the analogous definitions assuming $\cL(\emptyset)$-definability.
\end{rmk}

Note that there are inclusions $\mu\Def_\Gamma[d]\subset\mu\Def_\Gamma[d+1]$, given by $(R,\omega)\mapsto (\{1\}\times R,\omega)$. Those inclusions are equivalences of categories: full faithfulness is clear, and note that we might always add a $1$ at the beginning (shifting the rest to the right), up to changing $\omega$.

In particular, we can denote by $\Kp\mu\Def_\Gamma$ the associated Grothendieck semigroup, which is a semiring (due to case $d=0$).

In the same fashion, the inclusions $\Def_{\Kb}[d]\subset\Def_{\Kb}[d+1]$ are equivalences of categories.

\begin{lem}\label{lem_RV_dec_orth}
There are decompositions
\[\Kp\RV_{\cT'}[d] = \Kp\Def_{\Kb}[d]\otimes_\bN\Kp\Def_\Gamma\]
and
\[\Kp\mu_\Gamma\RV_{\cT'}[d] = \Kp\Def_{\Kb}[d]\otimes_\bN\Kp\mu\Def_\Gamma\,,\]
for all $d\in\bN\cup\{*\}$.
\end{lem}
\begin{proof}
Let us specify the morphism from right to left: when there is no volume form, it is $[X]\otimes[Y]\mapsto [X\times Y]$; when there are volume forms, it is the same morphism at the level of sets, and for the volume form we compose the one already present over $Y$ with the second projection.

Surjectivity then follows from orthogonality between $\Kb$ and $\Gamma$. On the other hand, injectivity also comes from orthogonality between $\Kb$ and $\Gamma$, applied to the graph of a bijection between two objects of $\RV_{\cT'}[d]$.
\end{proof}

We obtain in particular the following fact, to be compared with the results of Cluckers-Haskell \cite{cluckers_grothendieck_2001}:

\begin{prop}
Assume that the group $\Gamma$ is discrete. For all $d\in\bN$, the following Grothendieck groups cancel:
\[\mathrm K\VF_{\cT'}[d] = \mathrm K\RV_{\cT'}[d] = \mathrm K\Def_\Gamma=0\,.\]
\end{prop}
\begin{proof}
Let $\gamma_0 = \max \Gamma_{<1}$. There is a bijection between $\Gamma_{\leq 1}$ and $\Gamma_{< 1}$, given by multiplication by $\gamma_0$. Since $\Gamma_{\leq 1} = \Gamma_{<1}\cup\{1\}$, we find that $[1]=0$ in $\mathrm K\Def_\Gamma$, thus $\mathrm K\Def_\Gamma=0$. The previous lemma implies then that $\mathrm K\RV_{\cT'}[d]=0$, and Theorem \ref{thm_int_sans} yields $\mathrm K\VF_{\cT'}[d]=0$.
\end{proof}

\subsubsection{Relative case}
Let $Z$ be an $\emptyset$-definable set. We have the following relative variants of the categories on $\Kb$:

\begin{dfn}[Categories on $\Kb$]
Let $d,\delta\in\bN$.

Let $\Def_{\Kb,Z}^{\mathrm{rig}}[d]$ be the category whose objects are $\emptyset$-definable subsets of $Z\times(\Kb^\times)^d\times\Kb^m$, and whose morphisms are $\emptyset$-definable bijections commuting to the projections towards $Z$.

Let $\Def_{\Kb,Z}^\delta[d]$ be the category whose objects are the $R\in\Ob\Def_{\Kb,Z}^{\mathrm{rig}}[d]$ such that $\dim_\VF \Supp R \leq\delta$. A morphism $R\to S$ is a partial $\emptyset$-definable map $f:\pi^{-1}(Z_0)\subset R\to S$, where $\pi:R\to Z$ is the structure morphism, such that:
\begin{enumerate}
\item $f$ commutes to the projections to $Z$;
\item $Z_0\subset\Supp R\cap\Supp S$ and $\dim_\VF (\Supp(R\cup S)\setminus Z_0)<\delta$;
\item for all $z\in Z_0$, $f_z:R_z\to S_z$ is a bijection.
\end{enumerate}
We identify two morphisms $f,g:R\to S$ if there exists $Z_0\subset\Supp R$ an $\emptyset$-definable subset such that $\dim_\VF(\Supp(R)\setminus Z_0) < \delta$ and that $f_{|\pi^{-1}(Z_0)} = g_{|\pi^{-1}(Z_0)}$.

Then we define as usual the categories $\Def_{\Kb,Z}^{\mathrm{rig}}[*],\Def_{\Kb,Z}^\delta[*], \Def_{\Kb,Z}[*]$.
\end{dfn}

Similarly, we have relative variants of the categories on $\Gamma$. Since we already saw that, when $\Gamma$ is discrete, the group without volume forms is trivial, we now focus on the case with volume forms.

\begin{dfn}[Categories on $\Gamma$]
Let $d,\delta\in\bN$.

Let $\mu\Def_{\Gamma,Z}^{\mathrm{rig}}[d]$ be the category whose objects are pairs $(R,\omega)$, where $R\subset Z\times\Gamma^d\times\Gamma^m$ and $\omega:R\to\Gamma$ are $\emptyset$-definable. Morphisms are the $\emptyset$-definable bijections $f:(R,\omega)\to (S,\rho)$ commuting to projections towards $Z$ and such that, for all $r\in R$,
\[|r|\omega(r) = |f(r)|\rho(f(r))\,,\]
where, if $r=(z,\gamma,\gamma')\in Z\times\Gamma^d\times\Gamma^m$, we have $|r|=\prod_{i=1}^d\gamma_i$.

Let $\mu\Def_{\Gamma,Z}^\delta[d]$ be the category whose objects are the $(R,\omega)\in\Ob\mu\Def_{\Gamma,Z}^{\mathrm{rig}}[d]$ such that $\dim_\VF\Supp(R)\leq\delta$. A morphism $(R,\omega)\to (S,\rho)$ is a partial $\emptyset$-definable map $f:\pi^{-1}(Z_0)\subset R\to S$, where $\pi:R\to Z$ is the structure morphism, such that:
\begin{enumerate}
\item $f$ commutes to the projections to $Z$;
\item $Z_0\subset \Supp R\cap \Supp S$ et $\dim_\VF(\Supp(R\cup S)\setminus Z_0)<\delta$;
\item for all $z\in Z_0$, $f_z:(R_z,\omega_z)\to (S_z,\rho_z)$ is a morphism of $\mu\Def_{\Gamma,\cT'(z)}[d]$.
\end{enumerate}
\end{dfn}

\begin{rmk}\label{rmk_cat_rel}
If we rewrite the definition of relative categories on $\RV$ to replace the factor "$Y$ mapping into $\RV$" by $\RV$, as we did in Section, then these categories are actually by definition the full subcategories of the corresponding categories in $\RV$ whose objects are either entirely in $\Kb$ or entirely in $\Gamma$, with a trivial measure in the first case.
\end{rmk}

Again, there are inclusions $\mu\Def_ {\Gamma,Z}^\delta[d]\subset\mu\Def_{\Gamma,Z}^\delta[d+1]$ and $\mu\Def_{\Gamma,Z}^{\mathrm{rig}}[d]\subset\mu\Def_{\Gamma,Z}^{\mathrm{rig}}[d+1]$, given by $(R,\omega)\mapsto (\{1\}\times R,\omega)$. These inclusions are once more equivalences of categories.

In particular, we can write $\Kp\mu\Def_{\Gamma,Z}^\delta,\Kp\mu\Def_{\Gamma,Z}^{\mathrm{rig}}$ for the corresponding Grothendieck semigroups. In the rigid case or when $\delta=\dim_\VF Z$, they are semirings.

We define one last object.
\begin{dfn}
Let $\cP^{00}_+(Z)$ be the semiring of maps $Z\to\bN$ generated by the maps $\mathbf 1_{Z'}, Z'\subset Z$ $\emptyset$-definable.
\end{dfn}
In other words, $\cP^{00}_+(Z)$ is the semi-ring of $\emptyset$-definable maps over $Z$ with (bounded) integer values. This semiring acts by product on all Grothendieck semigroups relative to $Z$, rigid or not. Then we have:

\begin{lem}\label{lem_RV_dec_orth_rel}
There are decompositions
\[\Kp\mu_\Gamma\RV_Z^\delta[d] = \Kp\Def_{\Kb,Z}^\delta[d]\otimes_{\cP^{00}_+(Z)}\Kp\mu\Def_{\Gamma,Z}^\delta\]
and
\[\Kp\mu_\Gamma\RV_Z^{\mathrm{rig}}[d] = \Kp\Def_{\Kb,Z}^{\mathrm{rig}}[d]\otimes_{\cP^{00}_+(Z)}\Kp\mu\Def_{\Gamma,Z}^{\mathrm{rig}}\,,\]
for every $d\in\bN\cup\{*\},\delta\in\bN$.
\end{lem}
\begin{proof}
Same as Lemma \ref{lem_RV_dec_orth}.
\end{proof}

\subsubsection{Bounded case}
Again, let $Z$ be an $\emptyset$-definable set. Using Remark \ref{rmk_cat_rel}, one defines immediately the subcategories of bounded objects of $\Def_{\Kb}[d],\Def_{\Kb,Z}^{\mathrm{rig}}[d],\Def_{\Kb,Z}^\delta[d],\mu\Def_{\Gamma}[d], \mu\Def_{\Gamma,Z}^{\mathrm{rig}}[d], \mu\Def_{\Gamma,Z}^\delta[d]$ (for $d\in\bN\cup\{*\}, \delta\in\bN$) as the intersections of two full subcategories. In the case of categories on $\Kb$, we see that by definition all objects are bounded, because the condition is on $\Gamma$. For the categories on $\Gamma$, write $\mu\Def_{\Gamma}^\bdd[d], \mu\Def_{\Gamma,Z}^{\bdd,\mathrm{rig}}[d],\mu\Def_{\Gamma,Z}^{\bdd,\delta}[d]$ for the respective bounded variants of $\mu\Def_{\Gamma}[d],\mu\Def_{\Gamma,Z}^{\mathrm{rig}}[d], \mu\Def_{\Gamma,Z}^\delta[d]$. We still have equivalences of categories $\mu\Def_{\Gamma}^\bdd[d]\subset\mu\Def_{\Gamma}^\bdd[d+1]$ (and similarly in the relative case), hence we denote by $\Kp\mu\Def_\Gamma^\bdd$ (resp...) the common Grothendieck semigroup of these categories.

\begin{lem}\label{lem_RV_dec_orth_bdd}
We have decompositions of the Grothendieck semigroups of the categories of bounded objects
\[\Kp\mu_\Gamma\RV_Z^{\bdd,\delta}[d] = \Kp\Def_{\Kb,Z}^{\bdd,\delta}[d]\otimes_{\cP^{00}_+(Z)}\Kp\mu\Def_{\Gamma,Z}^{\bdd,\delta}\]
and
\[\Kp\mu_\Gamma\RV_Z^{\bdd,\mathrm{rig}}[d] = \Kp\Def_{\Kb,Z}^{\mathrm{rig}}[d]\otimes_{\cP^{00}_+(Z)}\Kp\mu\Def_{\Gamma,Z}^{\bdd,\mathrm{rig}}\,,\]
for all $d\in\bN\cup\{*\},\delta\in\bN$.
\end{lem}
\begin{proof}
This decomposition is induced by that of Lemma \ref{lem_RV_dec_orth_rel}.
\end{proof}

\begin{rmk}
In fact, in the first tensor product, there is no need for both factors to be taken "up to a set of small dimension", it is enough for one of them to carry the indeterminacy. Similarly, asking for one of the two factors to satisfy the condition that the support has dimension at most $\delta$ suffices. So actually, for all $\delta\in\bN,d\in\bN\cup\{*\}$, we have:
\[\Kp\mu_\Gamma\RV_Z^{\bdd,\delta}[d] = \Kp\Def_{\Kb,Z}^{\bdd,\delta}[d]\otimes_{\cP^{00,\delta}_+(Z)}\Kp\mu\Def_{\Gamma,Z}^{\bdd,\mathrm{rig}}\,.\]
The advantage of this new description is that, since the factor at the right does not depend anymore on $d$ or $\delta$, we can just take direct sum and find:
\[\Kp\mu_\Gamma\RV_Z^\bdd[*] = \Kp\Def_{\Kb,Z}^\bdd[*]\otimes_{\cP^{00}_+(Z)}\Kp\mu\Def_{\Gamma,Z}^{\bdd,\mathrm{rig}}\,.\]
\end{rmk}

\subsection{Presburger sets}
The point here is to give a description of $\Kp\mu\Def_\Gamma^\bdd$ (and its relative rigid variant) closer to what Cluckers and Loeser do, see Lemma \ref{lem_def_mu}. One may see this as a kind of integration on discrete ordered abelian groups: given a definable set with a volume form, we want to associate to it a universal additive invariant with a simpler description than the Grothendieck semigroup.

It is much more convenient in what follows to denote the value group additively. So, while $\Gamma$ will still be the value group embedded in $\RV$ (thus denoted multiplicatively), we will also write $\tGamma$ for the same group seen additively. By the usual conventions, we also reverse the order between $\Gamma$ and $\tGamma$. If $\gamma\in\Gamma$, we denote by $\tgamma$ the corresponding element of $\tGamma$. We also denote by $\tun=\min \tGamma_{>0}$, as well as $\gamma_0\in\Gamma$ the corresponding element.

We change the inclusion we defined above between the categories $\mu\Def_\Gamma[d]$, and rather take $(R,\omega)\mapsto (\{1\}\times R,\omega\cdot \gamma_0^{-1})$ (similarly in the relative case). Clearly this is still an equivalence of categories, and Lemma \ref{lem_RV_dec_orth} still holds.

\begin{dfn}
The \emph{reduced symmetric algebra} on $\tGamma$ is the $\bQ$-algebra $\widehat{\cS}\,\tGamma$ freely generated by the $\tgamma\in\tGamma$, with relations the obvious compatibility with addition and $\tun = 1$.

A map $f:Z\to \widehat{\cS}\,\tGamma$, where $Z$ is an definable set, is \emph{definable} if there exists a partition $Z = \bigsqcup_{i=1}^n Z_i$ of $Z$ into definable subsets, as well as polynomials $f_i\in\widehat{\cS}\,\tGamma[y], i=1,\dots, n$, such that for all $z\in Z_i$ we have $f(z)=f_i(z)$.
\end{dfn}

By \cite{cluckers_definable_2018} Proposition 5.2.1 \& Theorem 5.2.2, we have the following fundamental result:

\begin{thm}
Let $Z$ be a definable set and $(X_z)_{z\in Z}$ a definable family of bounded subsets $\tGamma$ (in the sense that, for all $z\in Z$, there is $\tgamma_z\in Z$ such that $X_z\subset [-\tgamma,\tgamma]^n$, if $X\subset Z\times\tGamma^n$). Then the map $z\in Z\mapsto \# X_z$, which to $z$ associates the hypercardinal of the bounded set $X_z$, is definable.

If $(X'_z)_{z\in Z}$ is another definable family of bounded subsets of $\tGamma$, then there is a definable bijection relative to $Z$ between $(X_z)_z$ and $(X'_z)_z$ if and only if, for all $z\in Z$, we have $\# X_z = \#X'_z$. Moreover, if this is the case and both families are $\emptyset$-definable, then the bijection itself can be taken to be $\emptyset$-definable.
\end{thm}

In particular:
\begin{cor}
Let $Z$ be an $\emptyset$-definable set. To any object $(R,\omega)\in\Kp\mu\Def_{\Gamma,Z}^{\bdd,\mathrm{rig}}$, we associate the definable function
\[(\tgamma,z)\in\tGamma\times Z\mapsto \#R_{\tgamma,z}\in\widehat{\cS}\,\tGamma\,,\]
where $R_{\tgamma,z} = \{r\in R_z:|r|\omega(z,r) = \gamma\}$.

This mapping is injective at the level of the Grothendieck semigroup.
\end{cor}

Actually, the image of this morphism enjoys a very nice description, which we now give. This is very similar to \cite{cluckers_constructible_2008}, 4.2, from where we borrowed notations.

\begin{dfn}
Let $Z$ be an $\emptyset$-definable set. Let $\cP_+(Z)$ be the subsemiring of maps $Z\to\widehat{\cS}\,\tGamma [(\bbL^{\tgamma})_{\tgamma\in\tGamma},(\frac{1}{1-\bbL^{-i}})_{i\in\bN^*}]$ generated by:
\begin{enumerate}
\item the constant maps equal to $\frac{1}{1-\bbL^{-i}}, i\in\bN^*$;
\item the maps $\bbL^{\widetilde\alpha}$ for every $\emptyset$-definable map $\widetilde\alpha:Z\to\tGamma$;
\item the maps $\widetilde\alpha:z\mapsto \widetilde\alpha(z)$ for every $\emptyset$-definable map $\widetilde\alpha:Z\to\tGamma_{\geq 0}$.
\end{enumerate}

Let $\cP^0_+(Z)$ be the subsemiring of maps $Z\to\widehat{\cS}\,\tGamma[\bbL]$ generated by:
\begin{enumerate}
\item the constant map equal to $\bbL-1$;
\item the maps $\mathbf 1_{Z'}$ for every $\emptyset$-definable subset $Z'\subset Z$.
\end{enumerate}
\end{dfn}

\begin{rmk}
We may see $\cP_+(Z)$ (and $\cP^0_+(Z)$) as a subsemiring of definable functions $Z\times\tGamma\to\widehat{\cS}\,\tGamma$, via the developments in (nonstandard) Laurent series in $\bbL^{-1}$, the image of $(z,\tgamma)$ being the coefficient of $\bbL^{-\tgamma}$.

Indeed, for the last two generators, this is clear. For $\frac{1}{1-\bbL^{-i}}$, the standard development in power series is $\sum_{n\equiv 0[i], n\geq 0}\bbL^{-n}$. But this clearly has a nonstandard analogue, as we wanted.
\end{rmk}

\begin{lem}\label{lem_def_mu}
For every $[R,\omega]\in\Kp\mu\Def_{\Gamma,Z}^{\bdd,\mathrm{rig}}$, the morphism
\[z\in Z\mapsto \sum_{\tgamma\in\Gamma}\#R_{z,\tgamma}\bbL^{-\tgamma}\in\widehat{\cS}\,\tGamma\llbracket\bbL^{-1}\rrbracket[\bbL]\]
induces a semiring isomorphism
\[\mu:\Kp\mu\Def_{\Gamma,Z}^{\bdd,\mathrm{rig}}\to\cP_+(Z)\,.\]
\end{lem}
\begin{proof}
We have already seen that the first morphism is injective. The fact that $\mu$ has its image inside $\cP_+(Z)$ is a consequence of the cell decomposition theorem for Presburger sets, proved by Cluckers \cite{cluckers_presburger_2003}. See also the construction of Cluckers-Loeser motivic integration \cite{cluckers_constructible_2008}.
\end{proof}
\begin{rmk}
The reason why we changed the previous inclusion is that the object $R=\{1\}$ endowed with $\omega=1$, in dimension $d$, is supposed to correspond to the ball $(1+\mathfrak m)^d$, with volume $q^{-d}$, where $q$ is the cardinal of the residue field. By making these choices, we do find $\mu([1]_d) = \bbL^{-d}$.
\end{rmk}

Actually, this isomorphism is itself a kind of integration. More precisely, let $(R,\omega)\in\Ob\mu\Def_{\Gamma,Z}^{\bdd,\mathrm{rig}}[0]$ be an object, say $R\subset Z\times\Gamma^m$. Let $\pi:Z\times\Gamma^m\to Z$ be the projection map, and let $f = \mathbf 1_Z \bbL^{-\omega}$ be the constructible motivic function (in the sense of \cite{cluckers_constructible_2008}) associated to $(R,\omega)$. Then $\mu([R,\omega]) = \pi_! f$, the morphism $\pi_!$ being taken in the sense of Cluckers-Loeser.

So we may see the isomorphism $\mu$ as an isomorphism analogous to Hrushovski-Kazhdan motivic integration, \emph{i.e.} an isomorphism between a more geometric side ($\emptyset$-definable subsets of $\Gamma$) and a side having a simpler algebraic description (the semiring $\cP_+(Z)$).

\subsection{Cluckers-Loeser constructible motivic Functions}
Now we go back to $\RV$ and constructible motivic functions (or Functions).Let $Z$ be an $\emptyset$-definable set. We will make use of the previous description to connect the objects we have used so far to those of Cluckers-Loeser -- or a least, to their closest analogue in $\cT'$. We will connect them to Cluckers-Loeser objects as well (for $\cT$) when $Z$ is $\cL(\emptyset)$-definable.

\begin{lem}\label{lem_RV_dec_1}
There is a semiring isomorphism
\[\sC_+^\bdd(Z) = (\Kp\Def_{\Kb,Z}[*]\otimes_{\cP_+^{00}(Z)}\cP_+(Z))/([1]_1\otimes 1, (\bbL-[1]_1)\otimes \frac{\bbL^{-1}}{1-\bbL^{-1}})\,.\]
\end{lem}
\begin{proof}
According to Lemma \ref{lem_RV_dec_orth_rel}, the only thing left to understand is the quotient by the congruence $\mu\Isp$. The congruence $\Isp^{Z,\mu,\mathrm{rig}}$ is generated by $([1]_1,[\RV_{<1}])$. The element $[1]_1$ on the left hand side corresponds to $[1]_1\otimes \bbL^{-1}$ on the right, while $[\RV_{<1}]_1$ on the left corresponds to $[\bG_m]_1\otimes\frac{\bbL^{-2}}{1-\bbL^{-1}}$ on the right. In other words, $\Isp^{Z,\mu,\mathrm{rig}}$ is generated by $([1]_1\otimes \bbL^{-1},(\bbL-[1]_1)\otimes\frac{\bbL^{-2}}{1-\bbL^{-1}})$. Let us multiply by the invertible element $1\otimes\bbL$; we rewrite the generator as $([1]_1\otimes 1, (\bbL-[1]_1)\otimes \frac{1}{\bbL^{-1}}{1-\bbL^{-1}}$, as we wanted.
\end{proof}

As one can see, this description is little satisfying. In order to get a more elegant one, we will do two different things:
\begin{itemize}
\item go to volume semigroups, \emph{i.e.} normalise the motivic measure so that $\int[\cO]_1=1$;
\item take the groupification for all objects and morphisms.
\end{itemize}

Note that, by construction, $\int[\cO]_1 = [\RV_{\leq 1}]_1 = [1]_1\otimes 1$. So for the equality $\int[\cO]_1=1$ to hold, the relation we have to ask for is $[1]_1=1$ in $\Kp\Def_{\Kb,Z}[*]$.

\begin{dfn}\label{def_ideal_volume}
Denote by $\tun$ the pair $(X,\widetilde\omega)=(\{*\},\tun)$ in $\Ob\RV_Z^{\mathrm{rig}}[0]$. Let $\Isp^{Z,\vol,\mathrm{rig}}$ be the congruence on $\Kp\mu_\Gamma\RV_{Z}^{\bdd,\mathrm{rig}}[*]$ generated by $([1]_1,[\RV_{<1}]_1)$ et $([1]_1,\tun)$. Also let $\sC_+^{\vol}(Z)$ be the quotient of $\Kp\mu_\Gamma\RV_Z^{\bdd,\mathrm{rig}}[*]$ by $\Isp^{Z,\vol,\mathrm{rig}}$.

Let $\Def_{\Kb}$ be the colimit category of the natural inclusions $\Def_{\Kb}[n]\subset \Def_{\Kb}[n+1]$.
\end{dfn}

The decomposition of $\RV$ now takes a nicer form:

\begin{lem}\label{lem_RV_dec_2}
Let $J$ be the congruence on $\Kp\Def_{\Kb}[*]$ generated by $([1]_1,1)$. There is a canonical isomorphism
\[\Kp\Def_{\Kb}[*]/J = \Kp\Def_{\Kb}\,.\]
In particular, there is an isomorphism
\[\sC_+^\vol(Z) = (\Kp\Def_{\Kb,Z}\otimes_{\cP^{00}_+(Z)}\cP_+(Z))/(1-(\bbL-1)\otimes\frac{\bbL^{-1}}{1-\bbL^{-1}})\,.\]
\end{lem}
\begin{proof}
The second point immediately follows from the first thanks to Lemma \ref{lem_RV_dec_1}. As for the first point, clearly there is a canonical morphism $\Kp\Def_{\Kb}[*]\to\Kp\Def_{\Kb}$, and $J$ lies in the equaliser, yielding the desired morphism. It is obviously surjective; let us prove that it is injective. To do so, let $a = \sum_{i=1}^n a_i, b=\sum_{i=1}^nb_i$ be two elements of $\Kp\Def_{\Kb}[*]$ having the same image under that morphism, with $a_i,b_i\in\Kp\Def_{\Kb}[i]$. Up to multiplying $a_i, b_i$ by $[1]_{n-i}$ (which is equal to $1$ in the quotient by $J$), we may assume that $a_i=b_i=0$ for all $i<n$. Then the result clearly follows from the fact that the inclusions $\Def_{\Kb}[m]\subset\Def_{\Kb}[m+1]$ are fully faithful.
\end{proof}

Now, the right hand side is almost a tensor product: more precisely, if we had $\bbL-1\in\cP_+(Z)$, then it would be the tensor product $\Kp\Def_{\Kb,Z}\otimes_{\cP^0_+(Z)}\cP_+(Z)$, where $\cP^0_+(Z) = \cP^{00}_+(Z)[\bbL-1]$ is the semiring generated by $\cP^{00}_+(Z)$ et $\bbL-1$. Of course, it certainly does not hold that $\bbL-1\in\cP_+(Z)$. To remedy this, we will groupify:

\begin{cor}
Let \[\sC^\vol(Z), \mathrm K\Def_{\Kb,Z},\cP(Z),\cP^0(Z)\] be the respective groupifications of \[\sC_+^\vol(Z), \Kp\Def_{\Kb,Z}, \cP_+(Z),\cP_+^0(Z)\,.\] There is an isomorphism
\[\sC^\vol(Z) = \mathrm K\Def_{\Kb,Z}\otimes_{\cP^0(Z)}\cP(Z)\,.\]
\end{cor}

For $Z$ an $\cL(\emptyset)$-definable set, denote by $\sC^{\cL,\vol}(Z)$ the ring of $\cL$-definable constructible motivic functions on $Z$, as defined in \cite{cluckers_hensel_2022}, Theorem 5.8.2, using the results of \cite{cluckers_motivic_2015} (and denoted by $\sC(Z)$ in that reference). By the above corollary, there is a canonical morphism
\[\sC^{\cL,\vol}(Z)\to \sC^\vol(Z)\,.\]
Note that, if $Z\subset\RV^m$, this is even an isomorphism by $\RV$-minimality.

For $\delta\in\bN$, we also denote by $\sC_+^{\vol,\leq\delta}(Z)$ the subsemigroup of $\sC_+^\vol(Z)$ of functions with support of dimension at most $\delta$, defined as the image of $\sC_+^{\bdd,\leq\delta}(Z)$.

Now we turn to constructible motivic Functions.

\begin{dfn}\label{def_Fonctions_CL}
Let $\delta\in\bN$. Define $C_+^{\vol,\delta}(Z)$ as the quotient of $\Kp\mu_\Gamma\RV_Z^{\bdd,\delta}[*]$ by the semigroup congruence $\Isp^{Z,\vol,\delta}$ generated via the semimodule structure by the homonymous semiring congruence on $\Kp\mu_\Gamma\RV_Z^{\bdd,\mathrm{rig}}[*]$. We also let $C_+^\vol(Z) = \bigoplus_{\delta\in\bN}C_+^{\vol,\delta}(Z)$, as well as $C^\vol(Z)=\bigoplus_{\delta\in\bN}C^{\vol,\delta}(Z)$ be its groupified.
\end{dfn}

There is an obvious morphism $C_+^{\bdd,\delta}(Z)\to C_+^{\vol,\delta}(Z)$. In addition:
\begin{lem}
There is an equality
\[C_+^{\vol,\delta}(Z) = \sC_+^{\vol,\leq\delta}(Z)/\sC_+^{\vol,\leq\delta-1}(Z)\,.\]
\end{lem}
\begin{proof}
This fact is known for $\bdd$, and we deduce the result by taking quotient.
\end{proof}
In particular, when $Z$ is $\cL(\emptyset)$-definable, if $C^{\cL,\vol}(Z)$ is the graded group of $\cL$-definable constructible motivic Functions over $Z$, then there is also a graded group morphism $C^{\cL,\vol}(Z)\to C^\vol(Z)$. Moreover, if $Z\subset\RV^m$, this is an isomorphism.

\subsection{Cluckers-Loeser motivic integration}
Now that we have found Cluckers-Loeser groups, we have to connect the $f_!$ we defined to the morphisms between these groups given by the motivic integration introduced in \cite{cluckers_constructible_2008}, and obtain a commutative square in Theorem \ref{thm_cluckers_loeser}.

Let $f:Z'\to Z$ be an $\emptyset$-definable map. Let $a\in \I_Z C_+^\bdd(Z')$. By the projection formula of Corollary \ref{cor_projection}, we have $f_!([1]_1a)=[1]_1f_!a$ and $f_!(\tun a)=\tun f_!a$. Hence $f_!$ descends to the quotient as a semigroup morphism $f_!:\I_Z C_+^\vol(Z')\to C_+^\vol(Z)$. We have shown the following:

\begin{prop}\label{prop_formalisme_CL}
For every $\emptyset$-definable map $Z\to S$, let $\I_S C_+^\vol(Z)$ be the image of $\I_S C_+^\bdd(Z)$ under the quotient morphism $C_+^\bdd(Z)\to C_+^\vol(Z)$.

For all $\begin{tikzcd} Z'\arrow[r, "f"] & Z \arrow[r] & S\end{tikzcd}$, there is a unique morphism $f_!:\I_S C_+^\vol(Z')\to \I_S C_+^\vol(Z)$ such that the diagram
\[\begin{tikzcd}
\I_S C_+^\bdd(Z') \arrow[r] \arrow[d, "f_!"] & \I_S C_+^\vol(Z') \arrow[d, "f_!"] \\
\I_S C_+^\bdd(Z) \arrow[r] & \I_S C_+^\vol(Z)
\end{tikzcd}\]
commutes.

Moreover, those morphisms satisfy Cluckers-Loeser axioms (\cite{cluckers_constructible_2008}, Theorem 10.1.1).
\end{prop}
\begin{proof}
Uniqueness is clear, since horizontal maps are surjective. We just proved the existence. That these morphisms satisfy Cluckers-Loeser axioms comes from the construction. For instance, the functoriality axiom (A0) has already been shown; axioms (A1) and (A2) are clear by construction; (A3) has be shown already; (A5), (A6), (A7), (A8) come from the constructions in the case of the $\RV$ or $\VF$-projections.
\end{proof}

Denote by $\I_S C^\vol(Z)$ the groupification of $\I_S C_+^\vol(Z)$. If $Z\to S$ is an $\cL(\emptyset)$-definable map, denote by $\I_S C^{\cL,\vol}(Z)$ the group of $\cL(\emptyset)$-definable constructible motivic F0unctions over $Z$ that are integrable relative to $S$.

\begin{thm}\label{thm_cluckers_loeser}
\begin{enumerate}
\item Let $Z\to S$ be an $\cL(\emptyset)$-definable map. The image of $\I_S C^{\cL,\vol}(Z)$ under the morphism $C^{\cL,\vol}(Z)\to C^\vol(Z)$ lies in $\I_S C^\vol(Z)$.
\item Let $\begin{tikzcd} Z' \arrow[r, "f"] & Z \arrow[r] & S \end{tikzcd}$ be $\cL(\emptyset)$-definable. The square
\[\begin{tikzcd}
\I_S C^{\cL,\vol}(Z') \arrow[r] \arrow[d, "f_!"] & \I_S C^\vol(Z') \arrow[d, "f!"]\\
\I_S C^{\cL, \vol}(Z) \arrow[r] & \I_S C^\vol(Z')
\end{tikzcd}\]
is commutative.
\end{enumerate}
\end{thm}
\begin{proof}
First we deal with the two points simultaneously by assuming that $Z=S$ in the second one. If $f:Z'\to Z$ is a projection map along $\Gamma$, $\Kb$ or $\VF$, or an injection, then the result holds by axioms (A6), (A5), (A7) and (A8), or (A4) respectively (as well as (A1)). The functoriality axiom (A0) now allows us to show the result for any $Z'\to Z\to S$.
\end{proof}

\subsection{Conclusion}
We have shown that Cluckers-Loeser motivic integration, over henselian discretely valued fields of residue characteristic zero and any base set $S\subset\VF^n\times\Gamma^m\times\Kb^l$, is the "universal" integration theory with:
\begin{enumerate}
\item the condition $\int[\cO]_1=1$;
\item a section $\sn:\RV^\times\to\VF^\times$.
\end{enumerate}
While the first point is obvious by definition and is a very reasonable decision by the authors, the second one is much more unexpected -- since there is no such section in Cluckers and Loeser work. However, we see that it appears naturally when we try to rephrase their work in terms of a universal integration theory.

It would also have been possible to take only a section of $\Gamma$ or $\Kb$, but in that case we would have to restrict the base set $S$ to $S\subset\VF^n\times\Gamma^m$ or $S\subset\VF^n\times\Kb^l$ respectively (and, in the case of $\Gamma$, to assume that $\Kb$ is algebraically bounded over $\emptyset$, so that the theory is effective).

\section{Appendix - Adding a section and dimension theory}\label{sec_appendice_dim}
In that work, we needed to develop a dimension theory when we have a dimension in a base language, and we want to expand that language and keep a dimension theory. Here we explain how to do so (see Proposition \ref{prop_appendice}).

Let $\cL$ be any (one-sorted) language and $\cT$ a theory. We assume that there is a notion of dimension on definable sets, as defined by van den Dries \cite{dries_dimension_1989}. In other words, if we let $\cS_n$ be the set of (parametrically) definable subsets of $M^n$, with $M\models \cT$, there is a map $\dim:\bigcup_{n\in\bN}\cS_n\to\bN\cup\{-\infty\}$ such that, for all $X,Y\in \cS_m,m\geq 0$, we have:

\begin{enumerate}
\item $\dim X=-\infty \iff X=\emptyset$, $\dim \{a\}=0$ for all $a\in M$, $\dim M^1=1$;
\item $\dim (X\cup Y)=\max (\dim X, \dim Y)$;
\item for every permutation $\sigma$ of $\{1,\dots,m\}$, if we let $X^\sigma$ be the image of $X$ under the action of $\sigma$ by permuting the coordinates, we have $\dim X^\sigma=\dim X$;
\item if $m\geq 1$, for $x\in M^{m-1}$, let $X_x=\{y\in M:(x,y)\in X\}$. Also set $X(0)=\{x\in M^{m-1}:\dim X_x=0\}$ and $X(1)=\{x\in M^{m-1}:\dim X_x=1\}$. Then $X(0)$ and $X(1)$ are definable, and moreover $\dim \{(x,y)\in X:x\in X(i)\}=\dim X(i)+1$ for $i=0,1$.
\end{enumerate}

We will allow us to generalise this beyond the case where $\cL$ is one-sorted. In that case, we assume that $\cL$ is many-sorted, with one main sort and auxiliary sorts, and that the dimension is taken in that main sort. In other words, the dimension of the auxiliary sorts is zero; or, if $X\subset M^n\times S$ is definable, with $M$ the main sort and $S$ auxiliary, then $\dim X=\max_{s\in S}\dim X_s$, where $X_s=\{x\in M^n:(x,s)\in X\}$.

Let $R$ be an imaginary sort of $\cT$, and let $\cL'$ be an expansion of the language $\cL$. Also let $\cT'$ be an $\cL'$-theory expanding $\cT$.

\begin{dfn}
Let $M\models\cT', A\subset M$ be a set of parameters and $X$ be an $\cL'(A)$-definable set. An \emph{$A$-definabel $R$-partition} of $X$ is the data of a subset $Y\subset M^l$ and a map $\chi:X\to Y$, both $A$-definable, such that:
\begin{enumerate}
\item for all $y\in Y$, the fibre $\chi^{-1}(y)$ is $\cL(y)$-definable;
\item there is an $\cL'(A)$-definable injective map $f:Y\to R^m$ for some $m\in\bN$.
\end{enumerate}
If such $R$-partitions exist for all $M,A$ and $X$, we say that $\cT'$ \emph{has $R$-partitions}. We always make this assumption in what follows.
\end{dfn}

In that context, we can introduce the following notion:
\begin{dfn}
For an $\cL'$-definable set (with parameters) $X$, and $\chi:X\to Y$ and $R$-partition of $X$, we define:
\[\dim (X,\chi) = \max_{y\in Y}\dim_\cT\chi^{-1}(y)\,.\]
\end{dfn}

If we want this notion to be well-behaved, it is clearly necessary to make some assumptions. For that purpose we will briefly need the notion of $R$-covering.
\begin{dfn}
For an $\cL$-definable set (with parameters) $X$, we call \emph{$R$-covering of $X$} an  $\cL'$-definable family $(Z_y)_{y\in Y}$, with $Z_y$ an $\cL(y)$-definable set and $Y$ an $\cL'$-definable set with parameters that maps into $R$, such that $X=\bigcup_{y\in Y} Z_y$.
\end{dfn}
In particular, any $R$-partition is an $R$-covering by its fibres.

Now we may state the following hypothesis:\\
\textbf{Hypothesis:} For $X$ an $\cL$-definable set with parameters, and $(Z_y)_{y\in Y}$ an $R$-covering of $X$, we have 
\begin{equation}\label{egalite_dim_sec}\dim X = \max \dim Z_y\,.
\end{equation}

This corresponds to the idea that $\dim R=0$. Note that the inequality $\dim X\geq\max\dim Z_y$ is trivial. A first immediate consequence is that, for any $\cL$-definable set (with parameters) $X$,for any $\chi$-partition of $X$, we have $\dim X=\dim (X,\chi)$.

Now we show that, by doing so, we got a good dimension theory for $\cL'$-definable sets.

\begin{prop}\label{prop_appendice}
Under this assumption, for every $\cL'$-definable set (with parameters) $X$ and every $\chi_i:X\to Y_i, i=1,2$ two $R$-partitions of $X$, we have $\dim (X,\chi_1)=\dim (X,\chi_2)$. So we define $\dim X$ unambiguously. Moreover, this notation is legitimate, in the sense that this dimension still satisfies the four van den Dries axioms.
\end{prop}
\begin{proof}
For the first point, since $(\chi_1,\chi_2)$ is again an $R$-partition of $X$, we may assume without loss that $\chi_2$ refines $\chi_1$. Then let $y\in Y_1$, and set $X_y=\chi_1^{-1}(y)$. By our hypothesis, we have $\dim X_y = \dim (X_y,(\chi_2)_{|X_y})$. So indeed $\dim (X,\chi_1)=\dim (X,\chi_2)$.

Now we show that $\dim$ satisfies the van den Dries axioms. The first one is clear, using the hypothesis for the last point. For the second axiom, we just take $R$-partitions of $X_1\cap X_2$, $X_1\setminus X_2$ and $X_2\setminus X_1$ and glue them together for the result to be clear. The third point presents no difficulty, so we are left with the fourth one. We use its notations. Let $\chi:X\to Y$ be an $R$-partition of $X$, with $Y\subset M^l$. Now there is a definable subset $Z\subset M^m\times M^l$ such that, for all $y\in Y$, the fibre $\chi^{-1}(y)$ identifies with $Z_y$. By van den Dries fourth axiom, applied to $Z$ relative to the last coordinate of $X$, we find $Z(0)$ et $Z(1)$ two definable sets such that $(x,y)\in Z(i)\iff\dim\{z\in M:(x,z,y)\in Z\}=i$ for $i=0,1$. But then, for $x\in M^{m-1}$, we have:

\[\begin{array}{rcl}
\dim X_x=1 &\iff& (\exists y\in Y)\dim X_x\cap\chi^{-1}(y)=1 \\
&\iff& (\exists y\in Y)\dim Z_{x,y}=1 \\
&\iff& (\exists y\in Y)(x,y)\in Z(1)
\end{array}\]
which is clearly a definable condition, and same for $X(0)$. Denote by $X'(i)=\{(x,z)\in X:x\in X(i)\},i=0,1$, and take some $i$. Let $\chi_i:X(i)\to Y_i$ be an $R$-partition of $X(i)$. Let $y_i\in Y_i$; as one can see by taking an $R$-partition of $X'(i)$ that refines $\chi_i$, it is enough to show that, if $X'=\{(x,z)\in X:x\in \chi_i^{-1}(y_i)\}$, then $\dim X'=\dim \chi_i^{-1}(y_i)+i$. From now on we assume $X=X'$ et $X(i)=\chi_i^{-1}(y_i)$. Let $y\in Y$, then the projection $\chi^{-1}(y)\to X(i)$ has fibres of dimension at most $i$ and image of dimension at most $\dim X(i)$, so $\dim\chi^{-1}(y)\leq\dim X(i)+i$. Consequently, $\dim X\leq\dim X(i)+i$. Conversely, the $R$-partition $\chi$ of $X$ induces a natural $R$-covering of $X(i)$, via the projection map $X\to X(i)$, which we denote by $(Z_y)_{y\in Y}$. By hypothesis, there is $y\in Y$ such that $\dim Z_y=\dim X(i)$.But the projection map $\chi^{-1}(y)\to Z_y$ has fibres of dimension $i$; since here everything is $\cL$-definable, with some parameters, we deduce that $\dim\chi^{-1}(y)=\dim X(i)+i$, thus $\dim X\geq\dim X(i)+i$.
\end{proof}
\begin{rmk}
Here we used a hypothesis on $R$-coverings, stronger \emph{a priori} than if we had stated it just for $R$-partitions, because coverings are preserved under more operations, and in particular under projections. However, in the article, except to prove that this hypothesis is satisfied, we never need such a notion; that is why we have only stated it in passing.
\end{rmk}

\printbibliography
\end{document}